\documentclass[10pt, reqno]{amsart}

\usepackage{fullpage}

\usepackage{textcmds} 
\usepackage{amsmath}
\usepackage{amsxtra}
\usepackage{amscd}
\usepackage{amsthm}
\usepackage{amsfonts}
\usepackage{amssymb}
\usepackage{eucal}
\usepackage[all]{xy}
\usepackage{graphicx}
\usepackage{comment}
\usepackage{epsfig}
\usepackage{psfrag}
\usepackage{mathrsfs}
\usepackage{amscd}
\usepackage{rotating}
\usepackage{lscape}
\usepackage{amsbsy}
\usepackage{verbatim}
\usepackage{moreverb}
\usepackage{url}
\usepackage{extarrows}
\usepackage{bbm}

\makeatletter
\@addtoreset{subsubsection}{section}
\makeatother 

\DeclareMathAlphabet{\mathpzc}{OT1}{pzc}{m}{it}

\newtheorem{cor}[subsubsection]{Corollary}
\newtheorem{lem}[subsubsection]{Lemma}
\newtheorem{prop}[subsubsection]{Proposition}

\newtheorem{conj}[subsubsection]{Conjecture}
\newtheorem{thm}[subsubsection]{Theorem}

\newtheorem{defn}[subsubsection]{Definition}

\newtheorem{notat}[subsubsection]{Notation}

\theoremstyle{remark}
\newtheorem{rem}[subsubsection]{Remark}
\newtheorem{example}[subsubsection]{Example}

\theoremstyle{remark}

\numberwithin{equation}{section}

\newcommand{\nc}{\newcommand}
\nc{\renc}{\renewcommand}
\nc{\ssec}{\subsection}
\nc{\sssec}{\subsubsection}
\nc{\on}{\operatorname}

\nc\ol{\overline}
\nc\wt{\widetilde}
\nc\tboxtimes{\wt{\boxtimes}}
\nc\tstar{\wt{\star}}
\nc{\alp}{\alpha}

\nc{\ZZ}{{\mathbb Z}}
\nc{\NN}{{\mathbb N}}
\nc{\OO}{{\mathbb O}}
\renc{\SS}{{\mathbb S}}
\nc{\DD}{{\mathbb D}}
\nc{\GG}{{\mathbb G}}
\renewcommand{\AA}{{\mathbb A}}
\nc{\Fq}{{\mathbb F}_q}
\nc{\Fqb}{\ol{{\mathbb F}_q}}
\nc{\Ql}{\ol{{\mathbb Q}_\ell}}
\nc{\id}{\text{id}}
\nc\X{\mathcal X}

\nc{\Hom}{\on{Hom}}
\nc{\Lie}{\on{Lie}}
\nc{\Loc}{\on{Loc}}
\nc{\Pic}{\on{Pic}}
\nc{\Bun}{\on{Bun}}
\nc{\IC}{\on{IC}}
\nc{\Aut}{\on{Aut}}
\nc{\rk}{\on{rk}}
\nc{\Sh}{\on{Sh}}
\nc{\pos}{{\on{pos}}}
\nc{\Conv}{\on{Conv}}
\nc{\Sph}{\mathcal{S}ph}
\nc{\Sym}{\on{Sym}}
\nc{\BunBb}{\overline{\Bun}_B}
\nc{\BunNb}{\overline{\Bun}_N}
\nc{\BunTb}{\overline{\Bun}_T}
\nc{\BunBbm}{\overline{\Bun}_{B^-}}
\nc{\BunBbel}{\overline{\Bun}_{B,el}}
\nc{\BunBbmel}{\overline{\Bun}_{B^-,el}}
\nc{\Buno}{\overset{o}{\Bun}}
\nc{\BunPb}{{\overline{\Bun}_P}}
\nc{\BunBM}{\Bun_{B(M)}}
\nc{\BunBMb}{\overline{\Bun}_{B(M)}}
\nc{\BunPbw}{{\widetilde{\Bun}_P}}
\nc{\BunBP}{\widetilde{\Bun}_{B,P}}
\nc{\GUb}{\overline{G/U}}
\nc{\GUPb}{\overline{G/U(P)}}

\nc{\Hhom}{\underline{\on{Hom}}}
\nc\syminfty{\on{Sym}^{\infty}}
\nc\lal{\ol{\lambda}}
\nc\xl{\ol{x}}
\nc\thl{\ol{\theta}}
\nc\nul{\ol{\nu}}
\nc\mul{\ol{\mu}}
\nc\Sum\Sigma
\nc{\oX}{\overset{\circ}{X}{}}
\nc{\hl}{\overset{\leftarrow}h{}}
\nc{\hr}{\overset{\rightarrow}h{}}
\nc{\M}{{\mathcal M}}
\nc{\N}{{\mathcal N}}
\nc{\F}{{\mathcal F}}
\nc{\D}{{\mathcal D}}
\nc{\Q}{{\mathcal Q}}
\nc{\Y}{{\mathcal Y}}
\nc{\G}{{\mathcal G}}
\nc{\E}{{\mathcal E}}
\nc{\CalC}{{\mathcal C}}
\nc\Dh{\widehat{\D}}
\renewcommand{\O}{{\mathcal O}}
\nc{\K}{{\mathcal K}}

\renewcommand{\S}{{\mathcal S}}
\nc{\T}{{\mathcal T}}
\nc{\V}{{\mathcal V}}
\renc{\P}{{\mathcal P}}
\nc{\A}{{\mathcal A}}
\nc{\B}{{\mathcal B}}
\nc{\U}{{\mathcal U}}
\renewcommand{\L}{{\mathcal L}}

\renewcommand{\mod}{{\on{-mod}}}
\newcommand{\comod}{{\on{-comod}}}

\nc\mathi\iota
\nc\Spec{\on{Spec}}
\nc\Mod{\on{Mod}}

\nc{\jw}{\widetilde j}

\nc{\I}{\mathcal I}

\nc{\lambdach}{{\check\lambda}}
\nc{\Lambdach}{{\check\Lambda}{}}
\nc{\much}{{\check\mu}}
\nc{\omegach}{{\check\omega}}
\nc{\nuch}{{\check\nu}}
\nc{\etach}{{\check\eta}}
\nc{\alphach}{{\check\alpha}}

\nc{\rhoch}{{\check\rho}}

\nc{\Hb}{\overline{\H}}

\nc{\BA}{{\mathbb{A}}}
\nc{\BC}{{\mathbb{C}}}
\nc{\BE}{{\mathbb{E}}}
\nc{\BF}{{\mathbb{F}}}
\nc{\BG}{{\mathbb{G}}}
\nc{\BM}{{\mathbb{M}}}
\nc{\BO}{{\mathbb{O}}}
\nc{\BD}{{\mathbb{D}}}
\nc{\BN}{{\mathbb{N}}}
\nc{\BP}{{\mathbb{P}}}
\nc{\BR}{{\mathbb{R}}}
\nc{\BZ}{{\mathbb{Z}}}
\nc{\BS}{{\mathbb{S}}}
\nc{\BV}{{\mathbb{V}}}

\nc{\CA}{{\mathcal{A}}}
\nc{\CB}{{\mathcal{B}}}

\nc{\CE}{{\mathcal{E}}}
\nc{\CF}{{\mathcal{F}}}
\nc{\CH}{{\mathcal{H}}}

\nc{\CL}{{\mathcal{L}}}
\nc{\CW}{{\mathcal{W}}}
\nc{\CI}{{\mathcal{I}}}

\nc{\csM}{{\check{\mathcal A}}{}}
\nc{\oM}{{\overset{\circ}{\mathcal M}}{}}
\nc{\obM}{{\overset{\circ}{\mathbf M}}{}}
\nc{\oCA}{{\overset{\circ}{\mathcal A}}{}}
\nc{\obA}{{\overset{\circ}{\mathbf A}}{}}
\nc{\ooM}{{\overset{\circ}{M}}{}}
\nc{\osM}{{\overset{\circ}{\mathsf M}}{}}
\nc{\vM}{{\overset{\bullet}{\mathcal M}}{}}
\nc{\nM}{{\underset{\bullet}{\mathcal M}}{}}
\nc{\oD}{{\overset{\circ}{\mathcal D}}{}}
\nc{\obC}{{\overset{\circ}{\mathbf C}}{}}
\nc{\obD}{{\overset{\circ}{\mathbf D}}{}}
\nc{\oA}{{\overset{\circ}{\mathbb A}}{}}
\nc{\op}{{\overset{\bullet}{\mathbf p}}{}}
\nc{\oU}{{\overset{\bullet}{\mathcal U}}{}}
\nc{\oZ}{{\overset{\circ}{\mathcal Z}}{}}

\nc{\oF}{{\overset{\circ}{\fF}}}

\nc{\bb}{{\mathbf{b}}}
\nc{\bc}{{\mathbf{c}}}
\nc{\bd}{{\mathbf{d}}}
\nc{\bbf}{{\mathbf{f}}}
\nc{\be}{{\mathbf{e}}}
\nc{\bg}{{\mathbf{g}}}
\nc{\bi}{{\mathbf{i}}}
\nc{\bj}{{\mathbf{j}}}
\nc{\bn}{{\mathbf{n}}}
\nc{\bo}{{\mathbf{o}}}
\nc{\bp}{{\mathbf{p}}}
\nc{\bq}{{\mathbf{q}}}
\nc{\bt}{{\mathbf{t}}}
\nc{\bu}{{\mathbf{u}}}
\nc{\bv}{{\mathbf{v}}}
\nc{\bx}{{\mathbf{x}}}
\nc{\bs}{{\mathbf{s}}}
\nc{\by}{{\mathbf{y}}}
\nc{\bw}{{\mathbf{w}}}
\nc{\bA}{{\mathbf{A}}}
\nc{\bK}{{\mathbf{K}}}
\nc{\bG}{{\mathbf{G}}}
\nc{\bD}{{\mathbf{D}}}
\nc{\bH}{{\mathbf{H}}}
\nc{\bM}{{\mathbf{M}}}
\nc{\bN}{{\mathbf{N}}}
\nc{\bO}{{\mathbf{O}}}
\nc{\bT}{{\mathbf{T}}}
\nc{\bV}{{\mathbf{V}}}
\nc{\bW}{{\mathbf{W}}}
\nc{\bX}{{\mathbf{X}}}
\nc{\bZ}{{\mathbf{Z}}}
\nc{\bS}{{\mathbf{S}}}

\nc{\sA}{{\mathsf{A}}}
\nc{\sB}{{\mathsf{B}}}
\nc{\sC}{{\mathsf{C}}}
\nc{\sD}{{\mathsf{D}}}
\nc{\sF}{{\mathsf{F}}}
\nc{\sG}{{\mathsf{G}}}
\nc{\sK}{{\mathsf{K}}}
\nc{\sM}{{\mathsf{M}}}
\nc{\sO}{{\mathsf{O}}}
\nc{\sW}{{\mathsf{W}}}
\nc{\sQ}{{\mathsf{Q}}}
\nc{\sP}{{\mathsf{P}}}
\nc{\sV}{{\mathsf{V}}}
\nc{\sS}{{\mathsf{S}}}
\nc{\sT}{{\mathsf{T}}}
\nc{\sZ}{{\mathsf{Z}}}
\nc{\sfp}{{\mathsf{p}}}
\nc{\sll}{{\mathsf{l}}}
\nc{\sr}{{\mathsf{r}}}
\nc{\bk}{{\mathsf{k}}}
\nc{\sg}{{\mathsf{g}}}

\nc{\sff}{{\mathsf{f}}}
\nc{\sfb}{{\mathsf{b}}}
\nc{\sfc}{{\mathsf{c}}}
\nc{\sd}{{\mathsf{d}}}
\nc{\se}{{\mathsf{e}}}

\nc{\BK}{{\bar{K}}}

\nc{\tA}{{\widetilde{\mathbf{A}}}}
\nc{\tB}{{\widetilde{\mathcal{B}}}}
\nc{\tG}{{\widetilde{G}}}
\nc{\TM}{{\widetilde{\mathbb{M}}}{}}
\nc{\tO}{{\widetilde{\mathsf{O}}}{}}
\nc{\TZ}{{\tilde{Z}}}
\nc{\tx}{{\tilde{x}}}
\nc{\tbv}{{\tilde{\bv}}}
\nc{\tz}{{\tilde{\zeta}}}
\nc{\tmu}{{\tilde{\mu}}}

\nc{\urho}{\underline{\rho}}
\nc{\uB}{\underline{B}}
\nc{\uC}{{\underline{\mathbb{C}}}}
\nc{\ui}{\underline{i}}
\nc{\uj}{\underline{j}}

\nc{\oB}{{\overline{\mathcal{B}}}}

\nc{\oI}{{\overline{I}}}

\nc{\eps}{\varepsilon}
\nc{\hrho}{{\hat{\rho}}}

\nc{\one}{{\mathbf{1}}}
\nc{\two}{{\mathbf{t}}}

\nc{\Rep}{{\mathop{\operatorname{\rm Rep}}}}
\nc{\Tot}{{\mathop{\operatorname{\rm Tot}}}}
\nc{\Ker}{{\mathop{\operatorname{\rm Ker}}}}
\nc{\Hilb}{{\mathop{\operatorname{\rm Hilb}}}}
\nc{\End}{{\mathop{\operatorname{\rm End}}}}
\nc{\Ext}{{\mathop{\operatorname{\rm Ext}}}}
\nc{\CHom}{{\mathop{\operatorname{{\mathcal{H}}\it om}}}}
\nc{\GL}{{\mathop{\operatorname{\rm GL}}}}
\nc{\gr}{{\mathop{\operatorname{\rm gr}}}}
\nc{\Id}{{\mathop{\operatorname{\rm Id}}}}
\nc{\de}{{\mathop{\operatorname{\rm def}}}}
\nc{\length}{{\mathop{\operatorname{\rm length}}}}
\nc{\supp}{{\mathop{\operatorname{\rm supp}}}}

\nc{\Fl}{\on{Fl}}
\nc{\Fib}{{\mathsf{Fib}}}
\nc{\Coh}{{\mathsf{Coh}}}
\nc{\QCoh}{{\mathsf{QCoh}}}
\nc{\IndCoh}{{\mathsf{IndCoh}}}
\nc{\FCoh}{{\mathsf{FCoh}}}

\nc{\reg}{{\text{\rm reg}}}

\nc{\cplus}{{\mathbf{C}_+}}
\nc{\cminus}{{\mathbf{C}_-}}
\nc{\cthree}{{\mathbf{C}_*}}
\nc{\Qbar}{{\bar{Q}}}
\nc\Eis{\on{Eis}}
\nc\Eisb{\ol\Eis{}}
\nc\Eisr{\on{Eis}^{rat}{}}
\nc{\Def}{\on{Def_{\check{\fb}}(E)}}
\nc{\barZ}{\overline{Z}{}}
\nc{\barbarZ}{\overline{\barZ}{}}
\nc{\barpi}{\overline\pi}
\nc{\barbarpi}{\overline\barpi}
\nc{\barpip}{\overline\pi{}^+}
\nc{\barpim}{\overline\pi{}^-}

\nc{\fqb}{\ol{\fq}{}}
\nc{\fpb}{\ol{\fp}{}}
\nc{\fpr}{{\fp^{rat}}{}}
\nc{\fqr}{{\fq^{rat}}{}}

\nc{\bh}{{\bar{h}}}
\nc{\bOmega}{{\overline{\Omega(\check \fn)}}}

\nc{\seq}[1]{\stackrel{#1}{\sim}}

\nc{\DefbE}{{\on{Def}_{\cB}(E_\cT)}}

\nc{\imathb}{{\ol{\imath}}}
\nc{\rlr}{\overset{\longrightarrow}{\underset{\longrightarrow}\longleftarrow}}

\nc{\oBun}{\overset{\circ}\Bun}
\nc{\LocSys}{\on{LocSys}}
\nc{\BunBbb}{\ol{\ol{Bun}}_B}
\nc{\BunBr}{\Bun_B^{rat}}
\nc{\BunBrsg}{\Bun_B^{rat,\on{s.g.}}}
\nc{\BunBrp}{\Bun_B^{rat,polar}}
\nc{\BunBrpbg}{\Bun_B^{rat,polar,\on{b.g.}}}
\nc{\BunBrpsg}{\Bun_B^{rat,polar,\on{s.g.}}}
\nc{\BunTrp}{\Bun_T^{rat,polar}}
\nc{\BunTrpbg}{\Bun_T^{rat,polar,\on{b.g.}}}
\nc{\BunTrpsg}{\Bun_T^{rat,polar,\on{s.g.}}}
\nc{\BunNr}{\Bun_N^{rat}}
\nc{\BunNre}{\Bun_N^{enh,rat}}
\nc{\BunTr}{\Bun_T^{rat}}
\nc{\Vect}{\mathsf{Vect}}
\nc{\Ran}{{\mathsf{Ran}}}
\nc\jmathr{\jmath^{rat}{}}
\nc{\ux}{\underline{x}}

\nc{\ind}{{\mathbf{ind}}}
\nc{\oblv}{{\mathsf{oblv}}}
\nc{\Oblv}{{\mathsf{Oblv}}}
\nc{\fset}{\mathsf{fSet}}
\nc{\LocSysG}{\LocSys_{\cG}}
\nc{\Sing}{\on{Sing}}
\nc{\dr}{{\on{dR}}}
\nc{\Ind}{\on{Ind}}
\nc{\Sat}{\on{Sat}}
\nc{\Ho}{\on{Ho}}
\nc{\Res}{\on{Res}}
\nc{\sotimes}{\overset{!}\otimes}

\nc{\mmod}{{\on{-}}{\mathbf{mod}}}
\nc{\ccomod}{{\on{-}}{\mathbf{comod}}}

\nc{\dgSch}{\on{DGSch}}
\nc{\dgindSch}{\on{DGindSch}}
\nc{\indSch}{\on{indSch}}
\nc{\Sch}{\mathsf{Sch}}
\nc{\affdgSch}{\on{DGSch}^{\on{aff}}}

\nc{\Grpd}{\mathsf{Grpd}}
\nc{\inftyGrpd}{\infty\on{-}\ms{Grpd}}
\nc{\inftyCat}{\infty\on{-Cat}}
\nc{\MoninftyCat}{\infty\on{-Cat}^{Mon}}
\nc{\SymMoninftyCat}{\infty\on{-Cat}^{\on{SymMon}}}
\nc{\SymMonStinftyCat}{\on{DGCat}^{\on{SymMon}}}
\nc{\MonStinftyCat}{\on{DGCat}^{Mon}}
\nc{\inftystack}{\on{Stk}}
\nc{\inftystackalg}{Stk^{1\text{-}alg}}
\nc{\inftyprestack}{\on{PreStk}}
\nc{\inftydgnearstack}{\on{NearStk}}
\nc{\inftydgstack}{\on{Stk}}
\nc{\inftydgstackalg}{DGStk^{1\text{-}alg}}
\nc{\inftydgprestack}{\on{PreStk}}

\nc{\csupp}{\supp}
\nc{\Arth}{\on{Arth}}
\nc{\ArthG}{{\on{Arth}_\cG}}
\nc{\ul}{\underline}

\nc{\Z}{\mathcal{Z}}

\nc{\calN}{\N}
\nc{\calW}{\mathcal{W}}
\nc{\calF}{\mathcal{F}}
\nc{\calH}{\mathcal{H}}
\nc{\calO}{\mathcal{O}}
\nc{\calK}{\mathcal{K}}

\nc{\Jets}{\on{Jets}}
\nc{\act}{\mathsf{act}}
\nc{\Av}{\mathsf{Av}}
\nc{\Ad}{\on{Ad}}
\nc{\BGRan}{BG_{\Ran}}
\nc{\colim}{\on{colim}}
\nc{\codim}{\on{codim}}
\nc{\cpt}{{\on{cpt}}}
\nc{\dR}{{\on{dR}}}
\nc{\DGCat}{\mathsf{DGCat}}
\nc{\DGCatcont}{\on{DGCat}_{cont}}
\nc{\glob}{{\on{glob}}}
\nc{\loc}{{\on{loc}}}

\renewcommand{\op}{{\on{op}}}
\nc{\pt}{{\on{pt}}}
\nc{\PreStk}{{\mathsf{PreStk}}}
\nc{\Cat}{{\on{Cat}}}
\nc{\DGSchaff}{{\on{DGSch}^{\text{aff}}}}
\nc{\ShvCat}{{\on{ShvCat}}}

\nc{\restr}[2]{\left. #1 \right |_{#2}}
\nc{\uprestr}[2]{\left. #1 \right |^{#2}}

\nc{\bLoc}{{\mathbf{Loc}}}
\nc{\bGamma}{{\mathbf{\Gamma}}}

\nc{\gen}{\mathsf{gen}}

\nc{\hto}{\hookrightarrow}

\nc{\Fun}{\on{Fun}}
\nc{\ext}{\mathsf{ext}}

\nc{\ev}{\mathsf{ev}}
\nc{\rat}{\mathsf{rat}}

\nc{\usotimes}[1]{\underset{#1}{\otimes}}
\nc{\ustimes}[1]{\underset{#1}{\times}}
\nc{\uscolim}[1]{\underset{#1}{\colim}}

\nc{\ch}{{\mathfrak{ch}}}

\nc{\fD}{{\Dmod}}
\nc{\fH}{{\mathfrak{H}}}

\nc{\p}{{\mathfrak{p}}}
\renc{\r}{{\mathfrak{r}}}

\nc{\xto}{\xrightarrow}

\renc{\sec}{\section}
\nc{\enh}{\mathsf{enh}}

\renc{\gen}{\mathsf{gen}}
\nc{\BunGBgen}{\Bun_G^{B-\gen}}
\nc{\BunGHgen}{\Bun_G^{H-\gen}}
\nc{\BunGNgen}{\Bun_G^{N-\gen}}

\renc{\Fun}{\mathsf{Fun}}

\nc{\rr}{\xymatrix{ \ar@<-0.1ex>[r] \ar@<.8ex>[r]  & } }
\nc{\rrr}{\xymatrix{ \ar@<.2ex>[r] \ar@<.9ex>[r] \ar@<-0.5ex>[r] & } }

\nc{\Stab}{\mathsf{Stab}}
\nc{\Orb}{\mathsf{Orb}}

\renc{\exp}{\mathit{exp}}

\renc{\q}{\mathfrak{q}}

\nc{\virg}[1]{``#1"}

\renc{\bold}[1]{\boldsymbol{#1}}

\nc{\bigt}[1]{\big( #1 \big) }
\nc{\Bigt}[1]{\Big( #1 \Big) }

\nc{\extwhit}{{\CW h}(G,\mathsf{ext})}

\nc{\footcite}{\footnote}

\nc{\GA}{G(\AA)}
\nc{\GO}{G(\OO)}

\nc{\Shv}{\mathsf{Shv}}
\nc{\inc}{\mathsf{inc}}

\nc{\Par}{\mathsf{Par}}

\renc{\i}{\mathfrak{i}}

\nc{\NA}{N(\AA)}
\nc{\VA}{V(\AA)}

\nc{\Glue}{\mathsf{Glue}}
\nc{\laxlim}{\text{laxlim}}

\nc{\Dom}{\mathsf{Dom}}

\nc{\FT}{\mathsf{FT}}

\nc{\out}{\mathsf{out}}

\nc{\hol}{\mathsf{hol}}
\nc{\Hol}{\on{Hol}}
\nc{\add}{\mathsf{add}}

\nc{\sto}{\rightsquigarrow}
\nc{\squigto}{\rightsquigarrow}

\nc{\fW}{\mathfrak{W}}

\nc{\vrho}{\varrho}

\nc{\counit}{\mathsf{counit}}
\nc{\unit}{\mathsf{unit}}
\nc{\corr}{\mathsf{corr}}

\nc{\IndSch}{\mathsf{IndSch}}
\nc{\Tate}{{\mathsf{Tate}}}

\nc{\surjto}{\twoheadrightarrow}

\renc{\j}{\mathfrak{j}}

\renc{\H}{\mathcal{H}}

\nc{\pro}{\mathsf{pro}}
\nc{\fty}{\mathsf{ft}}
\nc{\Pro}{\mathsf{Pro}}

\nc{\coact}{\mathsf{coact}}
\nc{\aff}{\mathsf{aff}}

\nc{\Nilp}{\on{Nilp}}
\nc{\Gch}{{\check{G}}}
\nc{\LL}{\mathbb{L}}

\nc{\LS}{\LocSys}

\nc{\x}{\varkappa} 
\nc{\ms}{\mathsf}

\nc{\Otimes}{\boldsymbol{\otimes}}
\nc{\Times}{\boldsymbol{\times}}
\nc{\flip}{\text{<}}

\nc{\coeffRan}{\mathsf{coeff}^{\Ran}}

\nc{\Ha}{H(\sA)}
\nc{\Groups}{\mathsf{Groups}}

\nc{\Groth}{\mathsf{Groth}}

\nc{\rlto}{\rightleftarrows}

\nc{\DGCatRan}{\ShvCatCrys(\Ran)}

\nc{\longto}{\longrightarrow}

\nc{\naive}{\ms{naive}}
\nc{\spec}{\mathit{spec}}

\renc{\Jets}{\mathsf{Jets}}
\nc{\mer}{\mathsf{mer}}

\nc{\W}{\mathcal{W}}

\nc{\Sect}{\mathsf{Sect}}
\nc{\Maps}{\mathsf{Maps}}

\renc{\bf}{\mathbf{f}}

\nc{\y}{\mathtt{y}}
\renc{\x}{\mathtt{x}}

\nc{\un}{{\it un}}
\nc{\indep}{\ms{indep}}
\nc{\CoAlg}{\ms{CoAlg}}

\nc{\coeff}{\ms{coeff}}

\nc{\R}{\mathcal{R}}

\renc{\hat}{\widehat}

\nc{\cores}{\bold{\on{cores}}}
\nc{\coind}{\bold{\on{coind}}}

\nc{\TtKK}{\Tt(\mathpzc{K})} 
\nc{\TtK}{\Tt(\mathsf{K})} 
\nc{\KK}{\mathpzc{K}}

\nc{\Dmod}{\mathfrak{D}}

\nc{\curs}[1]{\mathpzc{#1}}

\nc{\CC}{\mathcal{C}}
\nc{\Ccat}{\mathcal{C}}
\nc{\Cmod}[1]{\mathcal{C}_{#1}}
\nc{\Cshv}{\bold{\CC}}
\nc{\Cind}{\CC_{\indep}}
\nc{\CRan}{\CC_{\Ran}}

\nc{\Bshv}{\bold{\B}}
\nc{\Bind}{\B_{\indep}}
\nc{\BRan}{\B_{\Ran}}
\nc{\ARan}{\A_{\Ran}}
\nc{\Aind}{\A_{\indep}}

\nc{\Gr}{\mathsf{Gr}}
\nc{\GrGRan}{\Gr_{G}}
\nc{\GrGind}{\Gr_{G}^{\indep}}
\nc{\GrGdom}{\curs{Gr}_G}

\nc{\GMapsRan}[1]{\ms{GMaps}(X,{#1})}
\nc{\GSectRan}[1]{\ms{GSect}({#1}/X)}
\nc{\GMapsind}[1]{\ms{GMaps}(X,{#1})^\indep}
\nc{\GSectind}[1]{\ms{GSect}({#1}/X)^\indep}
\nc{\GMapsdom}[1]{\curs{GMaps}(X,{#1})}
\nc{\GSectdom}[1]{\curs{GSect}({#1}/X)}

\nc{\chind}{\ch^{\indep}}
\nc{\chdom}{\curs{ch}}

\nc{\QSect}[1]{\curs{QSect}(#1/X)} 
\nc{\QMaps}[1]{\curs{QMaps}(X,#1)} 

\nc{\Zar}{\mathit{Zar}}

\nc{\loccit}{\textit{loc.$\,$cit.}}
\nc{\ShvCatCrys}{\ShvCat\!^\nabla\!}

\nc{\BB}{\mathbb{B}}
\nc{\BPE}{{\BP E}}
\nc{\BVE}{{\BV E}}
\nc{\BBE}{{\BB E}}

\newcommand{\mapsfrom}{\mathrel{\reflectbox{\ensuremath{\mapsto}}}}

\begin{document}

\title{On the extended Whittaker category}
\author{Dario Beraldo}

{\let\thefootnote\relax}

\begin{abstract} 
Let $G$ be a connected reductive group with connected center and $X$ a smooth complete curve, both defined over an algebraically closed field of characteristic zero. Let $\Bun_G$ denote the stack of $G$-bundles on $X$.
In analogy with the classical theory of Whittaker coefficients for automorphic functions, we construct
a \virg{Fourier transform} functor $\coeff_{G,\ext}$ from the DG category of $\fD$-modules on $\Bun_G$ to a certain DG category $\extwhit$, called the \emph{extended Whittaker category}. This construction allows to formulate the compatibility of the Langlands duality functor $\LL_G: \IndCoh_\N(\LocSys_{\Gch}) \to \Dmod(\Bun_G)$ with the Whittaker model.
For $G=GL_n$ and $G=PGL_n$, we prove that $\coeff_{G,\ext}$ is fully faithful.  This result guarantees that, for those groups, $\LL_G$ is unique (if it exists) and necessarily fully faithful.
\end{abstract}

\maketitle

\tableofcontents

\sec{Introduction}

\renc{\extwhit}{\mathit{Wh}(G,\ext)}
\nc{\Wh}{\mathit{Wh}}
\nc{\kk}{\mathbbm{k}}

Let $G$ be a connected reductive group, with connected center, defined over an algebraically closed field $\kk$ of characteristic zero. 
We fix once and for all a maximal torus and a Borel subgroup containing it: $T \subseteq B \subseteq G$. The unipotent radical of $B$ will be denoted $N$. Let $r$ be the semisimple rank of $G$ and $(\alpha_1, \ldots, \alpha_r)$ the collection of simple (positive) roots.
We also fix Chevalley generators of $N$.

\medskip

Let $X$ be a smooth complete connected curve, also over $\mathbbm{k}$, and $\Bun_G$ the stack of $G$-bundles on $X$.
The main object of interest in the geometric Langlands program is $\Dmod(\Bun_G)$, the DG category of $\fD$-modules on $\Bun_G$.
Loosely speaking, in the present paper we propose to study this category using the Fourier transform. This idea, informally reviewed in Sections \ref{ssec:function-theoretic-analogy}-\ref{ssec:ext-whit-functions}, is taken directly from the theory of \emph{automorphic functions}, which are the function-theoretic analogue of $\Dmod(\Bun_G)$.

\ssec{A function theoretic analogy} \label{ssec:function-theoretic-analogy}

\sssec{}

In the geometric Langlands program we always work over an algebraically closed field of characteristic zero.
On the contrary, when discussing the function theoretic version of the program, which happens only in Sections \ref{ssec:function-theoretic-analogy}-\ref{ssec:ext-whit-functions}, we work over $\mathbb F_q$: i.e., $G$ is split reductive over $\mathbb F_q$ and $X$ is a curve over $\mathbb F_q$.
The reason is that in this case $G(\AA)$ is a locally compact group, so that harmonic analysis of locally compact groups becomes available.
Note furthermore that $G(\OO) \subset G(\AA)$ is a compact subgroup and $G(K) \subset G(\AA)$ a discrete subgroup. We fix once and for all a Haar measure on $\GA$.\footnote{By $\AA$ (resp. $\OO$), we have denoted the algebra of ad{\`{e}}les (resp., integral ad{\`{e}}les), while $K:=K(X)$ is the function field of $X$.}

Under the \virg{\emph{function theory $\leftrightarrow$ algebraic geometry}} dictionary spelled out in the course of this paper, we will see that the terms \virg{locally compact group}, \virg{compact group}, \virg{discrete group} go over to \virg{group indscheme of pro-finite type}, \virg{group scheme}, \virg{group indscheme of finite type} respectively. See Sections \ref{SEC:background} and \ref{SEC:FT}.

\sssec{Whittaker invariant functions}

By definition, the vector space of \emph{automorphic functions} for the group $G$ consists of $\mathbb C$-valued functions on the double quotient 
$$
G(K) \backslash G(\AA) / G(\OO).
$$
We often view automorphic functions as $G(K)$-invariant functions on the quotient $G(\AA) / G(\OO)$ and denote them by $\Fun^{G(K)}(G(\AA) / G(\OO))$.

\medskip

The classical strategy to analyze $\Fun^{G(K)}(G(\AA) / G(\OO))$ is to forget the $G(K)$-invariance and then decompose $\Fun(G(\AA)/ G(\OO))$, considered as a representation of $N(\AA)$,
into eigenspaces parametrized by the ($\BC^\times$-valued) characters of $N(\AA)$. 
Explicitly, for $\chi$ a character of $N(\AA)$, the $\chi$-eigenspace is by definition 
$$
\Fun^{\NA, \chi} (\GA/\GO) := \{f \in \Fun(\GA/\GO) \, | \, f(ng) = \chi(n) f(g) \, \text{ for any }\, n \in \NA\}.
$$
We call $\Wh^\chi := \Fun^{\NA, \chi} (\GA/\GO)$ the space of \emph{$\chi$-Whittaker functions}.

\sssec{}

Thanks to the factorization
\begin{equation} \label{eqn:oblv:2}
\Fun^{G(K)}(G(\AA) / G(\OO)) \to 
\Fun^{N(K)} (G(\AA) / G(\OO)) \to
\Fun(\GA/\GO),
\end{equation}
the only characters of $N(\AA)$ that matter are the ones trivial on $N(K) \subset N(\AA)$. Let $\ch$ denote the space of such characters and pick $\chi \in \ch$. It is convenient to regard $\Fun^{\NA, \chi} (\GA/\GO)$ as the space of functions on $N(K) \backslash \GA/\GO)$ that are further $(N(\AA), \chi)$-equivariant. In particular, since the quotient $\NA/N(K)$ is compact, we may define the operator of \emph{$\chi$-Whittaker coefficient}:
$$
\coeff_{G,\chi}:\Fun^{G(K)}(G(\AA) / G(\OO)) \to 
\Fun^{N(\AA), \chi}( G(\AA) / G(\OO)) =
\Fun^{N(\AA), \chi}( N(K) \backslash G(\AA) / G(\OO)), 
$$
$$
f \mapsto (\coeff_{G,\chi}f) (g) := \int_{ \NA / N(K)} f(n \cdot g) \, \chi^{-1}(n) \, dn.
$$

\sssec{} \label{intro:ch-in-function theory}

The vector space $\ch$ is non-canonically isomorphic to $K^{\oplus r}$. The isomorphism depends on two pieces of data: an \virg{exponential function}, i.e. a character $\exp: \mathbb F_q \to \BC^\times$, and a meromorphic $1$-form on $X$. When discussing the function theoretic version of the construction, we assume these have been fixed once and for all. (On the contrary, in the geometric counterpart of the construction, we will work canonically and make neither of these choices.)

\medskip

The above choices allow to define the character $\psi := \exp \circ \Res$ on $\AA/K$. The mentioned isomorphism $K^{\oplus r}\simeq \ch$ is then described as $(f_1, \ldots, f_r) \mapsto \chi_{f_1, \ldots, f_r}$,
where
\begin{equation} \label{eqn:characters=K^r}
\chi_{f_1, \ldots, f_r} (a_1, \ldots, a_r) := \psi ( f_1 a_1 + \cdots + f_r a_r).
\end{equation}
Indeed, characters of $N(\AA)$ factor through its abelianization, which is canonically the direct sum of $r$ copies of $\AA$, one for each simple root.

\medskip

Amongst the elements of $\ch$, there are $2^r$ special ones that will play a role in the sequel. They are in bijection with subsets of $\{\alpha_1 \ldots, \alpha_r\}$ (equivalently, subsets of the Dynkin diagram $\I$) and correspond to setting $f_i=1$ or $f_i=0$ in (\ref{eqn:characters=K^r}) for each $i \in \I$. 
In turn, subsets of $\I$ parametrize standard parabolics subgroups of $G$: for each $\I_M \subseteq \I$, let $P$ be the standard parabolic whose Levi decomposition is $M \ltimes U_P$, where $M \supseteq T$ has simple roots corresponding to $\I_M$. Denote by $\chi_{G,P}$ the corresponding character. For instance, $\chi_{G,B} =0$ whereas $\chi_{G,G}$ is maximally non-degenerate.

\sssec{} \label{sssec:TK acts on NA}

The action of $T(K)$ on $N(\AA)$ preserves $N(K)$, whence $T(K)$ acts on $\ch$ too.
Recall that $G$ is assumed to have connected center (e.g., $G= GL_n$, $PGL_n$). This condition guarantees that the $T(K)$-orbits of $\ch$ are in bijection with the set of standard parabolics:
$$
\ch = \bigsqcup_{P \supseteq B} \ch_{G,P},
$$
where $\ch_{G,P}$ is the set of characters $\chi_{f_1, \ldots f_r}$ such that $f_i \neq 0$ for $i \in \I_M$ and zero otherwise. Our privileged representative of the orbit $\ch_{G,P}$ is, needless to say, $\chi_{G,P}$.

\medskip

By viewing $\Fun(\GA/\GO)$ as a representation of $T(K) \ltimes N(\AA)$, we see that Fourier eigenspaces corresponding to characters in the same $T(K)$-orbit are canonically identified and that the stabilizer of $\chi_{G,P}$ in $T(K)$ is exactly $Z_M(K)$. Hence, $\coeff_{G,P}:=\coeff_{G,\chi_{G,P}}$ will be regarded as a map
$$
\coeff_{G,P}:
\Fun^{G(K)}(G(\AA) / G(\OO)) 
\to 
\Fun^{N(\AA), \chi_{G,P}}( Z_M(K) \backslash G(\AA) / G(\OO)).
$$
To simplify the notation, set 
\begin{equation} \label{eqn:wh(G,P)-at-privileged char}
\Wh(G,P) := \Fun^{N(\AA), \chi_{G,P}}( Z_M(K) \backslash G(\AA) / G(\OO)).
\end{equation}

\sssec{}

Putting all Whittaker coefficients together, we obtain the operator
\begin{equation} \label{coeff-P-all}
\underset{P \supseteq B}{\prod} \coeff_{G,P}:
\Fun^{G(K)}(G(\AA) / G(\OO)) 
\longto 
\prod_{P \supseteq B} 
\Wh(G,P).
\end{equation}
When $G= GL_2$ or $G= PGL_2$, the group $N(\AA)$ is a vector group: hence, the theory of Fourier transform ensures that the above map is injective. \emph{For general $G$, there is no reason to expect this.} However, a special feature of $GL_n$ ensures that the following result holds.

\begin{thm} \label{thm:GLn-functiontheory}
For $G= GL_n$ or $G=PGL_n$, the operator (\ref{coeff-P-all}) is injective.
\end{thm}

\sssec{}

The mentioned feature of $GL_n$ (and consequently of $PGL_n$) has to do with the standard parabolic $Q \subseteq GL_n$ corresponding to the partition $n = (n-1)+1$, whose Levi decomposition reads $Q \simeq (GL_{n-1} \times \GG_m) \ltimes V$ with $V:=(\GG_a)^{n-1}$. Theorem \ref{thm:GLn-functiontheory} follows immediately from the stronger result below.

\begin{thm} 	\label{thm:mirabolic-fun-theory}
The map (\ref{coeff-P-all}) factors as the composition of two injections:
\begin{equation} \label{eqn:pirabolic-fun-theory}
\Fun^{GL_n(K)}( GL_n(\AA) / GL_n(\OO) )  \xto{ \, \mathsf{forget} \,}
\Fun^{Q(K)}( GL_n(\AA) / GL_n(\OO) ) \hto
\prod_{P \supseteq B} \Wh(GL_n,P).
\end{equation}
\end{thm}

\sssec{} \label{sssec:sketch}

Let us briefly sketch the argument (see \cite{Shalika}, \cite{Laumon}). The action 
$$
(GL_{n-1} \times \GG_m)(K) \, \, 
{\displaystyle\curvearrowright}
\,\, V(K)^\perp :=\{\mbox{characters of $V(\AA)$ trivial on $V(K)$}\}
$$
has two orbits, the singleton $\{0 \}$ and its complement. Choose the obvious representatives $0$ and $\xi_{n-1}$ of these orbits\footnote{$\xi_{n-1}$ is the \virg{last} standard element of $V(K)^\perp$}, with stabilizer $(GL_{n-1} \times \GG_m)(K)$ and $Q' (K)$ respectively. Here $Q'$ is the standard parabolic of $GL_{n-1}$ corresponding to the partition $(n-2)+1$.
Then the Fourier transform implies that
\begin{equation} \label{eqn:induction-function-theory}
\begin{array}{lcl}
\Fun^{Q(K)}( GL_n(\AA) / GL_n(\OO) ) 
& \simeq &
\Fun^{Q'(K) \ltimes V(\AA), \xi_{n-1}} ( GL_n(\AA) / GL_n(\OO))
\vspace{.2cm} \\
& & \times \,\,
\Fun^{(GL_{n-1} \times \GG_m)(K) \ltimes V(\AA)} (  GL_n(\AA) / GL_n(\OO)).
\end{array}
\end{equation}
If $n=2$, we are done, and we have actually proved that the rightmost arrow in (\ref{eqn:pirabolic-fun-theory}) is a bijection. If $n \geq 3$, we treat each of the two terms in (\ref{eqn:induction-function-theory}) again by Fourier transform. The first term is the space of $Q'(K)$-invariants of $\Fun^{V(\AA), \xi_{n-1}} ( GL_n(\AA) / GL_n(\OO))$. We apply the above reasoning for $Q'(K)$ to split it in two.
As for the second term, we write
$$
\Fun^{(GL_{n-1} \times \GG_m)(K) \ltimes V(\AA)} (  GL_n(\AA) / GL_n(\OO))
\hto
\Fun^{Q'(K) \times (\GG_m(K) \ltimes V(\AA))} (  GL_n(\AA) / GL_n(\OO)),
$$
and again apply Fourier transform with respect to $Q'(K)$.
Iterating, we end up with the product of $\Wh(GL_n, P)$'s for all standard parabolic subgroups of $GL_n$. It is immediately checked that this sequence of Fourier transforms concides with the product of the coefficient operators.

\ssec{The main geometric trick} \label{ssec:trick intro}

In the geometric setting, the proof of the above statement proceeds along different lines, resorting to a blow-up trick which could be regarded as the main new idea of the present paper.
Such idea is carried out in Sections \ref{ssec:blowup} and \ref{ssec:crucial}. Here we just describe the two main ingredients: the first is elementary, the second one less so.

\sssec{The elementary ingredient}

The above space of characters $V(K)^\perp$ is isomorphic to $\AA^{n-1}(K)$. Considering the blow-up $\wt{\AA^{n-1}}$ of $\AA^{n-1}$ at the origin, we observe that the \virg{last} standard inclusion $\iota_{n-1}: \AA^1 \hto \AA^{n-1}$ yields the isomorphism
$$
\AA^1/(Q' \times \GG_m)
\simeq
\wt{\AA^{n-1}}/(GL_{n-1} \times \GG_m).
$$
In other words, blowing up allows to do induction on $n$ without having to break $V(K)^\perp$ in two parts.

\sssec{The non-elementary ingredient}

Let $W$ be a finite dimensional vector space and $\wt W$ its blow-up at the origin.
In Section \ref{ssec:crucial}, we show that the geometric version of the tautological map $\pi: \wt W(K) \to W(K) $ enjoys excellent properties: first, its fibers are cohomologically contractible and, second, it behaves like a proper map.

\ssec{Extended Whittaker functions} \label{ssec:ext-whit-functions}

It is convenient to repackage the product $\prod_{P \supseteq B} 
\Wh(G,P)$ into a single space of functions as follows.

\begin{defn} \label{defn:extWhit-functions}
By $\extwhit$, called the space of \emph{extended Whittaker functions}, we denote the space of functions on $G(\AA)/G(\OO) \times \ch$ satisfying

\begin{itemize}

\smallskip

\item 
\textit{$N(\AA)$-equivariance}:
$f(n \cdot g, \chi ) = \chi(n) \cdot f(g, \chi)$ for all $n \in \NA$;

\smallskip

\item
\textit{$T(K)$-invariance}:
$f(t \cdot g, \Ad_t(\chi)) = f(g, \chi)$ for all $t \in T(K)$.

\end{itemize}
\end{defn}

\sssec{}

We often regard $\extwhit$ as the subspace of \emph{$N(K)$-invariant} functions on $G(\AA)/G(\OO) \times \ch$ determined by the same conditions as above. Again, as the quotient $\NA/N(K)$ is compact, we are allowed to define the \emph{extended coefficient operator}
$$
\coeff_{G,\ext} :\Fun(  G(K) \backslash G(\AA) / G(\OO) ) 
\to
\extwhit,
\hspace*{.35cm}
f \mapsto \wt f (g, \chi) := \int_{\NA/N(K)} f(n \cdot g) \, \chi^{-1}(n) \, dn.
$$
To better replicate this operation in the geometric context, observe that $\coeff_{G,\ext}$ is the composition of three operators of different kind:
$$
\begin{array}{l}
\Fun(G(K) \backslash \GA/\GO)
\xto{\mathsf{forget}}
\Fun \Big( T(K)  \big\backslash \big( N(K) \backslash \GA/\GO \big) \Big) 
\vspace{.3cm} \\
\Fun \Big( T(K)  \big\backslash \big( N(K) \backslash \GA/\GO \big) \Big)
\xto{\mathsf{pullback}}
\Fun \Big( T(K)  \big\backslash \big( N(K) \backslash \GA/\GO \times \ch \big) \Big) 
\vspace{.3cm}
\\
\Fun \Big( T(K)  \big\backslash \big( N(K) \backslash \GA/\GO \times \ch \big) \Big)
\xto{\mathsf{integrate}}
\Wh(G,\ext).
\end{array}
$$

\begin{rem} \label{rem:source-coeff-operator}
In particular, $\coeff_{G,\ext}$ is well-defined on any subspace of $\Fun^{B(K)} (\GA/\GO)$, e.g. $\Fun^{P(K)} (\GA/\GO)$ with $P \subseteq G$ an arbitrary standard parabolic. 
\end{rem}

\sssec{} \label{sssec:degenerate-Whittaker-functions}

Along the lines of Definition \ref{defn:extWhit-functions}, $\Wh(G,P)$ can be viewed as the space of functions on $\GA/\GO \times \ch_{G,P}$ satisfying the same two conditions. Indeed, restriction along $\{\chi_{G,P}\} \subset \ch_{G,P}$ yields an equivalence between the latter definition of $\Wh(G,P)$ and (\ref{eqn:wh(G,P)-at-privileged char}).

\medskip

Restriction along the locally closed embedding $\ch_{G,P} \subset \ch$ (or alternatively $\{\chi_{G,P}\} \subset \ch$) induces the map $\vrho_{G,\ext \to P}: \extwhit \to \Wh(G,P)$. Tautologically:

\begin{prop} \label{prop:taut-product-coeff-function-theory}
The map
$$
\prod_{P \supseteq B} \vrho_{G,\ext \to P} : \extwhit \to \Wh(G,P)
$$ 
is an isomorphism. 
\end{prop}

Furthermore, the above map fits in the commutative triangle
\begin{gather}  \label{diag:triangle}
\xy
(30,0)*+{ \Fun^{G(K)}(\GA/\GO).}="basso";
(60,25)*+{ \underset{P \supseteq B}\prod \Wh(G,P) }="dx";
(0,25)*+{ \extwhit }="sx";
{\ar@{->}_{\simeq}^{ \prod \vrho_{G,\ext \to P} } "sx";"dx"};
{\ar@{->}^{ \coeff_{G,\ext}} "basso";"sx"};
{\ar@{->}_{ \prod \coeff_{G,P} } "basso";"dx"};
\endxy
\end{gather}	

\ssec{What is done in this paper}

\renc{\Wh}{\CW{h}}
\renc{\extwhit}{\Wh(G,\ext)}

\nc{\WhRan}{\Wh^{\Ran}}

In this paper, we wish to provide a geometric analog of the story narrated above, and especially of Theorem \ref{thm:GLn-functiontheory}.

\medskip

More specifically, we will first define a DG category $\Wh(G,\ext)$ (respectively, $\Wh(G,P)$) by directly translating Definition \ref{defn:extWhit-functions} (respectively, formula (\ref{eqn:wh(G,P)-at-privileged char})).
This is the subject of Section \ref{SEC:ExtWhit}. 
Before getting there, we need to discuss three pieces of algebro-geometric technology.

The first is the theory of $\fD$-modules on ind-schemes or pro-finite type. This is necessary due to the presence of $N(\AA)$. The relevant aspects of the theory are
recalled (amongst other background material) in Section \ref{SEC:background}.

The second ingredient is the theory of prestacks parametrizing generically-defined structures. This is discussed in Section \ref{SEC:unital} and it is needed to treat the geometrizations of, e.g., $N(K), \ch, G(\AA)/G(\OO)$. 

The third ingredient, treated in Section \ref{SEC:FT}, is the theory of the Fourier-Deligne transform for meromorphic jets into a vector group, that is, the geometrization of $V(\AA)$. This is enough to prove, \textit{ante litteram}, our main theorem for $G=GL_2$ and $G=PGL_2$.

\medskip

The way $\Wh(G,\ext)$ is assembled out of the $\Wh(G,P)$'s is more subtle than in the function theoretic case: it is an instance of \emph{gluing}, or \emph{recollement}, of DG categories. In technical terms, $\Wh(G,\ext)$ is equivalent to the \emph{lax limit} of the $\Wh(G,P)$'s along some specific functors. 
More precisely (see \cite{Outline} for details):
\smallskip

\begin{itemize}
\item[(1)]
each \virg{restriction} functor $\vrho_{G,\ext \to P} : \Wh(G,\ext) \to \Wh(G,P)$ admits a left adjoint; 
\smallskip

\item[(2)]
by combining the counit of these adjunctions, one obtains a lax functor 
$$
\Wh(G, ?): (\Par, \subseteq)^\op \to \DGCat,
$$
where $\Par$ is the poset of standard parabolics;
\smallskip

\item[(3)]
taken together, the $\vrho$'s yield a functor $\Wh(G,\ext) \to \laxlim \,\Wh(G,?)$, which can be proven to be an equivalence. This is the counterpart of Proposition \ref{prop:taut-product-coeff-function-theory}, no longer tautological in the geometric theory.

\end{itemize}

\medskip

Coming back to the present paper, Section \ref{SEC:coeff} is devoted to the definition of the functors $\coeff_{G,\ext}: \Dmod(\Bun_G) \to \Wh(G,\ext)$ and $\coeff_{{G,P}}: \Dmod(\Bun_G) \to \Wh(G,P)$ fitting in a commutative diagram
\begin{gather}   \label{triangle:geom}
\xy
(30,0)*+{ \Dmod(\Bun_G).}="basso";
(60,25)*+{  \Wh(G,P)  }="dx";
(0,25)*+{ \Wh(G,\ext) }="sx";
{\ar@{->}^{ \vrho_{G,\ext \to P} } "sx";"dx"};
{\ar@{->}^{ \coeff_{G,\ext}} "basso";"sx"};
{\ar@{->}_{ \coeff_{G,P} } "basso";"dx"};
\endxy
\end{gather}	
Finally, our goal in Section \ref{SEC:MAIN} is to prove the following main result:
\begin{thm} [Main Theorem] \label{thm:main_Gr}
The functor $\coeff_{G,\ext}: \Dmod(\Bun_G) \to \Wh(G,\ext)$ is fully faithful for $G= GL_n$ and $G=PGL_n$.
\end{thm}
In view of the argument of Section \ref{sssec:sketch}, it is perhaps surprising that the proof of this theorem has nothing to do with the gluing equivalence $\Wh(G,\ext) \simeq \on{laxlim}\,\Wh(G,P)$. Moreover, the proof can be run in greater generality: namely, the input needed to define a Whittaker-like category is a \emph{naive-unital category $\CRan$ over the Ran space of $X$, equipped with a naive-unital action of the geometrization of $T(K) \ltimes \NA$}. In such generality, the source of the coefficient functor is the invariant category $(\CRan^{B(\sK)})_\indep$ (cf. Remark \ref{rem:source-coeff-operator}).\footnote{The relevant definitions will be supplied in Sections \ref{SEC:background} and \ref{SEC:unital}.}
In the standard case discussed so far, $\CRan=\Dmod(\GrGRan)$ is the category over $\Ran = \Ran(X)$ attached to the Beilinson-Drinfeld Grassmannian.

\medskip

What we really prove in Section \ref{SEC:MAIN} (see Theorem \ref{thm:MAIN}) is a version of Theorem \ref{thm:main_Gr} for $\CRan$ equipped with a (naive-unital) action of the geometrization of $Q(\AA)$.

\medskip

The following conjecture, whose classical analogue is false in general, was proposed in \cite{Outline}:

\begin{conj} \label{conj:generalization}
For any connected reductive group $G$ with connected center, the functor $\coeff_{G,\ext}: \Dmod(\Bun_G) \to \Wh(G,\ext)$ is fully faithful.
\end{conj}

\ssec{Implications for the geometric Langlands program} \label{sssec:relevance-for-GLC}

To conclude this introduction, let us explain the role our Theorem \ref{thm:MAIN} (or Theorem \ref{thm:main_Gr}) plays in the geometric Langlands program. All the material in this section is taken from \cite{AG}, \cite{AG2} and \cite{Outline}, to which we refer for any undefined symbol.

\sssec{}

Let $\Gch$ be the Langlands dual of $G$, and let $\LocSys_{\Gch} := \Maps(X_\dR, B\Gch)$ denote the derived stack of $\Gch$-local systems on $X$. Since this stack is quasi-smooth, to any coherent sheaf one can associate its \emph{singular support}, which is a closed substack of
$$
\on{Arth} : S \mapsto \{
(\E, A) \, \big| \, 
\E \in \LocSys_{\Gch}(S), \, A \in H^0(S \times X_\dR, \check{\mathfrak{g}}_\E )\}.
$$
Of particular interest is the \emph{global nilpotent cone} $\N \subset \on{Arth}$, i.e. the locus cut out by the requirement that $A$ be nilpotent. Denote by $\IndCoh_{\N} (\LocSys_{\Gch})$ the ind-completion  of the category of coherent sheaves with singular support contained in $\N$.

\sssec{}

According to \cite{AG}, the \emph{geometric Langlands conjecture} calls for an equivalence
$$
\mathbb{L_G}: \IndCoh_{\N}(\LocSys_{\Gch}) \xto{ \;\; \; \simeq \;\;\;} \Dmod(\Bun_G),
$$
which is supposed to satisfy various \emph{desiderata}, one of which, the \emph{compatibility with the extended Whittaker model}, we are about to spell out. 

\medskip

To start with, recall from the above discussion that we have the extended Whittaker category $\Wh(G,\ext)$, glued out of the $2^r$ categories $\Wh(G,P)$ and equipped with a functor, $\coeff_{G,\ext}$, from $\Dmod(\Bun_G)$. 

\medskip

A similar picture holds on the Langlands dual side, where $\IndCoh_{\N}(\LocSys_{\Gch})$ embeds \emph{fully faithfully} inside a certain DG category $\Glue(\Gch)$.
See \cite{AG2} for the relevant constructions and for the proof of fully faithfulness.
Here, we only point out that $\Glue(\Gch)$ is a category expressed as a lax limit of $2^r$ categories $\ms Q(\Gch, \check P)$, parametrized by the poset $\Par(G)$. For instance, $\ms Q(\Gch, \Gch) := \QCoh(\LocSys_{\Gch})$; we omit the definition for general $\check P$.
Denote by 
$$
\ms{coeff}_{\Gch, \ext}^{\spec}: \IndCoh_\N(\LocSys_{\Gch}) \hto \Glue(\Gch)
$$
the inclusion.

\sssec{}

For any $P$, the categories $\Wh(G,P)$ and $\ms Q(\Gch, \check P)$ are related. In fact, there are \emph{fully faithful} functors 
$$
\LL_{\Wh,P} : \ms Q(\Gch, \check P) \hto \Wh(G,P),
$$
which are compatible with gluing and therefore yield a fully faithful functor $\LL_{\Wh, \ext} : \Glue(\Gch) \hto \Wh(G,\ext)$.

\nc{\RRep}{\mathfrak{R}ep}
\medskip

For instance, let us outline the ingredients concurring in the construction of $\LL_{\Wh,G}$. 
The geometric Casselman-Shalika formula 
implies that 
$$
\Wh(G,G) 
\simeq
\QCoh(\LocSys_{\Gch/[\Gch,\Gch]}) \usotimes{\RRep({\Gch/[\Gch,\Gch]})_\indep} \RRep(\Gch)_\indep,
$$
where $\RRep(H)_\indep$ (see \cite{Outline}) is the independent category of a unital sheaf of categories associated to the ordinary $\Rep(H)$.
Now, Proposition 4.3.4 of \loccit $\,$ asserts the existence of inclusions $i_H: \QCoh(\LocSys_{H}) \hto \RRep(H)_\indep$, functorial with respect to the affine algebraic group $H$.
The inclusion $\LL_{\Wh,G}: \QCoh(\LocSys_{\Gch}) \hto \Wh(G,G)$ is the one induced by $i_{\Gch}$ and $i_{\Gch/[\Gch,\Gch]}$.

\medskip

For more information about the functors $\LL_{\Wh,P}$, the reader may consult \cite{Outline}, \cite{R}, and references therein.

\sssec{} \label{sssec:L_G}

The compatibility of the Langlands conjecture with the extended Whittaker model requires that the diagram
\begin{gather} 
\xy
(55,0)*+{ \Wh(G,\ext) }="00";
(0,0)*+{ \Glue(\Gch )}="10";
(0,24)*+{ \IndCoh_\N(\LocSys_\Gch) }="11";
(55,24)*+{ \Dmod(\Bun_G) }="01";
{\ar@{<-}_{ \LL_{\Wh,\ext} } "00";"10"};
{\ar@{<--}_{ \LL_G \, ? } "01";"11"};
{\ar@{->}^{ \coeff^{\spec}_{\Gch,\ext}  } "11";"10"};
{\ar@{->}^{ \coeff_{G,\ext}   } "01";"00"};
\endxy
\end{gather}	
be commutative.
Since the arrows $\coeff^{\spec}_{\Gch,\ext}$ and $\LL_{\Wh,\ext}$ are fully faithful, the validity of Conjecture \ref{conj:generalization} would imply that \emph{$\LL_G$ is unique, if it exists, and automatically fully faithful}. Our main theorem shows this is the case for $G=GL_n$ and $G=PGL_n$.

\ssec*{Acknowledgements}

As explained in the introduction, this paper realizes a part of the program indicated in \cite{Outline}. It is a pleasure to thank Dennis Gaitsgory for several crucial discussions: for instance, the usage of the blow-up in Section \ref{SEC:MAIN} was inspired by his idea of the proof of Theorem \ref{thm:main_Gr} for $GL_3$. 
%

I am indebted to Sam Raskin for his help regarding unital structures on sheaves of categories over the Ran space. 
I am also grateful to Dima Arinkin, Jonathan Barlev, Ian Grojnowski and Constantin Teleman.

The constructions and the results of the present paper were announced during the workshop \virg{Towards the proof of the geometric Langlands conjecture} held at IIAS, Jerusalem, in March 2014. I wish to thank the organizers and the participants of the workshop for their inspiring interest.

\sec{Notation and background}\label{SEC:background}

\ssec{Prestacks}

Let $\Sch^\aff$ (resp., $\Sch^{\aff,\fty}$) denote the $1$-category of affine $\kk$-schemes (resp., affine $\kk$-schemes of finite type) and $\Grpd$ the $\infty$-category of $1$-groupoids.

\sssec{}

In this paper, by the term \virg{prestack}, we mean what is sometimes called \emph{classical $1$-prestack}, that is, an arbitrary functor $(\Sch^{\aff})^\op \to \mathsf{Grpd}$. The $1$-category $\PreStk$ of prestacks admits limits and colimits, computed pointwise. The term \virg{space} is a synonym for \virg{prestack}: prestacks are just the spaces that we allow in algebraic geometry. For instance, they are the spaces on which we are able to define quasi-coherent sheaves.

When defining a prestack, we simply write it as an assignment $\Y : S \mapsto \Y(S)$, with $S$ denoting an arbitrary test affine scheme. It should then be obvious how $\Y$ operates on arrows in $\Sch^\aff$.
When the curve $X$ is involved, we usually write $X_S$ for $S \times X$.

\medskip

Restriction along the inclusion $\Sch^{\aff,\fty} \hto \Sch^\aff$ yields a functor 
$$
\PreStk \longto \PreStk^{\fty}:= \Fun \bigt{ (\Sch^{\aff,\fty})^\op, \Grpd}.
$$
This functor is a colocalization, whence we regard $\PreStk^{\fty}$ as faithfully embedded inside $\PreStk$ (by left Kan extension).

\sssec{The Ran space}

The presence of the ad{\`{e}}les in the function theory version of the Langlands program is reflected in geometry by  the appearance of the Ran space of $X$, that is, the moduli space of finite non-empty subsets of $X$. To give a formal definition of $\Ran := \Ran_X$, let $\fset$ be the $1$-category of non-empty finite sets and surjections among them. Then, $\Ran := \colim_{(\fset)^\op} X^I$, where the transition maps are the diagonal closed embeddings $\Delta_{\phi}: X^J \hto X^I$ for any surjection $\phi: I \twoheadrightarrow J$.

We usually denote points of $\Ran$ (that is, finite subsets of $X$) with the symbol $\x$. An $S$-point $\x$ of $\Ran$ determines an effective divisor on $X_S := S \times X$, finite and flat over $S$, which we denote by $\Gamma_\x$. Let $U_\x := X_S - \Gamma_\x$ be its open complement.

\medskip

Much of the geometry of this paper is \emph{relative to the Ran space}, which means that our prestacks and group prestacks fiber over $\Ran$.
The \virg{correct} notation for such an object would be $\Y_{\Ran} \to \Ran$, say; however we often omit the subscript \virg{$\Ran$} when there is no risk of confusion. This will be the case, e.g., for (meromorphic) jets, defined below.

However, there \emph{is} a risk of confusion with spaces parametrizing generically defined geometric structures: for instance, the affine Grassmannian and the space of rational sections of a prestack $\Y$ over $X$. In fact, there are two ways to parametrize families of open sets of $X$: the \virg{$\Ran$ version}, which uses $U_\x$ for some $\x: S \to \Ran$, and the \virg{independent} version (cf. \cite{Ba}, or Section \ref{sssec:domains}), which uses the notion of \virg{domain}. 

As it turns out, we will need both approaches. To distinguish them notationally, we use no decoration for the Beilinson-Drinfeld Grassmannian over $\Ran$, i.e. the prestack
\begin{equation} \label{defn:Grass}
\GrGRan(S) :=
 \left\{ 
(\x, \P_G, \beta) \left| 
\begin{array}{l}
 \x \in \Ran(S), \, \P_G \mbox{ a $G$-bundle on $X_S$,} \\
\beta \mbox{ a trivialization of $\P_G$ on $U_\x$}
\end{array}
 \right.
\right\},
\end{equation}
whereas $\GrGind$ is its independent version, defined in Section \ref{gen-red-G-bundles}. See Section \ref{sssec:GSECT-indep} for parallel conventions regarding spaces of rational sections.
We also let ${\Gr}_{G,X^I} := \GrGRan \times_{\Ran} X^I$: this is defined as in (\ref{defn:Grass}) with the only difference that $\x \in X^I(S)$.
In general, for $\Y$ a prestack over $\Ran$, we set 
$$
\Y_{X^I} := X^I \ustimes{\Ran} \Y,
$$
so that $\Y\simeq \uscolim{\fset^\op}\, \Y_{X^I}$.

\begin{notat} \label{notation}
Let $\Y \to \Ran$ be a prestack over $\Ran$. We say that $\Y \to \Ran$ is a relative indscheme (resp., group indscheme, vector bundle, and so on) if $\Y_{X^I} \to X^I$ is an indscheme (resp., group indscheme, vector bundle, and so on) over $X^I$ for any $I \in \fset$.

In particular, when $\Y \to \Ran$ is a relative indscheme, $\Y$ is a \emph{pseudo-indscheme}, in the terminology of \cite{GL:contr}, expressed as the colimit of $\Y_{X^I}$ along closed embeddings.  

When dealing with pseudo-indschemes, we use the definition in \cite{R}, slightly different from the one in \loccit: for us, pseudo-indschemes are always equipped with a presentation and maps between them are supposed to be compatible with the presentations.
\end{notat}

\ssec{DG categories and sheaves thereof}

\sssec{}

We will need to consider differential graded (DG) categories of sheaves (mostly, $\fD$-modules) on our prestacks. By default, the term \virg{category} means \virg{co-complete DG category}. We denote by $\DGCat$ the $\infty$-category of such categories, with continuous (i.e., commuting with all colimits) functors as morphisms.
This $\infty$-category is symmetric monoidal (hence, there is a notion of duality) and admits limits and colimits. 
See \cite{Lu:HA} for the relevant constructions and proofs, or \cite{GL:DG} for a review.

\begin{rem}
The functor $\QCoh$ is defined as a contravariant functor out of $\PreStk$, while $\fD$ (the functor of $\fD$-modules) only out of $\PreStk^{\fty}$.
\end{rem}

\sssec{}

In particular, consider the category of $\fD$-modules on the Ran space: $\R:= \Dmod(\Ran) \simeq \lim_{\Delta^!} \Dmod(X^I)$, which is monoidal with respect to the usual $\otimes$. We shall consider the $\infty$-category of DG categories tensored over ($\R, \otimes)$, denoted by $\R\mmod$.
\footnote{Notation: given $\S$ a monoidal $\infty$-category and $A \in \ms{Alg}(\S)$, we denote by $A \mod(\S)$ the $\infty$-category of $A$-modules. We also write $A \mmod := A \mod(\DGCat)$, when $A$ is a monoidal DG category. Similarly for coalgebras and comodules.}

For a prestack $\Y$ of finite type over $\Ran$, the category $\Dmod(\Y)$ is naturally an object of $\R\mmod$. 
Actually, more is true: $\Dmod(\Y)$ is the category of global sections of a sheaf of categories over $\Ran_\dR$ (Lemma \ref{lem:LocRan-fullyfaith}).

\sssec{}

Let us first fix some notation regarding sheaves of categories over a prestack, following \cite{GL:shvcat}. Set $\ShvCatCrys(\X) := \ShvCat(\X_\dR)$, for $\X$ a prestack of finite type. This $\infty$-category is symmetric monoidal; we indicate by $\Otimes$ the tensor product.
Abusing notation, we indicate by $\Otimes := \otimes_{\Dmod(\X)}$ the tensor product in $\Dmod(\X) \mmod$ as well.
There is an adjunction
$$
\bLoc^\nabla_{\X} : \Dmod(\X) \mmod 
\rightleftarrows
\ShvCatCrys(\X) : \bGamma^\nabla(\X, - ),
$$
with $\bLoc^\nabla_{\X}$ symmetric monoidal. We say that $\X_\dR$ is $1$-affine if the above functors are equivalences. E.g., $Y_\dR$ is $1$-affine whenever $Y$ is a scheme of finite type (\loccit).

\begin{rem}
Consequently, 
$$
\ShvCatCrys(\Ran) \simeq \lim_{I \in\fset} \Dmod(X^I) \mmod.
$$ 
Informally, an object of $\ShvCatCrys(\Ran)$ is a collection $\{\M_I\}_{I \in \fset}$, with $\M_I \in \Dmod(X^I) \mmod$, together with compatible equivalences $\psi_{I \twoheadrightarrow  J}: \M_I \otimes_{\Dmod(X^I)} \Dmod(X^J) \to \M_J$, for any surjection $I \twoheadrightarrow J$.
By adjunction, we obtain a diagram of functors $\M_I \to \M_J$ of plain DG categories and $\bGamma(\{\M_I\}_{I \in \fset}):= \lim_I \M_I$.
The monoidal product in $\ShvCatCrys(\Ran)$ is computed \virg{pointwise} : if $\Cshv = \{\CC_{X^I} \}_I$ and $\Cshv' =\{ \CC'_{X^I} \}_I$, then 
$$
\Cshv \Otimes \Cshv' 
\simeq
 \big\{ \CC_{X^I} \otimes_{\Dmod(X^I)} \CC'_{X^I} \big\}_{I}.
$$
\end{rem}

\begin{lem} \label{lem:LocRan-fullyfaith}
The functor $\bLoc^\nabla_{\Ran}$ is fully faithful. In particular, for $\Y \to \Ran$ a prestack of finite type, we see that 
$\Dmod(\Y) \simeq \bGamma^\nabla(\Ran, \bold\Dmod(\Y))$, where 
$$
\bold\Dmod(\Y) := \bLoc^\nabla_\Ran(\Dmod(\Y)) \simeq \{\Dmod(\Y_{X^I})\}_{I \in \fset}.
$$
\end{lem}

\begin{proof}
For any $I \surjto J$, the pushforward $\Delta_*: \Dmod(X^J) \to \Dmod(X^I)$ is $\R$-linear, thanks to the projection formula. Using this, it is easy to check that the unit of the adjunction $\id \to \bGamma^\nabla (\Ran, \bLoc^\nabla_\Ran(-))$ is an isomorphism.
\end{proof}

\ssec{Contractibility, generic maps, (meromorphic) jets}

We say that a prestack $\Y \in \PreStk^{\fty}$ is \emph{(homologically) contractible} if the $\fD$-module pullback along the tautological map $p: \Y \to \pt$ is fully faithful. An important feature of $\Ran$ is its contractibility, see \cite{BD:chiral}.

\sssec{Some notation}

For any map of schemes $\x =\{x_i\}_{i \in I}: S \to X^I$, recall the divisor $\Gamma_\x \subset X_S$, that is, the pull-back of the incidence divisor $\Gamma_{\inc} \subset X^I \times X$ along $\x \times \id$. Denote by $\hat \D_\x$ the formal completion of $X_S$ along $\Gamma_\x$: this is a formal scheme with a well-defined associated scheme $\D_\x$. Denote by $\D_\x^\circ$ the \emph{punctured disc} (more precisely: \emph{punctured tubular neighbourhood} of $\x$), i.e. the scheme $\D_\x - \Gamma_\x$.

\sssec{} 

For a prestack $\Y$ fibering over $X$, let us recall the definition of the space of \emph{rational sections} of $\Y \to X$:
$$
\GSectRan{\Y} :
S \mapsto
 \left\{ 
(\x, \sigma) \left| 
\begin{array}{l}
 \x \in \Ran(S), \\
 \sigma \mbox{ a section of $\restr {S \times \Y}{U_\x} \to U_\x$}
\end{array}
 \right.
\right\}.
$$
We also write $\GMapsRan{\Z}$ for the space of rational sections of $\Z \times X \to X$, and occasionally abbreviate it by $\Z(\sK)$.
\begin{thm}[\cite{GL:contr}] \label{thm:contr:Ga}
The prestack $\GMapsRan{H}$ is contractible for any algebraic group $H$. More generally, $\GMapsRan{Y}$ is contractible for any affine scheme $Y$ which is covered by open subsets isomorphic to open subsets of $\mathbb A^n$.
\end{thm}

Let now $Y \to X$ be a Zariski locally trivial fibration whose fibers satisfy the above hypothesis. Using the methods of \loccit, one deduces that $\GSectRan{Y}$ is also contractible.

\sssec{}

The groups $G(\AA)$ and $G(\OO)$ appearing in the introduction will be replaced in geometry by the group prestacks of \emph{meromorphic jets} and \emph{regular jets} into $G$, respectively. 
Let us briefly recall the well-known definitions. For $\Y \to X$ an arbitrary prestack over $X$, define the spaces of jets and meromorphic jets of $\Y \to X$ as the prestacks
$$
\Jets(\Y/X): S \mapsto \left\{ 
(\x, f) \left| 
\begin{array}{l}
 \x \in \Ran(S), \\
 \mbox{$f$ a section of $\restr{\Y}{\D_x} \to \D_\x$}
\end{array}
 \right.
\right\};
$$
$$
\Jets^\mer(\Y/X): S \mapsto  \left\{ 
(\x, f) \left| 
\begin{array}{l}
 \x \in \Ran(S), \\
 \mbox{$f$ a section of $\restr{\Y}{\D^\circ_x} \to \D^\circ_\x$}
\end{array}
 \right.
\right\}.
$$
%
For $H$ an affine algebraic group, we put $H(\sA) := \Jets^\mer(H\times X/X)$ and $H(\sO):= \Jets(H\times X /X)$. 
We will need to consider $\fD$-modules on $H(\sA)$ and $H(\sO)$. These prestacks are of infinite type and therefore they do not fit in the context for $\fD$-modules presented so far. One needs the theory of $\fD$-modules on schemes and indschemes of infinite type (\cite{R}, \cite{Be}). The following restricted context suffices for us.

\ssec{$\fD$-modules on indschemes of pro-finite type}

\sssec{}

The $1$-category $\IndSch^\pro_{/B}$ of \emph{indschemes of pro-finite type} over $B \in \Sch^{\fty}$ is constructed out of $\Sch^{\fty}_{/B}$ in two steps. 
First, let $\Sch_{/B}^\pro$ be the full subcategory of $\Pro (\Sch_{/B}^\fty)$ consisting of those schemes that are projective limits of finite type schemes along affine smooth surjective maps. As it turns out, the tautological map $\Sch_{/B}^\pro \to \Sch_{/B}^{\mathsf{qc,qs}}$ is fully faithful.
By definition, a map $Y \to Z$ in $\Sch_{/B}^\pro$ is finitely presented if it is base-changed from a map in $\Sch_{/B}^\fty$.

The second and last step to construct $\IndSch^\pro_{/B} \hto \PreStk_{/B}$ consists of ind-extending $\Sch^\pro_{/B}$ along \emph{finitely presented} closed embeddings.\footnote{This definition is made to exclude, say, the closed embedding $\iota: \Spec(k) \to \Spec ({k[x_1, x_2, \ldots]})$. In fact $\iota$ (and more generally inclusions of infinite codimension) exhibits some pathological functorial behaviour.}

\sssec{}

As in \cite{Be}, we define 
$$
\Dmod^!: \big( \IndSch^\pro_{/B} \big)^\op \to \DGCat, 
\hspace{.5cm}
\Dmod^*: \IndSch^\pro_{/B} \to \DGCat.
$$
For any $\Y \in \IndSch^{\pro}_{/B}$, the categories $\Dmod^!(\Y)$ and $\Dmod^*(\Y)$ are in duality.

From the finite-type case, one checks that $\Dmod^!$ is symmetric monoidal; in particular, pull-back along the structure map $p:\Y \to B$ gives $\Dmod^!(\Y)$ the structure of $\Dmod(B)$-module. The dual action also equips $\Dmod^*(\Y)$ with the structure of $\Dmod(B)$-module. 
(In the notation of \textit{loc.cit.}, this action  was denoted by $\overset{!*}{\otimes}$.)
Hence, both $\Dmod^!$ and $\Dmod^*$ upgrade to symmetric monoidal functors
$$
\Dmod^!: \big( \IndSch^\pro_{/B} \big)^\op \to \DGCat_{/B},
\hspace*{.5cm}
\Dmod^*:  \IndSch^\pro_{/B}  \to \DGCat_{/B}.
$$
Furthermore, $\Dmod^*(\Y)$ and $\Dmod^!(Y)$ are mutually dual also as objects of $\DGCat_{/B}$, the evaluation being given by
$$
\Dmod^!(\Y) \usotimes{\Dmod(B)} \Dmod^*(\Y) \to \Dmod(B),
\hspace{.4cm}
[M,N] \mapsto p_*(M \overset{!*}\otimes N).
$$

\sssec{}

Let us now redo the above construction over $\Ran$. Let $\IndSch^\pro_{/\Ran}$ be the $1$-category of relative indschemes of pro-finite type, see \ref{notation}.

\medskip

For $\Y \to \Ran$ a relative indscheme of pro-finite type, we define 
$$
\Dmod^!(\Y) := \lim_{I \in \fset, \wt\Delta^!} \Dmod^!(\Y_{X^I}), 
\hspace*{.4cm}
\Dmod^*(\Y) := \uscolim{I \in \fset^\op, \wt\Delta_*} \Dmod^*(\Y_{X^I}),
$$
where $\wt\Delta: \Y_{X^J} \to \Y_{X^I}$ is the closed embedding induced by $\Delta: X^J \hto X^I$ via pullback.

\medskip

It is clear that $\Dmod^!(\Y)$ and $\Dmod^*(\Y)$ are the global sections of their localizations over $\Ran$, denoted $\bold{\Dmod}^!(\Y)$ and $\bold{\Dmod}^*(\Y)$.

\sssec{}

For any affine algebraic group $H$, the group prestack $\Ha$ is a relative indscheme of pro-finite type over $\Ran$, while $H(\sO)$ is a relative group scheme of pro-finite type. When $H$ is unipotent, for instance $H=N$, the structure is richer.

\begin{lem}
For $H$ unipotent, the prestack $H(\sA)_{X^I}$ is an $\aleph_0$ ind-object (with our definition of closed embeddings!) in $\Groups (\Sch_{/X^I}^\pro)$. 
\end{lem}

\begin{proof}
It suffices to prove the result for $H = N(GL_n)$, the unipotent radical of $GL_n$. More generally, assume that $H$ is the unipotent radical of a reductive group $G$ and consider the map
$$
\mathsf{Ad} : T(\sA)_{X^I} \ustimes{X^I} N(\sO)_{X^I} \to N(\sA)_{X^I}
$$
given by conjugation.
Each  $T$-cocharacter $\mu$ naturally yields section $\sigma_\mu$ of $T(\sA)_{X^I} \to X^I$.
Denote by $N(\sA)_{X^I}^{\mu}$ the image of the map
$$
N(\sO)_{X^I} 
\xto{ \, \sigma_\mu \times \id_{N(\sO)} \, \,  }
T(\sA)_{X^I} \ustimes{X^I} N(\sO)_{X^I} 
\xto{\ms{Ad} \,}
 N(\sA)_{X^I}.
$$
Each $N(\sA)_{X^I}^{\mu}$ is a group subscheme of $N(\sA)_{X^I}$. It is easy to verify that, for $\mu$ dominant regular, $N(\sA)_{X^I}$ is the increasing union of $N(\sA)_{X^I}^{m \cdot \mu}$ as $m \to + \infty$.
\end{proof}

\begin{rem} \label{rem:NA-sequence}
Fix an ind-group-scheme presentation $N(\sA)_{X^I} \simeq \colim_{n \in \NN} N(\sA)_{n,X^I}$. Tautologically, $\Dmod^!(N(\sA)_{X^I}) \simeq \lim_n\Dmod^!(N(\sA)_{n,X^I})$, so that the datum of an action of $N(\sA)_{X^I}$ on $\CC \in \DGCat_{/X^I}$ is the same as a sequence of compatible $N(\sA)_{n,X^I}$-actions.
\end{rem}

\ssec{Categorical representation theory}

Let $B \in \Sch^{\fty}$ or $B= \Ran$. Since the functor 
$$
\Dmod^!: \IndSch^\pro_{/B} \to \Dmod(B) \mmod
$$
is symmetric monoidal, any group object $(H,m) \in \IndSch^\pro_{/B}$ gives rise to the Hopf algebra $(\Dmod^!(H), m^!, \Delta^!)$ in $\Dmod(B) \mmod$. 
As in \cite{Be}, we say that $H$ \emph{acts} on $\CC \in \Dmod(B) \mmod$ if $\CC$ is given the structure of a comodule over $(\Dmod^!(H),m^!)$; equivalently: a module over $(\Dmod^*(H), m_*)$.

\medskip

The object $\CC^H \in \Dmod(B) \mmod$, defined by the totalization of the usual cobar construction
$$
\CC^H := \lim \Bigt{ \CC \rr \CC \usotimes{\Dmod(B)} \Dmod^!(H) \rrr \cdots },
$$
is called the \emph{invariant category}; $\oblv^H: \CC^H \to \CC$ denotes the tautological conservative functor.
When $H$ belongs to $\Sch^\pro_H$, this functor admits a continuous right adjoint, $\Av_*^H$. If $H$ is also cohomologically contractible, $\Av_*^H$ is a colocalization.
In contrast, if $H$ is a group indscheme, then $\oblv^H$ typically does not admit a continuous right adjoint.

\sssec{} 

Let now $\G \to \Ran$ be a relative group indscheme. We say that $\G$ acts on $\Cshv \in \ShvCatCrys(\Ran)$ if $\Cshv$ is endowed with a coaction of 
$$
\bold{\Dmod}^!(\G) \in \CoAlg(\DGCatRan) \simeq \lim_{I \in \fset} (\ms{CoAlg}(\Dmod(X^I) \mmod)).
$$
We define $\Cshv^\G \in \DGCatRan$ formally as before.

\begin{lem} 
If $\G$ acts on $\CRan \in \R\mmod$, then $\G$ acts on $\Cshv:= \bLoc^\nabla_\Ran(\CRan)$ and there is a canonical equivalence $\bGamma^\nabla(\Ran,\Cshv^\G) \simeq \CRan^\G$.
\end{lem}

\begin{proof}
Since $\bLoc^\nabla_\X$ is symmetric monoidal in general, it is immediate to deduce that $\Cshv$ inherits a $\G$-action. We now write
$$
(\CRan)^{\G} 
\simeq
(\uscolim{I\in\fset^\op} \Dmod(X^I)) \usotimes{\Dmod(\Ran)} (\CRan)^{\G}  \simeq
\lim_{I \in \fset} \bigt{ \Dmod(X^I) \usotimes{\Dmod(\Ran)} (\CRan)^{\G} }
\simeq
\lim_{I \in \fset} (\Cmod{X^I})^{\G_{X^I}},
$$
where the last step is a consequence of the dualizability of $\Dmod(X^I)$ as an object of $(\R\mmod, \Otimes)$.
Finally, note that the RHS equals $\bGamma^\nabla(\Ran,\Cshv^\G)$ by definition. 
\end{proof}

\sssec{} \label{sssec:invariants-NA}
\renc{\NA}{{N(\sA)}}
\nc{\Na}{N(\sA)}
\nc{\Nk}{N(\sK)}
\nc{\Tk}{T(\sK)}

In the two lemmas below, we present two cases in which the functor $\oblv^{\G}$ is fully faithful despite lacking a continuous right adjoint.

\begin{lem}
Let $\CRan \in \R\mmod$ be equipped with an action of $\NA$, where $N$ is a unipotent group. Then $\oblv^{\NA}: (\CRan)^{\NA} \to \CRan$ is fully faithful.
\end{lem}

\begin{proof}
Thanks to Lemma \ref{lem:LocRan-fullyfaith} and the lemma above, it suffices to prove the following claim: for $\Cmod{X^I} \in \Dmod(X^I) \mmod$ acted on by $\NA_{X^I}$, the forgetful functor $\Cmod{X^I}^{\NA_{X^I}} \to \Cmod{X^I}$ is fully faithful.

Writing $\NA_{X^I} = \colim_{n \in \NN} \NA_n$ as in Remark \ref{rem:NA-sequence}, we see that $(\Cmod{X^I})^{\NA_{X^I}}$ is the intersection of the subcategories $(\Cmod{X^I})^{\NA_{n,X^I}} \subseteq \Cmod{X^I}$.
\end{proof}

\begin{lem}
Let $\CRan$ be a dualizable object of $\Dmod(\Ran) \mmod$, equipped with an action of a contractible relative group indscheme $\G$ (for instance, $\G:= H(\sK)_\Ran$ with $H$ an algebraic group). Then the forgetful functor $\oblv^{\G}: (\CRan)^{\G} \to \CRan$ is fully faithful. 
\end{lem}

\begin{proof}
Writing $\CRan$ as the limit of the split cosimplicial category $\CRan \Otimes \Dmod(\G) \rr \CRan \Otimes \Dmod(\G)^{\otimes 2} \cdots$, the functor $\oblv^\G$ arises as a limit of functors $\phi_n: \CRan \Otimes \Dmod(\G)^{n} \to \CRan \Otimes \Dmod(\G)^{n+1}$, each of which is obtained from $p^!: \Dmod(\Ran) \to \Dmod(\G)$ by tensoring up with a dualizable object in $\Dmod(\Ran)\mmod$. By the contractibility of $\G$, each $\phi_n$ is fully faithful, whence so is $\oblv^\G$.
\end{proof}

\sssec{Actions by semi-direct products} \label{semi-direct}

Let $B$ a scheme of finite type or $B=\Ran$. Assume $\G$ and $\H$ are group inschemes over $B$ with $\G$ acting on $\H$.
If $\G \ltimes \H$ acts on $\Cmod{B}$, then the $\oblv^\H: (\Cmod{B})^\H \to \Cmod{B}$ is $\G$-equivariant. By commuting limits with limits, we see that $((\Cmod{B})^\H)^\G$ is the fiber product $(\Cmod{B})^\G \times_{\Cmod{B}} (\Cmod{B})^\H$. Also, the natural functor from $(\Cmod{B})^{\G \ltimes \H}$ to such fiber product is an equivalence.

\ssec{Sheaves of categories over quotient prestacks}

To prove that the two definitions of the Whittaker category $\Wh(G,P)$ (cf. Section \ref{sssec:degenerate-Whittaker-functions} for the function-theoretic analogy) are equivalent, we need to discuss crystals of categories over quotients of a prestack by a group prestack. 
To guarantee the expected result (e.g. Proposition \ref{prop:stabilizer}), one has to impose a technical condition on the prestacks involved. This technical condition is satisfied in all examples considered in the sequel.

\sssec{}

Let $(\G,m)$ a group object in $\PreStk^{\fty}$, which can be written as a colimit of $1$-affine schemes along proper transitions maps. We have the \emph{monoidal} equivalence 
$$
(\ShvCatCrys(\pt/\G), \Otimes) \simeq \bigt{ (\Dmod(\G),m^!) \ccomod, \otimes}.
$$
Under this equivalence, the invariant category $\Ccat^\G$ corresponds to the value of $\bGamma^\nabla(\pt/\G, -)$ on the sheaf associated to $\Ccat$.

\medskip

Let now $\Y \in \PreStk^{\fty}$, also assumed to be a colimit of schemes along proper maps. Assume that $\G$ acts on $\Y$ and denote by $f: (\Y/\G)_\dR \to (\pt/\G)_\dR$ the obvious map.
Consider the usual adjunction
$$ 
\cores_f: \ShvCatCrys(\pt/\G) \rightleftarrows \ShvCatCrys(\Y/\G): \coind_f.
$$
Now, the composition $\coind_f \circ \cores_f$, considered as an endofunctor of $\Dmod(\G) \ccomod$, sends $\Ccat \squigto \Ccat \otimes \Dmod(\Y)$, where $\Ccat \otimes \Dmod(\Y)$ is equipped with the diagonal $\G$-action.

\begin{prop} \label{prop:stabilizer}
With the notation as above, let $\H \hto \G$ a subgroup and $\Z \hto \Y$ be an $\H$-invariant closed subscheme such that $i: (\Z/\H)_\dR \to (\Y/\G)_\dR$ is an isomorphism. Then, for $\CC$ a category acted on by $\G$, pullback along $\Z \to \Y$ yields an equivalence
$$
\bigt{ \Ccat \otimes \Dmod(\Y) }^\G
\xto{\;\; \simeq \;\;}
\bigt{ \Ccat \otimes \Dmod(\Z) }^\H.
$$
\end{prop}

\begin{proof}
The diagram
\begin{gather} \nonumber
\xy
(00,0)*+{ \ShvCatCrys(\pt/\H) }="00";
(40,0)*+{ \ShvCatCrys(\Z/\H)  }="10";
(0,20)*+{ \ShvCatCrys(\pt/\G)  }="01";
(40,20)*+{ \ShvCatCrys(\Y/\G)  }="11";
(80,00)*+{ \DGCat  }="2";
{\ar@{->}^{ } "00";"10"};
{\ar@{->}^{ \cores_f } "01";"11"};
{\ar@{->}_\simeq^{ \cores_{i}  } "11";"10"};
{\ar@{->}^{    } "01";"00"};
{\ar@{->}^{ \bGamma^\nabla   } "10";"2"};
{\ar@{->}^{ \bGamma^\nabla   } "11";"2"};
\endxy
\end{gather}	
is commutative, thanks to the fact that $i$ is an equivalence. Start with $\Ccat$ at the top-left corner and follow it along the two possible paths towards the botton-right corner, using the identity $\bGamma^\nabla(\Y/\G,-) \simeq \bGamma^\nabla(\pt/\G, \coind_f(-))$. 
\end{proof}

\sec{Naive-unital structures on categories tensored over $\Ran(X)$} \label{SEC:unital}

\nc{\co}{\mathsf{co}}
\renc{\comod}{\mathsf{\on- comod}}

Recall our notation: $\R := (\Dmod(\Ran), \otimes)$. Objects of the $\infty$-category $\R \mmod$ are called \virg{categories tensored over $\Ran$}. This definition has of course nothing to do with the semigroup structure of $\Ran$ (given by union of finite sets), reviewed below. \emph{Naive-unitality} is a piece of structure on an object $\CRan \in \R\mmod$ that takes the latter into account.\footnote{The term \virg{unital} has been reserved for sheaves of categories over $\Ran_\dR$, see \cite{GL:Ran} and \cite{R}. Here, we just consider their global sections, hence the choice of the term \virg{naive-unital}.}

We shall first describe unital spaces over $\Ran$. Categories of $\fD$-modules on such spaces inspire the definition of \emph{naive-unitality} for an object $\CRan \in \R\mmod$. This kind of structure is important because it allows to \virg{mod out the Ran space}: i.e., it allows to extract a subcategory, called $\Cind$, of the DG category underlying $\CRan$.

\ssec{Definition and basic examples}

\sssec{}

The operation $\add: \Ran \times \Ran \to \Ran$ of union of non-empty finite sets  makes $\Ran$ into a monoid in the category of prestacks. Let $\R^\co$ denote the comonoidal category $(\Dmod(\Ran), \add^!)$.

\begin{defn}[See \cite{Ba2}, \cite{R}]

We say that a prestack $\Z$ over $\Ran$ is \emph{unital} if it has been given the structure of an object of the slice category 
$$
\bigt{ (\Ran, \add)\mod(\PreStk)}_{/\Ran};
$$
in other words, $\Z$ is equipped with an action of $(\Ran,\add)$ making the structure map $\Z \to \Ran$ a morphism of $(\Ran,\add)$-modules.

\end{defn}

Let $\Y$ be an arbitrary prestack (resp., a prestack over $X$). It easy to see that $\Y(\sK):= \GMapsRan{\Y}$ (resp., $\GSectRan{\Y}$) and $\GrGRan$ all possess a unital structure. For instance, in the first case, the action map is
\begin{eqnarray}
\ms{act}: \Ran \times \GMapsRan{\Y} \to \GMapsRan{\Y}, \nonumber \\
(\x; \y, f: X - \y \to \Y) \mapsto (\x \cup \y, \restr{f}{X- (\x \cup \y)}).
\end{eqnarray}

\sssec{}

The category $(\Ran, \add)\mod(\PreStk)$ is naturally symmetric monoidal and $(\Ran, \Delta)$ is a coalgebra object therein. By pullback, $\R$ acquires the structure of a commutative algebra in the symmetric monoidal $\infty$-category $\R^\co \ccomod$. Hence, it makes sense to consider the $\infty$-category $\R \mod (\R^\co \ccomod)$.

\medskip

Let now $Y := \Y$ be a scheme of finite type. By pull-back, $\Dmod( \GMapsRan{Y})$, $\Dmod(\GSectRan{Y})$ and $\Dmod(\GrGRan)$ acquire the structure of objects of the latter $\infty$-category. These examples motivate the following definition.

\begin{defn}
The $\infty$-category
$$
\R\mmod^{\un} := \R \mod (\R^\co \ccomod)
$$
will be referred to as the $\infty$-category of \emph{naive-unital $\R$-modules}.

\end{defn}

\sssec{}

Since the forgetful functor $\R^\co \ccomod \to \DGCat$ is symmetric monoidal, we obtain a functor $\theta: \R \mod (\R^\co \ccomod) \to \R \mmod$. This is the functor that \virg{forgets the naive-unital structure}.
After the foundational work \cite{Lu:HA}, the following result is formal.

\begin{lem}  \label{lem:unital-formal}
$\R\mmod^{\un}$ admits:
\begin{itemize}
\item
all limits and colimits, compatible with the obvious forgetful functor $\theta: \R\mmod^{\un} \to \R\mmod$;
\item
a symmetric monoidal product with respect to which $\theta$ is monoidal.
\end{itemize}
\end{lem}

\ssec{Less basic examples}

\nc{\Corr}{\ms{Corr}}

Besides unital spaces, there are several other spaces $\Y \to \Ran$, e.g. $\Y = G(\sA)$, whose categories of $\fD$-modules naturally (but not obviously!) belong to $\R\mmod^\un$. The most interesting examples are obtained by considering indschemes with an action of $\Ran$ in the setting of correspondences.

\sssec{}

\nc{\PsIndSchPro}{\ms{PsIndSch}^{\pro}}

Recall $\PsIndSchPro$, the $1$-category of pseudo-indschemes of pro-finite type, see Notation \ref{notation}, and consider the symmetric monoidal $1$-category $\Corr := \Corr(\PsIndSchPro)$ defined as follows.

\medskip

Objects of $\Corr$ are the same as the ones of $\PsIndSchPro$; the set of morphisms between $\X$ and $\Y$ consists of diagrams $\X \leftarrow \H \to \Y$ in $\PsIndSchPro$, where $\H \to \X$ is arbitrary while $\H \to \Y$ is required to be finitely presented. 
The symmetric monoidal structure of $\Corr$ is inherited by the one of $\PsIndSchPro$, that is, the tautological functor $\PsIndSchPro \to \Corr$ is symmetric monoidal. It follows that the monoid object $(\Ran, \add) \in \PsIndSchPro$ is in particular a monoid object of $\Corr$.

\medskip

Thanks to \cite{GaiRos} and \cite{R}, the usual functor $\Dmod^!$ admits an enhancement to a \emph{symmetric monoidal} functor 
$\Dmod^!: (\Corr, \times) \to (\DGCat, \otimes)$, whose action on $1$-morphisms is
$$
\X \xleftarrow{\alpha} \H \xto{\beta} \Y
\squigto
\beta_+ \circ \alpha^! : \Dmod^!(\X) \to \Dmod^!(\Y).
$$
Here, $\beta_+: \Dmod^!(\H) \to \Dmod^!(\Y)$ is the functor that satisfies base change with respect to $!$-pullbacks. It is defined whenever $\beta$ is finitely presented.

\sssec{}
\nc{\RanCorr}{\ms{RanCorr}}

 Let us consider the $1$-category $\RanCorr$ defined as follows. Its objects are pseudo-indschemes of pro-finite type equipped with an action of $(\Ran, \add)$ in $\Corr$, informally presented as pairs $(\Y, \Ran \times \Y \leftarrow \H \to \Y)$, where $\Y, \H \in \IndSch^\pro$ and $\H \to \Y$ is of finite type. 
Morphims in $\RanCorr$ are maps of pseudo-indschemes commuting with the $\Ran$-action: in other words, morphisms between $(\Y, \Ran \times \Y \leftarrow \H \to \Y)$ and $(\Y', \Ran \times \Y' \leftarrow \H' \to \Y')$ consist of commutative diagrams
\begin{gather}  
\xy
(0,0)*+{ \Ran \times \Y'}="00";
(27,0)*+{ \H'}="10";
(54,0)*+{ \Y'}="20";
(0,12)*+{ \Ran \times \Y }="01";
(27,12)*+{ \H}="11";
(54,12)*+{ \Y}="21";
{\ar@{->}^{   } "10";"00"};
{\ar@{->}^{     } "11";"01"};
{\ar@{->}^{   } "10";"20"};
{\ar@{->}^{  } "11";"21"};
{\ar@{->} "11";"10"};
{\ar@{->} "01";"00"};
{\ar@{->} "21";"20"};
\endxy
\end{gather}	
compatible with further compositions.
\nc{\PseudoIndSch}{\mathsf{Pseudo}\on-\ms{IndSch}}

\begin{rem}
The objects of $\RanCorr$ and of $(\Ran, \add) \mod(\Corr)$ coincide. However, these two $1$-categories are not equivalent: morphisms are defined differently.
\end{rem}


\sssec{}

Consider now the slice category $\RanCorr_{/\Ran}$: informally, its objects are described by commutative diagrams 
\begin{gather}   \label{diag:object of RanCorroverRan}
\xy
(27,0)*+{ \Ran \times \Ran}="10";
(54,0)*+{ \Ran.}="20";
(0,12)*+{ \Ran \times \Y }="01";
(27,12)*+{ \H}="11";
(54,12)*+{ \Y}="21";
{\ar@{->}^{     } "11";"01"};
{\ar@{->}^{ \add  } "10";"20"};
{\ar@{->}^{  } "11";"21"};
{\ar@{->} "11";"10"};
{\ar@{->}_{\id \times p} "01";"10"};
{\ar@{->}^{p} "21";"20"};
\endxy
\end{gather}	
Note that $\RanCorr_{/\Ran}$ is symmetric monoidal, the product being $\times_\Ran$. More precisely, the product between the object represented by (\ref{diag:object of RanCorroverRan}) and its analogue decorated with primes is given by
\begin{gather} 
\xy
(40,0)*+{ \Ran \times \Ran}="10";
(80,0)*+{ \Ran.}="20";
(0,15)*+{ \Ran \times (\Y \ustimes{\Ran} \Y' ) }="01";
(40,15)*+{ \H \ustimes{\Ran \times \Ran} \H'}="11";
(80,15)*+{ \Y \ustimes{\Ran} \Y'}="21";
{\ar@{->}^{     } "11";"01"};
{\ar@{->}^{ \add  } "10";"20"};
{\ar@{->}^{  } "11";"21"};
{\ar@{->} "11";"10"};
{\ar@{->}_{} "01";"10"};
{\ar@{->}^{} "21";"20"};
\endxy
\end{gather}	

\begin{lem}
The functor $\Dmod^!: \RanCorr_{/\Ran} \to \DGCat$ lifts naturally to a symmetric monoidal functor $\Dmod^!: \RanCorr_{/\Ran} \to \R\mmod^\un$.
\end{lem} 
\begin{proof}
Immediate from the construction.
\end{proof}

\sssec{}

The above lemma allows to construct several objects in $\R\mmod^\un$. To exhibit our main example, let $H \in \Sch^{\aff,\fty}$ and consider the space 
$$
H(\sO,\sA) := H(\sO,\sA)_{\Ran \times \Ran}:
S \mapsto
 \left\{ 
(\x,\y, f) \left| 
\begin{array}{l}
 \x \in \Ran(S), \, \y \in \Ran(S) \\
 \mbox{$f: \D_{\x \cup \y} - \Gamma_\y \to H$}
\end{array}
 \right.
\right\}.
$$
We leave it to the reader to verify that the correspondence
$$
\Ran \times H(\sA) \longleftarrow
H(\sO,\sA) \longto
H(\sA)
$$
$$
(\x; \y, \restr{f}{\D^\circ_\y}) 
\mapsfrom
(\x, \y, f)
\mapsto
(\x\cup\y; \restr{f}{\D^\circ_{\x \cup \y}}),  
$$
makes $H(\sA)$ into an object of $\RanCorr_{/\Ran}$, thereby endowing $\Dmod^!(H(\sA))$ with the structure of a naive-unital $\R$-module. 

A completely parallel discussion holds for $H(\sO)$, where the intermediate object is $$
H(\sO,\sO) := H(\sO,\sO)_{\Ran \times \Ran}:
S \mapsto
 \left\{ 
(\x,\y, f) \left| 
\begin{array}{l}
 \x \in \Ran(S), \, \y \in \Ran(S) \\
 \mbox{$f: \D_{\x \cup \y} \to H$}
\end{array}
 \right.
\right\}.
$$

\sssec{}

Assume now that $H$ is an algebraic group. Then the group structure of $H(\sA)$ makes the latter into a group object of $\RanCorr_{/\Ran}$. Consequently,
$\Dmod^!(H(\sA))$ is a coalgebra object in the symmetric monoidal $\infty$-category $\R\mmod^\un$. 

\begin{rem} \label{rem:examples}
All in all, our familiar prestacks $G(\sA)$, $\Gr_G$, $T(\sK) \ltimes \NA$ are all naturally in $\RanCorr_{/\Ran}$. Moreover, it is straightforward to check that the usual maps between them, e.g. $G(\sO) \to G(\sA)$, $T(\sK) \ltimes N(\sA) \to G(\sA)$, $G(\sA) \to \Gr$ are morphisms in $\RanCorr_{/\Ran}$.
\end{rem}

\sssec{} \label{sssec:examples-unital-actions}

Given $\G$ be a arbitrary group object $\G$ of $\RanCorr_{/\Ran}$ and $\CRan \in \R\mmod^\un$, 
we say that $\G$ acts \emph{unitally} on $\CRan$ if the latter is given the structure of a $\Dmod^!(\G)$-comodule in the $\infty$-category $\R\mmod^\un$.
(By Remark \ref{rem:examples}, the usual actions of $G(\sA)$ and $T(\sK) \ltimes N(\sA)$ on $\Dmod(\GrGRan)$ are examples of unital actions.)

\medskip

In the presence of such structure, we can form the invariant category $(\CRan)^{\G}$ by the usual totalization. The invariant category, as well as the tautological morphism $\oblv: (\CRan)^{\G} \to \CRan$, are naive-unital. This is simply because, by Lemma \ref{lem:unital-formal}, $\R\mmod^\un$ admits limits.

\ssec{Independent subcategories}

Consider $\Vect$ as an object of $\R^\co \on- \mathbf{comod}$ via the comonoidal functor $\omega_\Ran \otimes - : \Vect \to \R^\co$. Using this, define the functor
$$
\indep: \R\mmod^\un \longto \DGCat, \,\,\,
\CRan \sto \lim \ms{Cobar}_{\R^\co}(\Vect, \CRan).
$$
We often write $\Cind$ in place of $\indep(\CRan)$. Observe that $\Cind$ is just a plain DG category, no longer tensored over $\Ran$.

\begin{rem} \label{rem:indep-monoidal}
By abstract nonsense, $\indep$ preserves arbitrary limits and is colax-monoidal. However, as $\Dmod(\Ran)$ is self-dual, $\indep$ can be described as the functor of tensoring with $\Vect$ over the monoidal category $(\Dmod(\Ran), \add_*)$, which is canonically lax-monoidal. Hence, $\indep$ is actually monoidal.
\end{rem}

\begin{prop}[See \cite{GL:Ran} for a sketch]
For $\CRan \in \R\mmod^\un$, the structure functor $\Cind \to \CRan$ is fully faithful.
\end{prop}

The next result will be tacitly used several times.

\begin{lem} \label{lem:indep-on.one.side}
Let $\ARan \to \CRan \leftarrow \BRan$ be a diagram in $\R \mmod^\un$ with $\ARan \to \CRan$ \emph{fully faithful}. Then the natural arrow $(\ARan \times_{\CRan} \BRan)_\indep \to {\ARan} \times_{\CRan} \Bind$ is an equivalence of categories. 
\end{lem}

\begin{proof}
Let us first treat the case when $\BRan \simeq \CRan$: the claim then reduces to the equivalence $\Aind \simeq \ARan \times_{\BRan} \Bind$, as subcategories of $\ARan$, which is clear. 
To deduce the general case, it suffices to apply the above to $\ARan \times_{\CRan} \BRan \to \BRan \leftarrow \BRan$, noticing that the leftmost arrow is fully faithful by construction.
\end{proof}

\sssec{Examples of independent subcategories} \label{sssec:domains}

We wish to recall some examples of independent subcategories, crucial in the sequel. To do so, let us recall the following definition.
An open subset $U \subseteq X_S$ is said to be a \emph{domain} if it is universally dense with respect to the projection $X_S \to S$. That is, for any map of affine schemes $T \to S$ we require that $U_T := U \times_{X_T} X_S$ be non-empty. Equivalently, $U$ is a domain iff the fiber $U_s$ is non-empty for any geometric point $s \in S$.  

\sssec{} \label{sssec:GSECT-indep}

For $\Y \to X$ a prestack, set
\begin{equation}
\GSectind{\Y} :
S \mapsto
 \left\{ 
(U, \sigma) \left| 
\begin{array}{l}
\mbox{$U$ is a domain in $X_S$,} \\
 \mbox{$\sigma$ a section of $\restr{S \times \Y}U \to U$}
\end{array}
 \right.
\right\} \big/ \sim,
\end{equation}
where $\sim$ is the equivalence relation that identifies $(U,\sigma) \sim (U', \sigma')$ whenever $\sigma$ and $\sigma'$ agree on $U \cap U'$.
Observe that the notion of domain is designed to make the above assignment into a functor. 
We also define $\GMapsind{-}$ in the obvious way. 
The reason for the choice of these notations will become clear after Proposition \ref{prop:GSect/Ran}. 

\begin{lem}[\cite{Ba}]
The functor $\GSectind{-}: \Sch^{\fty}_{/X} \to \PreStk$ preserves finite limits (in particular: fiber products) and open embeddings.
\end{lem}

\sssec{}

From \cite{Ba} and \cite{GL:contr}, one deduces the following important result.

\begin{thm}  \label{thm:contr:Barlev}
Let $Y \to X$ be a Zariski-locally trivial fibration with fibers satisfying the assumption of Theorem \ref{thm:contr:Ga}. Then $\GSectind{Y}$ is contractible. In particular, $H(\sK)^\indep$ is contractible for any affine algebraic group $H$.
\end{thm}

\begin{notat}
To reduce cuttler, we sometimes use $\GSectdom{-} := \GSectind{-}$ and $H(\KK)$ in place of $H(\sK)^\indep$.
\end{notat}

\sssec{}

Let us now give some explicit calculations of independent subcategories. These will be formal consequences of the following result. 

\begin{prop}[\cite{Ba}] \label{prop:GSect/Ran}
The pull-back functor along the obvious map $\GSectRan{\Y} \to \GSectRan{\Y}^\indep$ factors through an equivalence 
$$
\Dmod  \bigt{  \GSectRan{\Y}^\indep} 
\xto{\;\; \simeq \;\;} 
\Dmod ( \GSectRan{\Y})_\indep.
$$
\end{prop}

\sssec{} 

Assume $\CRan \in \R\mmod^\un$ is a naive-unital category over $\Ran$, acted on unitally by $H(\sK)$ for some group $H$. Suppose also that $H(\sK)$ acts on $ \GSectRan{\Y}$ in the symmetric monoidal $1$-category $\RanCorr_{/\Ran}$. 
Then the tensor product $\CRan \Otimes \Dmod(\GSectRan{\Y})$ inherits a unital action of $H(\sK)$ and 
$$
\bigt{\CRan \Otimes \Dmod(\GSectRan{\Y}) }^{H(\sK)}
$$
is also naive-unital, by Lemma \ref{lem:unital-formal}. 

\medskip

Remark \ref{rem:indep-monoidal} guarantees that $H(\K)$ acts on $\Cind$, as well as on $\Dmod(\GSectRan{\Y})$, and that there is a natural equivalence
$$
\bigt{\CRan \Otimes \Dmod(\GSectRan{\Y}) }^{H(\sK)}_\indep
\simeq
\bigt{\Cind \otimes \Dmod(\GSectind{\Y}) }^{H(\KK)}.
$$

\sssec{} \label{gen-red-G-bundles}

Similarly to the discussion for generic sections, for any sub-group $H \subseteq G$, define the prestack
$$
\Bun_G^{H\on-\gen}:
S \mapsto 
 \left\{ 
(\P_G; U, \alpha_H) \left| 
\begin{array}{l}
\mbox{$\P_G$ a $G$-bundle on $X_S$,} \\
\mbox{$U \subseteq X_S$ a domain, } \\
 \mbox{$\alpha_H$ an $H$-reduction of $\restr{\P_G}U$ }
\end{array}
 \right.
\right\} \big/ \sim.
$$
In particular, $\GrGdom$, the independent version of the affine Grassmannian, equals $\Bun_G^{1-\gen}$, while $\Bun_G^{G-\gen} \simeq \Bun_G$.

\medskip

Now, look at the natural action of $H(\sK)$ on $\GrGRan$, which is obviously unital. Hence, $\Dmod(\GrGRan)^{H(\sK)}$ is naive-unital and 
$$
\Dmod({\GrGRan})^{H(\sK)}_ \indep
\simeq 
\Dmod(\GrGind)^{H(\KK)}.
$$ 
The latter is equivalent to $\Dmod(\Bun_G^{H -\gen})$: in fact, by \cite{DS}, the natural arrow $\GrGind \to \Bun_G^{H-\gen}$ is an effective epimorphism in the {\'e}tale topology.

\ssec{Quasi-sections vs generic sections} \label{ssec:QSECT}

For future use, we record the relation between generic sections and the more classical notion of quasi-sections.

\sssec{}

For a locally-free sheaf $\E$ on $X$, recall that the \emph{space of quasi-sections of $\BP\E \to X$} is the prestack $\QSect{\BP\E}$ whose $S$-points are pairs $(\L, i: \L \hto \E \boxtimes \O_S)$, with $\L$ a line bundle on $X_S$ and $i$ an inclusion of coherent sheaves with $S$-flat cokernel.
It is well-know that $\QSect{\BP\E}$ is a scheme, expressible as a disjoint union of projective schemes. 

\medskip

There exists a natural map $\q: \QSect{\BP\E} \to \GSectdom{\BP\E}$, defined using the following observation (\cite{Ba}): given an $S$-point $(\L, i: \L \hto \E \boxtimes \O_S)$ of $\QSect{\BP\E}$, the locus where $i$ is $X_S$-flat is a domain. 
This map realizes $\GSectdom{\BP\E}$ as the quotient of $\QSect{\BP\E}$ by a proper equivalence relation (up to Zariski sheafification). In other words, letting $Q_\bullet: \bold\Delta^\op \to \Sch$ be the Cech simplicial scheme generated by $\q$, the following assertions hold:
\begin{itemize}
\item
$Q_1 \to Q_0 \times Q_0$ is a closed embedding;

\smallskip

\item
for any $n \geq m$, each face map $Q_n \to Q_m$ is proper;

\smallskip

\item
the tautological map $L^{\Zar}(\colim Q_\bullet) \to \GSectdom{\BP\E}$ is an isomorphism. (Here $L^{\Zar}$ denotes the Zariski sheafification functor.)
\end{itemize}

\sssec{}

For a locally closed subscheme $Y \subseteq \BP\E$, one defines
$$
\QSect{Y} := 
\GSectdom{Y}
\ustimes{\GSectdom{\BP\E}}
\QSect{\BP\E}.
$$
The notation $\QSect{Y}$ is abusive, as it does not mention the immersion $Y \subseteq \BP\E$ on which the definition depends. However, such an immersion should always be clear from the context.

\begin{lem}[\cite{BFG}] \label{lem:QSect-closed-emb}
If $Y \subseteq \BP\E$ is closed, then $\QSect{Y} \to \QSect{\BP\E}$ is a closed embedding.
\end{lem}

\begin{cor} \label{cor:QSect-closed-embedding}
With the notation above, let $f: Z \hto Y$ be an immersion such that the resulting map $f: Z \hto \BP\E$ is a closed embedding. Then the induced map $F: \GSectdom{Z} \to \GSectdom{Y}$ is a closed embedding, too.
\end{cor}

\sec{Fourier transform for meromorphic jets} \label{SEC:FT}
\renc{\VA}{V(\sA)}
\nc{\dualVA}{V^\vee_\omega(\sA)}
\renc{\K}{\mathsf{K}}
\renc{\A}{\mathsf{A}}

Let $V$ be a vector group, i.e. a group isomorphic to $\GG_a^n$ for some $n \geq 1$. We wish to develop the theory of Fourier transform for the relative group prestack $\VA \to \Ran$. 

As reviewed below, each $\VA_{X^I}$ is an ind-vector bundle of pro-finite type.\footcite{This stands to \virg{vector bundle} the same way $\IndSch^\pro_{/B}$ stands to $\Sch^{\fty}_{/B}$.}
We need the notion of dual of such an object: if $E = \colim_n \lim_r E_{n,r}$, then $E^\vee := \colim_r \lim_n (E_{n,r})^\vee$ along the dual transition maps. We refer to \cite{Be} for a discussion of the legitimacy of such a definition.
To fix conventions, whenever $\E$ is locally free on a scheme $B$, we write $\mathbb V(\E)$ for $\Spec_B (\Sym \E^\vee)$.

\ssec{The dual of the prestack of jets into a vector group}

Let $\Omega:= \Omega_X$ the canonical line bundle on $X$.
In this section, we recall that the dual of $\VA$ identifies with the prestack of meromorphic jets of $V^\vee$-valued $1$-forms. 

\sssec{}

For $n \geq 0$, define
$$
V(\A)^r_{n,X^I}
:=
\mathbb V  \Bigt{ p_* \big( V \otimes \O_{\D}(n \Gamma_\inc)/ \O_{\D}(-r\Gamma_\inc) \big)},
$$
where $p$ denotes the obvious map $p: \D \to X^I$.
Each $V(\A)^r_{n,X^I}$ is a vector bundle on $X^I$. By Grothendieck duality, its dual is described as follows.

\begin{lem}
The dual of $V(\A)^r_{n,X^I}$ is the prestack
$$
\dualVA_{n,X^I}^r: 
S \mapsto 
\left\{ 
(\x, \sigma) \left| 
\begin{array}{l}
 \x \in X^I(S), \\
\sigma \mbox{ a section of $V^\vee \otimes \big( \O_S \boxtimes \Omega_X(r \Gamma_\inc )/ \O_S \boxtimes \Omega_X (- n \Gamma_\inc) \big)$}
\end{array}
 \right.
\right\}.
$$
\end{lem}

\sssec{}

By taking the limit of $\dualVA^r_{n, X^I}$ as $n \to \infty$, we obtain
$$
\dualVA^r_{X^I}:
S \mapsto 
\left\{ 
(\x, \sigma) \left| 
\begin{array}{l}
 \x \in X^I(S), \\
\sigma \mbox{ a $V^\vee$-valued section of $\restr{(\O_S \boxtimes \Omega_X(r \Gamma))}{\D_\x}$}
\end{array}
 \right.
\right\}.
$$
This is a pro-vector bundle over $X^I$, and the natural map $\dualVA^r_{X^I} \to \dualVA^{r+1}_{X^I}$ an injection. Finally, taking the colimit as $r \to \infty$ and assembling over the Ran space, we obtain:
$$
(\VA)^\vee \simeq \dualVA:
S \mapsto 
\left\{ 
(\x, \sigma) \left| 
\begin{array}{l}
 \x \in \Ran(S), \\
\sigma \mbox{ a $V^\vee$-valued section of $\O_S \boxtimes \Omega_X$ over $\D_\x^\circ$}
\end{array}
 \right.
\right\}.
$$

\sssec{}

Let
\begin{equation} \label{eqn:residue}
\Res := \Res_{X_S/S} : (\GG_a)_\omega(\sA) \to \GG_a.
\end{equation}
denote the relative residue map. Unwinding the constructions, the evaluation pairing 
\begin{equation} \label{eqn:ev-VA}
\ev: \VA \Times \dualVA \to \GG_a
\end{equation}
is defined combining the residue map with the evaluation of vector groups $\wt \ev: V \otimes V^\vee \to \GG_a:$
$$
\VA \ustimes{\Ran} \dualVA
\xto{\, \mathsf{mult} \,}
(V \times V^\vee)_\omega(\sA)
\xto{ \, \wt\ev \,}
(\GG_a)_\omega(\sA)
\xto{ \, \Res \,}
\GG_a.
$$

\begin{rem} \label{rem:dual-twist}
Let $\L$ be a line bundle on $X$. The above argument shows that $V_\L(\sA)$, the group prestack of meromorphic jets of $V$-valued sections of $\L$, is dual to $(V^\vee)_{\L^{-1} \otimes \Omega}(\sA)$.
\end{rem}

\ssec{The Fourier transform equivalence} 

The abelian group structure on $\VA$ endows $\Dmod^*(\VA)$ with the structure of a commutative algebra object of $\R\mmod$.\footnote{Recall that $\R := (\Dmod(\Ran),\otimes)$.} The product is as usual called convolution and denoted by $\star$. 
We wish to construct a Fourier transform functor $\FT_{\VA}$ such that the following theorem holds.

\begin{thm} \label{thm:FT-Ran}
$\FT_{\VA}$ yields a equivalence
\begin{equation}
\Big(\Dmod^*(\VA), \star \Big) 
\xto{\, \, \simeq \,\,}
\Big(\Dmod^!(\dualVA), \Otimes \Big). 
\end{equation}
of commutative algebra objects in $\R\mmod$.
\end{thm}

The construction for each fixed $X^I$ was performed in \cite{Be}. Indeed, the usual Fourier transform for vector bundles over a scheme yields an equivalence 
\begin{equation}
\Big(\Dmod^*(\VA^r_{n,X^I}), \star \Big) 
\xto{\, \, \simeq \,\,}
\Big(\Dmod^!(\dualVA_{n,X^I}^r), \Otimes \Big)
\end{equation}
of commutative algebras in $\DGCat_{/X^I}$.
Now, Fourier transform transforms a push-forward along a linear map into the pull-back along the dual map, hence we obtain a monoidal equivalence
$$
\Big(\colim_r \on{lim}_n \, \Dmod^*(\VA^r_{n,X^I}), \star \Big) 
\xto{\, \, \simeq \,\,}
\Big(\colim_r \on{lim}_n \,\Dmod^!(\dualVA_{n,X^I}^r), \Otimes \Big).
$$
The category $\colim_r \on{lim}_n \, \Dmod^*(\VA^r_{n,X^I})$ is naturally equivalent to $\Dmod^*(\VA_{X^I})$. It remains to show that 
$$
\colim_r \on{lim}_n \,\Dmod^!(\dualVA_{n,X^I}^r) \simeq \Dmod^!(\dualVA_{X^I}).
$$
This holds because the natural map
$$
\colim_r \on{lim}_n \, \VA^r_{n,X^I}
\to
\on{lim}_n \, \colim_r  \VA^r_{n,X^I}
\simeq
\dualVA_{X^I}
$$
is an isomorphism (\textit{loc.$\,$cit.}). It is immediate to assemble these equivalences as $I \in \fset$ varies, and we obtain Theorem \ref{thm:FT-Ran}.

\sssec{}
\renc{\ch}{\mathsf{ch}}

Consider now the subfunctor $V(\K) \hto \VA$. The annihilator of $V(\K)$ inside $\dualVA$ will be denoted by $\ch_V$.
Using the theorem about the sum of residues, we deduce the following characterization of $\ch_V$.

\begin{lem}
The prestack $\ch_V \to \Ran$ is equivalent to the relative indscheme (of ind-finite type) parametrizing rational sections of $V^\vee \otimes \Omega_X$, that is, $\ch_V \simeq \GSectRan{\BV(V^\vee \otimes \Omega_X)}$.
\end{lem}
%

\begin{rem}
In the context of the Fourier transform, we always need to consider $\ch_V$ as fibered over $\Ran$. However, in the definition of the Whittaker categories and in the proof of Theorem \ref{thm:MAIN}, it will be essential to consider also the independent version $\chind_V$.
\end{rem}

\ssec{The extended coefficient functor for vector groups}

\sssec{}

Let $\W \hto \V$ an inclusion of relative ind-pro-vector bundles over a base $B$ (where $B \in \Sch^{\fty}$ or $B= \Ran$) and $\W^\perp$ the annihilator of $\W$ inside $\V^\vee$. Given $\CC \in \Dmod(B) \mmod$ a category acted on by $\V$, Fourier transform induces an equivalence
\begin{equation} \label{eqn:VK-invariants}
\CC^{\W} \simeq \uprestr {\CC}{\W^\perp},
\end{equation}
compatible with the forgetful functors to $\CC$; see \cite{Be}. (Of course, we are interested in the case where $\W \hto \V$ is $V(\sK) \hto V(\sA)$.)
We now manipulate the RHS of (\ref{eqn:VK-invariants}) in a trivial, yet useful, manner.

\sssec{} \label{sssec:Whit-annihilator}

Written out explicitly, the RHS of (\ref{eqn:VK-invariants}) is
$$
\uprestr {\CC}{\W^\perp} 
:=
\Hom_{\Dmod^!(\V^\vee)} \bigt{\Dmod^!(\W^\perp), \CC}.
$$
Thanks to the Bar resolution of $\Dmod^!(\W^\perp)$ as a $\Dmod^!(\V^\vee)$-module, we obtain
$$
\uprestr \CC{\W^\perp} 
\simeq
\lim \left(
\Hom(\Dmod^!(\W^\perp), \CC) \rr \Hom \bigt{\Dmod^!(\V^\vee) \Otimes \Dmod^!(\W^\perp), \CC} \rrr \cdots
\right).
$$
Next, we use duality:
$$
\uprestr \CC{\W^\perp} 
\simeq
\lim \left(
\Dmod^*(\W^\perp)\Otimes \CC  \rr \Dmod^*(\V^\vee) \Otimes \Dmod^*(\W^\perp) \Otimes \CC \rrr \cdots
\right).
$$
Finally, via the Fourier transform for $\V^\vee$, we get:
\begin{equation} \label{eqn:totalization-extWhit}
\uprestr \CC{\W^\perp} 
\simeq
\lim \left(
\Dmod^*(\W^\perp)\Otimes \CC  \rr \Dmod^!(\V) \Otimes \Dmod^*(\W^\perp) \Otimes \CC \rrr \cdots
\right).
\end{equation}
Let $i: \W^\perp \hto \V^\vee$ be the inclusion. Unraveling the equivalences, the arrows of this cosimplicial complex are generated by the following three kinds, according to the usual pattern:

\begin{itemize}
\item
 $\coact$, the coaction of $\Dmod^!(\V)$ on $\CC$;
\smallskip
\item
 $m_\V^!$, the pullback along the multiplication of $\V$;
\smallskip
\item
the coaction of $\Dmod^!(\V)$ on $\Dmod^*(\W^\perp)$ given by the formula
\begin{equation} \label{eqn:FT-weird}
\Dmod^*(\W^\perp)
\xto{\, \Delta_* \,}
\Dmod^*(\W^\perp) \Otimes \Dmod^*(\W^\perp)
\xto {\id \otimes i_*}
\Dmod^*(\W^\perp) \Otimes \Dmod^*(\V^\vee)
\xto{\id \otimes \FT_{\V^\vee}}
\Dmod^*(\W^\perp) \Otimes \Dmod^!(\V).
\end{equation}

\smallskip
\end{itemize}

\begin{defn}
The RHS of (\ref{eqn:totalization-extWhit}) will be denoted by $\bigt{\CC \otimes_{\Dmod(B)} \Dmod^*(\W^\perp)}^{\V,\ev}$.
\end{defn}

The following easy lemma, whose proof is left to the reader, reformulates the coaction (\ref{eqn:FT-weird}) without explicitly mentioning the Fourier transform. Therefore, it allows to extend the above definition to categories acted by nonabelian group indschemes (see Section \ref{SEC:ExtWhit}).

\begin{lem} \label{lem:FT-weird}
Assume for simplicity that $\W^\perp$ is an indscheme of finite type, so that $\Dmod(\W^\perp) \simeq \Dmod^*(\W^\perp)$ canonically. Let $\ev: \V \Times \W^\perp \to \GG_a$ the evaluation and $p:  \V \Times \W^\perp \to \W^\perp$ the projection. The coaction (\ref{eqn:FT-weird}) can be written as
$$
\Dmod(\W^\perp) \to \Dmod(\W^\perp) \Otimes \Dmod^!(\V) \simeq \Dmod^!(\W^\perp \Times \V),
\hspace{.4cm}
M \mapsto p^!(M) \otimes \ev^!(\exp).
$$
\end{lem}

\sssec{}

Let $\H$ be a group indscheme over $B$ acting on $\V$ and assume that $\W \hto \V$ is $\H$-equivariant. Then, $\H$ acts on $\W^\perp$ via the dual action. 

If $\CC \in \Dmod(B)\mmod$ is acted on by $\H \ltimes \V$, both $\CC^\W$ and $\uprestr\CC{\W^\perp}$ retain an action on $\H$. By rerunning the above construction, we obtain a natural equivalence
\begin{equation} \label{eqn:coeff_V-abstract}
\CC^{\H \ltimes \W}
\xto{ \, \, \simeq \,\,}
\bigt{\CC \otimes_{\Dmod(B)} \Dmod^*(\W^\perp)}^{\H \ltimes \V,\ev}
:=
\Bigt{
\bigt{
\CC \Otimes \Dmod^*(\W^\perp)
}^{\V,\ev}
}^\H,
\end{equation}
where now $\Otimes$ stands for the tensor product in the $\infty$-category $(\Dmod(\H), m^!)\comod(\Dmod(B)\mmod)$.

\begin{example}

For instance, assume $V$ is equipped with an action of an algebraic group $H$. Note that there is an induced action of $\H:= H(\sK)$ on $\V := \VA$. For $\W := V(\sK)$, we will see in Section \ref{sssec:coeff-GL2-conclusion} that the equivalence (\ref{eqn:coeff_V-abstract}) is exactly the extended coefficient functor $H \ltimes V$, thereby proving the following result.

\begin{prop} \label{prop:first-step}
Let $\CRan \in \R\mmod$ be acted on by $H(\sK) \ltimes \VA$. The the \emph{$\Ran$-version of the extended coefficient functor}
\begin{equation} \label{eqn:obvious2}
\coeffRan_{\ch_V}:
(\CRan)^{H(\sK) \ltimes V(\sK)} 
\longto
\big( \CRan \usotimes{\Dmod(\Ran)} \Dmod(\ch_V) \big) ^{H(\sK) \ltimes \VA, \ev}
\end{equation}
is an equivalence.
\end{prop}

We will discover that, when $H \ltimes V$ is the Levi decomposition of $B_{GL_2}$ (resp., $B_{PGL_2}$), this is our main theorem for the group $GL_2$ (resp., $PGL_2$).

\end{example}

\sec{Definition of the extended Whittaker category} \label{SEC:ExtWhit}

\renc{\NA}{N(\sA)}
\renc{\GA}{G(\sA)}
\nc{\TK}{T(\sK)}
\renc{\c}{\mathsf{c}}

To provide a geometric version of Definition \ref{defn:extWhit-functions}, one has to meticolously convert all the objects appearing there into objects of algebraic geometry.
For instance, it is well-known that the quotient $G(\AA)/\GO$ ought to be replaced by the Beilinson-Drinfeld Grassmannian $\GrGRan$ over the Ran space, and $\Fun(G(\AA)/\GO)$ by the category $\Dmod(\GrGRan)$.
However, it is convenient to proceed in greater generality: instead of $\Dmod(\GrGRan)$, we start with $\CRan \in \R\mmod$, equipped with an action of $T(\K) \ltimes N(\A)$. This datum is enough to make sense of the \emph{$\Ran$-version of the extended Whittaker category of $\CRan$}, denoted $\Wh_{G,\ext}[\CRan]$, which is again an object of $\R\mmod$.

However, $\Dmod(\Gr_G)$ has more structure: it is \emph{naive-unital} and it is acted on by $T(\sK) \ltimes N(\sA)$ \emph{naive-unitally}, see Section \ref{sssec:examples-unital-actions}.
Assume $\CRan$ possesses the same structure: in this case, $\Wh_{G,\ext}[\CRan]$ is naive-unital as well, and we define the genuine \emph{extended Whittaker category} of $\CRan$ to be the DG category
$$
\Wh_{G,\ext}[\Cind] := \Bigt{\Wh_{G,\ext}[\CRan]}_\indep.
$$
Note the abuse of notation on the LHS: $\Wh_{G,\ext}[\Cind]$ does depend on $\CRan$, not just on $\Cind$. However, any $\Cind$ arising in practice has a \virg{canonical} $\CRan$ attached to it.

\ssec{General framework for the construction}

We start by explaining the main contruction over a base scheme of finite type $B$.
In the most general situation, the ingredients of the definition are a category acted on by a semidirect product $\T \ltimes \N$ of group indschemes and a space $\c$ of additive characters of $\N$ which is $\T$-invariant with respect to the dual $\T$-action. The details are given below.
(As promised, the construction that follows is an extension of  the one in Section \ref{sssec:Whit-annihilator}, with the aid of Lemma \ref{lem:FT-weird}.)

\sssec{} \label{sssec:conditions-whit-abstract}

Assume given the following data.

\begin{itemize}

\medskip
\item[(i)]
Let $B$ be a $\kk$-scheme of finite type, viewed as a base. Consider the symmetric monoidal $\infty$-category $\DGCat_{/B} := \ShvCat(B_\dR) \simeq \Dmod(B) \mmod$ of DG categories fibered over $B$. 
We denote the monoidal operation of $\DGCat_{/B}$ by $\Otimes$ and $\times_B$ by $\Times$.
\medskip
\item[(ii)]
Let $(\N,m)$ be a group object in the $1$-category $\IndSch^{\pro}_{/B}$, so that $(\Dmod^!(\N),m^!)$ is a coalgebra object of $\DGCat_{/B}$.

\medskip
\item[(iii)]
Let $\CC \in \DGCat_{/B}$ be equipped with an $\N$-action, i.e. a coaction of $(\Dmod^!(\N), m^!)$ on $\CC$ is given.

\medskip
\item[(iv)]
Let $\c$ be another group object of $\IndSch^{\pro}_{/B}$ endowed with a bilinear pairing $\ev: \N \Times \c \to \GG_a$. Note that the natural functor
\begin{equation}
\boxtimes: \Dmod^!(\G) \Otimes \Dmod^!(\c) \longrightarrow \Dmod^! \big( \G \Times \c \big)
\end{equation}
is an equivalence.

\medskip
\item[(v)]
Assume furthermore that a group object $\T$ of $\IndSch_{/B}$ acts on $\N$ and $\c$ compatibly with the evaluation, i.e. $\ev$ is $\T$-equivariant with respect to the diagonal action on $\N \Times \c$ and the trivial action on $\GG_a$. We require that the $\N$-action on $\CC$ extends to a $\T \ltimes \N$-action. 
\end{itemize}

\sssec{}

Under the assumptions (i)-(iv) above, we define the object $\big( \CC \Otimes \Dmod({\mathfrak{c}}) \big)^{\G,\ev}$ of $\DGCat_{/B}$, equipped with a conservative functor
\begin{equation} \label{eqn:oblv_G,ev}
\oblv^{\G,\ev}:
\left(\CC \Otimes\Dmod^!(\c) \right )^{\G,\ev}
\longrightarrow
\CC \Otimes \Dmod^!(\c) .
\end{equation}
Let us regard $ \CC \Otimes \Dmod^!(\c)$, as a category over $B$ equipped with an $\N$-action, where $\N$ acts trivially on $\Dmod^!(\c)$. Recall also the \emph{exponential} $\fD$-module $\exp \in \Dmod(\GG_a)$, which is \virg{additive}:
\begin{equation} \label{eqn:add-exp}
\ms{add}^! (\exp) \simeq \exp \boxtimes \exp.
\end{equation}
Since $\ev$ is bilinear and $\exp$ additive, the $\fD$-module $\ev^!(\exp)$ is a \virg{bilinear} sheaf on $\Dmod^!(\N) \Otimes \Dmod^!(\c)$: this means that compatible identifications
$$
\begin{array}{l}
m_{\N}^!(\ev^!(\exp)) \simeq (p_1)^!(\ev^!(\exp)) \Otimes (p_2)^!(\ev^!(\exp)) \vspace*{.3cm} \\
m_\c^!(\ev^!(\exp)) \simeq (p'_1)^!(\ev^!(\exp)) \Otimes (p'_2)^!(\ev^!(\exp)) \vspace{.3cm}
\end{array}
$$
are given, where $p_i : \N \Times \N \Times \c \to \N \Times \c$ and $p'_i : \N \Times \c \Times \c \to \N \Times \c$ are the projections.

\medskip

The bilinearity allows to define the following cosimplicial object in $\DGCat_{/B}$:
\begin{equation} \label{eqn:totalization-abstract}
\CC \Otimes \Dmod^!(\c) 
\rr
\CC \Otimes \Dmod^!(\c) \Otimes \Dmod^!(\N) 
\rrr 
\CC \Otimes \Dmod^!(\c) \Otimes \Dmod^!(\N)^{\Otimes 2}
\cdots,
\end{equation}
where the arrows are generated, as in \ref{sssec:Whit-annihilator}, by the following three types:

\begin{itemize}
\item
 $\coact$, the coaction of $\N$ on $\CC$;
\smallskip
\item
 $m_\N^!$, the pullback along the multiplication of $\N$;
\smallskip
\item
$p^! \Otimes \ev^!(\exp) : \Dmod^!(\c) \to \Dmod^!(\c) \Otimes \Dmod^!(\N)$, the pullback along the structure projection $p: \N \to B$ \emph{twisted} by $\ev^!(\exp)$.
\smallskip
\end{itemize}
Thus, we define $\bigt{\CC \Otimes \Dmod^!(\c) }^{\N, \ev}$ as the totalization of (\ref{eqn:totalization-abstract}) and 
$$
\oblv^{\N, \ev} : \bigt{\CC \Otimes \Dmod^!(\c) }^{\N, \ev} \to \CC \Otimes \Dmod^!(\c)
$$
as the tautological forgetful functor. Such functor is fully faithful whenever $\N$ is a colimit of pro-unipotent group schemes (cf. Section \ref{sssec:invariants-NA}).

\sssec{} 
\renc{\TK}{T(\K)}

The category $\bigt{\CC \Otimes \Dmod^!(\c)}^{\N, \ev}$ is meant to be the analogue of the space of functions of Definition \ref{defn:extWhit-functions} that satisfy the first condition there ($N(\AA)$-equivariance).
To address the second condition of \textit{loc.cit}, we need item (v) of the list in \ref{sssec:conditions-whit-abstract}: regard $\CC \Otimes \Dmod^!(\c)$ as a category acted on by $\T$ (diagonally) and let
$$
\oblv^\T: \bigt{\CC \Otimes \Dmod^!(\c) }^\T \to \CC \Otimes \Dmod^!(\c)
$$
the usual oblivion functor. 

\begin{defn} 

We shall consider the object $\bigt{ \CC \Otimes \Dmod^!(\c) }^{\T \ltimes \N, \ev} \in \DGCat_{/B}$. According to Section \ref{semi-direct}, this is the fiber product
\begin{equation} \label{defn:Whittaker-B}
\bigt{ \CC \Otimes \Dmod^!(\c) }^{\T \ltimes \N, \ev}
:= 
\bigt{ \CC \Otimes \Dmod^!(\c) }^{\T } 
\ustimes{\CC \Otimes \Dmod^!(\c)}
\bigt{ \CC \Otimes \Dmod^!(\c) }^{\N, \ev}.
\end{equation}
\end{defn}

\sssec{}

We now extend these constructions to categories over $\Ran$, in the obvious way.
To be precise, assume given the data:
\begin{itemize}

\item[(ii')]
$(\N,m)$, a group object in the $1$-category of relative indschemes over $\Ran$, so that $(\bold\Dmod^!(\N),m^!)$ is a coalgebra object of $\DGCatRan$; 

\medskip
\item[(iii')]
$\CRan \in \R \mmod$, equipped with an $\N$-action;

\medskip
\item[(iv')]
$\c$ another relative indscheme over $\Ran$, and $\ev: \N \times_\Ran \c \to \GG_a$ an bilinear map;

\medskip
\item[(v')]
$\T \to \Ran$ a relative group indscheme, satisfying the same conditions as before.
\end{itemize}

\medskip

\noindent 

Then we may define the following sheaf of categories (over $\Ran_\dR$):
\begin{equation}
\bigt{\Cshv \Otimes \bold\Dmod^!(\c)}^{\T \ltimes \N,\ev}
:= \left\{  
\Bigt{
\Cmod{X^I} \usotimes{\Dmod(X^I)} \Dmod^!(\c_{X^I})
}
^{\T_{X^I} \ltimes \N_{X^I}, \ev}
 \right\}_{I \in \fset}.
\end{equation}
Since we are mostly interested in the global sections of this sheaf of categories, we set
$$
\bigt{ \CRan \Otimes \Dmod^!(\c) }^{\T \ltimes \N, \ev}
:= 
\bGamma \Bigt{ \Ran_\dR, \bigt{\Cshv \Otimes \bold\Dmod^!(\c)}^{\T \ltimes \N,\ev} }.
$$

\begin{rem}
This is not an abuse of notation, for $\bGamma(\Ran_\dR, \Cshv \Otimes \bold{\Dmod}^!(\c))$ is equivalent to the RHS of (\ref{defn:Whittaker-B}) upon substituting $B$ with $\Ran$. To prove this, use the dualizability of $\Dmod^!(\Y)$ as a $\Dmod(\Ran)$-module for $\Y= \c, \N, \T$.
\end{rem}

\ssec{Functoriality}

Let us discuss natural functors between categories of the above sort. 

\sssec{} \label{sssec:adjoint-naively}

The following basic fact will be used several times. Let $I$ be an arbitrary index $\infty$-category and $\Delta^1$ the $1$-simplex category. Let $I \times \Delta^1 \to \DGCat$ be a diagram: this consists of two $I$-diagrams $i \squigto \CC_i$, $i \squigto \E_i$ and a family of compatible functors $f_i: \CC_i \to \E_i$. We obtain a functor $F: \lim_i \CC_i \to \lim_i \E_i$. Assume now that each $f_i$ has a left adjoint and that they together give rise to a diagram $I \times (\Delta^1)^\op \to \DGCat$. Then $F$ has a left adjoint, computed termwise.

\begin{lem} \label{lem:f^!-oblv}
Let $f: \c \to \c'$ be a $\T$-equivariant morphism in $\IndSch^\pro_{/B}$ and let $\ev': \N \Times \c' \to \GG_a$ be such that $\ev = \ev' \circ f$.
Then, the pull-back $f^!: \Dmod^!(\c') \to \Dmod^!(\c)$ induces a morphism
$$
\bf^!: 
\bigt{ \Cmod{B} \Otimes \Dmod^!(\c')  }^{\T \ltimes \N, \ev'}
\longrightarrow
\bigt{ \Cmod{B} \Otimes \Dmod^!(\c)  } ^{\T \ltimes \N, \ev}.
$$
\end{lem}

\begin{proof}
Obviously, $f^!: \Dmod^!(\c') \to \Dmod^!(\c)$ is $\Dmod(B)$-linear, whence it yields $\bf^! : \Cmod{B} \Otimes \Dmod^!(\c')  \to \Cmod{B} \Otimes \Dmod^!(\c)$. The latter is manifestly compatible with the coactions of $\Dmod^!(\T)$ and $\Dmod^!(\N)$ on both categories, so that it descends to functors $\bf^!: \bigt{ \Cmod{B} \Otimes \Dmod^!(\c') } ^{\T} \to \bigt{ \Cmod{B} \Otimes \Dmod^!(\c)  } ^{\T}$ and $\bigt{ \Cmod{B} \Otimes \Dmod^!(\c')} ^{ \N, \ev'} \to \bigt{ \Cmod{B} \Otimes \Dmod^!(\c)  } ^{ \N, \ev}$.
\end{proof}

The proof of the following result consists of rearranging fiber products and it is left to the reader. (It is parallel to the one of Lemma \ref{lem:indep-on.one.side}.)

\begin{lem} \label{lem:useful-base-change-limits}
In the notation of the above lemma, assume that 
$$
\bf^!:
 \Cmod{B} \Otimes \Dmod^!(\c')  
\longto
\Cmod{B} \Otimes \Dmod^!(\c)  
$$
is \emph{fully faithful}. Then, the natural functor
$$
\bigt{ \Cmod{B} \Otimes \Dmod^!(\c')  } ^{\T \ltimes \N, \ev'}
\longto
\bigt{ \Cmod{B} \Otimes \Dmod^!(\c')  } ^{\T}
\ustimes
{\bigt{ \Cmod{B} \Otimes \Dmod^!(\c)  } ^{\T}}
\bigt{ \Cmod{B} \Otimes \Dmod^!(\c)  } ^{\T \ltimes \N, \ev}
$$ 
induced by $\bf^!$ is an equivalence.
\end{lem}

From now on, assume $\c$ (as well as $\c'$) is of ind-finite type, so that $\Dmod^!(\c) \simeq \Dmod(\c)$ canonically.

\begin{lem} \label{lem:f_*-oblv}
Let $f: \c \to \c'$ be a $\T$-equivariant morphism in $\IndSch^\fty_{/B}$ such that $\ev' = \ev \circ f$. Then, $f_*: \Dmod(\c) \to \Dmod(\c')$ induces a morphism
$$
\bf_*: 
\bigt{ \Cmod{B} \Otimes \Dmod(\c)  }^{\T \ltimes \N, \ev}
\longrightarrow
\bigt{ \Cmod{B} \Otimes \Dmod(\c')  } ^{\T \ltimes \N, \ev'}.
$$
If $f$ is proper, then $\bf_*$ is left adjoint to $\bf^!$. If $f$ is an open embedding, then $\bf_*$ is right adjoint to $\bf^!$.
\end{lem}

\begin{proof}
By base change, $f_*: \Dmod(\c) \to \Dmod(\c')$ is $\Dmod(B)$-linear, whence $\bf_* : \Cmod{B} \Otimes \Dmod(\c)  \to \Cmod{B} \Otimes \Dmod(\c')$ is defined. It remains to check that this functor is $\T$- and $\N$-equivariant: both claims follow again from base change.
The stated adjunctions are automatic.
\end{proof}

\begin{cor}
In the situation of the above lemma, if $f$ is an open or closed embedding, then $\bf_*$ is fully faithful.
\end{cor}

\ssec{The $\Omega$-twist} \label{ssec:twists}

Before getting to the proper definitions of the (extended, degenerate) Whittaker categories of $\CRan \in \R\mmod$, we need to introduce twisted versions of the prestacks of jets into $G$ and \emph{similia}.

\nc{\Pom}[1]{\P_{#1}^{\,{\Omega}}}
\nc{\Omrho}{\Omega^\rhoch}
\nc{\twistedform}[1]{\dot{#1}}

\nc{\Ht}{\twistedform{H}}
\nc{\Gt}{\twistedform{G}}
\nc{\Pt}{\twistedform{P}}
\nc{\Qt}{\twistedform{Q}}
\nc{\Bt}{\twistedform{B}}
\nc{\Tt}{\twistedform{T}}
\nc{\Nt}{\twistedform{N}}
\nc{\Ut}{\twistedform{U}}
\nc{\Vt}{\twistedform{V}}
\nc{\Mt}{\twistedform{M}}
\nc{\Zt}{\twistedform{Z}}

\sssec{}

Let $2\rhoch$ be the coweight of $G$ equal to the sum of all positive coroots and denote by $\Omega^{2\rhoch} := \Omega_X \times^{\GG_m} T$ the induced $T$-bundle on $X$. We choose once and for all a square root of $\Omega^{2\rhoch}$, which we denote by $\Omega^\rhoch$.

\medskip

For any subgroup $H \subseteq G$ containing $T$ (e.g.: $P$, $M$, $Z_M$ for any $P \in \Par$), we form the $H$-bundle $\Pom{H}:= \Omrho \times^T H$. Let $\Ht$ be the group scheme over $X$ defined by 
$$
\Ht := \Aut_H(\Pom{H}).
\footnote{Perhaps $H^\Omega$ would have been a better notation for $\Ht$, but we find the latter less invasive.}
$$
This is a pure inner form of the trivial group scheme $H \times X$. In particular, $\Bun_H \simeq \Bun_{\Ht}$, canonically.

\medskip

We need to twist also the unipotent radicals of $P$'s accordingly: for any $P \in \Par$ with Levi decomposition $P = M \ltimes U_P$, set
$$
\Ut_P := \ker (P^\Omega \to M^\Omega).
$$ 
This way, we have $\Pt \simeq \Mt \ltimes \Ut_P$ canonically.

\sssec{}

Next, for any $H \subseteq G$ for which $\Ht$ has been defined, we consider generic sections and (meromorphic) jets of sections into $\Ht$; e.g.:
$$
\Ht(\sA) := \Jets^{\mer} \bigt{ \Ht / X }.
$$

The only reason we adopt this twist is to have canonical sections of the spaces of characters $\ch_{G,\ext}$: see Proposition \ref{prop:Whit(G,P)-at-chi} and Section \ref{sssec:twists-xi}.

\ssec{The extended Whittaker category}
\renc{\ch}{\mathsf{ch}}

\nc{\Ntk}{\Nt(\sK)}
\nc{\Ttk}{\Tt(\sK)}
\nc{\NtA}{\Nt(\sA)}

Let $\CRan \in \R\mmod$ be acted upon by $\Tt(\sK) \ltimes \Nt(\A)$. To define the extended Whittaker category of $\CRan$, we wish to apply verbatim the constructions above. Hence, it remains to specify what $\c$ will be. 

\sssec{}

As in the function theoretic situation, we set $\c:=\ch_{G,\ext}$ to be the indscheme (of ind-finite type) of characters of $\NtA$ that are trivial on $\Ntk \hto \NtA$. This inclusion is obviously $\Ttk$-equivariant.

\medskip 
 
To place this definition in the framework of Section \ref{SEC:FT}, note that the natural maps
$$
\NtA/[\NtA, \NtA] \to \NtA/[\Nt,\Nt](\sA) \to \Nt/[\Nt,\Nt](\sA)
$$
are all isomorphisms. As a consequence, $\ch_{G,\ext}$ is the annihilator of $\Nt/[\Nt,\Nt](\sK)$ inside the dual of $\NtA/[\NtA, \NtA]$.
Explicitly, using our chosen Chevalley generators,
\begin{equation} \label{eqn:def-ch}
\ch_{G,\ext}:
S \mapsto
 \left\{ 
\big( \x, \eta= \{\eta_i\}_{i \in \I} \big)
\left| 
\begin{array}{l}
 \x \in \Ran(S), \\
\mbox{$\eta_i$ a section of $\restr{\O_S \boxtimes \O_X}{U_\x}$ }
\end{array} \right.
\right\},\footnote{
Had we not followed the convention of \ref{ssec:twists}, the $\O_X$ in the above formula would have been replaced by $\Omega_X$.}
\end{equation}
see Remark \ref{rem:dual-twist}. When there is no risk of confusion, we may occasionally write $\ch:= \ch_{G,\ext}$.

\begin{rem}
There is slight clash of notation with the parallel definition in the function theory setting. There, $\mathfrak{ch}$ was the space of $\mathbb{C}^\times$-valued characters, whereas in the geometric picture $\ch$ represents additive characters. However, this is harmless as a canonical \virg{exponential} character, the $\fD$-module $\exp$, is available in the current setting.
\end{rem}

\begin{rem}
Let us remind the reader that we are under the assumption that $G$ has connected center. Otherwise, $\ch$ is to be defined differently: the general case is explained in \cite[Sections 8.1, 8.2.]{Outline}.
\end{rem}

\sssec{}

For the sake of clarity, let us unravel the evaluation pairing $\ev$ between $\NtA$:
$$
\ev:
\Nt(\A) \underset {\Ran} \times \ch
\twoheadrightarrow
\Nt/[\Nt,\Nt](\A) \underset {\Ran} \times \ch
\hto
(\GG_a^{\oplus \I})_\omega(\A) \underset {\Ran} \times (\GG_a^{\oplus \I})(\sA)
\to
\GG_a,
$$
where the last map is (\ref{eqn:ev-VA}) associated to the standard self-duality of $\GG_a^{\oplus \I}$.
Again, for the sake of clarity, the \virg{dual} $\Tt(\K)$-action on $\ch$ is given by
$$
\Big(\x, t, \{\eta_i\}_{i \in \I} \Big) \mapsto \Big(\x, \{\alpha_i(t) \cdot \eta_i\}_{i \in \I} \Big).
$$

\begin{defn} \label{defn:whit-ext-C-Ran}
For $\CRan \in \R\mmod$ acted on by $\Tt(\K) \ltimes \Nt(\A)$, define $\Wh_{G,\ext}[\CRan]$, the \emph{$\Ran$ version of the extended Whittaker category of $\CRan$}, as 
$$
\Wh_{G,\ext}[\CRan] 
:=
\big( \CRan \Otimes \Dmod(\ch) \big)^{\Ttk \ltimes \NtA, \ev}.
$$
Note that $\Wh_{G,\ext}[\CRan]$ is the category of global sections of 
$$
\Wh_{G,\ext}[\Cshv] :=
\big( \Cshv \Otimes \bold\Dmod(\ch) \big)^{\Ttk \ltimes \NtA, \ev}
\in \DGCatRan.
$$
\end{defn}

\sssec{}

This is as far as the definition can go for a general $\CRan \in \R\mmod$; to extract a plain DG category from $\Wh_{G,\ext}[\CRan]$, we need a naive-unital structure on $\CRan$.

\begin{lem} \label{lem:Whitt-unital}
Let $\CRan \in \R\mmod^\un$ be acted unitally on by $\Tt(\K) \ltimes \Nt(\A)$. Then $\Wh_{G,\ext}[\CRan]$ admits a canonical naive-unital structure.
\end{lem}

\begin{proof}
It suffices to show that $(\CRan \Otimes \Dmod(\ch))^{\NtA,\ev}$ is canonically naive-unital. By writing it as a totalization, we just need to see that the arrow $p^! \otimes \ev^!(\exp): \Dmod(\ch) \to \Dmod^!(\NtA) \Otimes \Dmod(\ch)$ is naturally a morphism in $\R\mmod^\un$. Let us form the correspondence expressing the unitality of $\NtA \times_\Ran \ch$, together with the evaluations maps to $\GG_a$:
\begin{gather}  
\xy
(50,0)*+{ \GG_a.}="10";
(0,15)*+{ \Ran \times (\NtA \ustimes\Ran \ch)   }="01";
(50,15)*+{ \Nt(\sO, \A) \ustimes{\Ran \times \Ran} (\Ran \times \ch) }="11";
(100,15)*+{  \NtA \ustimes\Ran \ch }="21";
{\ar@{->}_{  \alpha   } "11";"01"};
{\ar@{->}^{ \;\;\;\;\beta  } "11";"21"};
{\ar@{->}_{\ev} "01";"10"};
{\ar@{->}^{\ev} "21";"10"};
\endxy
\end{gather}	
Notice that the diagram is commutative, as $\ev$ is null on $N(\sO) \subseteq N(\sA)$. Hence, our claim follows from the contractibility of the fibers of $\beta$.
\end{proof}

We are now ready to give the main definition of this paper.

\begin{defn} \label{defn:whit-ext-C-indep}
For $\CRan \in \R\mmod^\un$, acted on unitally by $\Tt(\K) \ltimes \Nt(\A)$, define $\Wh_{G,\ext}[\Cind]$, the \emph{extended Whittaker category of $\CRan$}, as 
$$
\Wh_{G,\ext}[\Cind] :=
\big( \CRan \Otimes \bold\Dmod(\ch) \big)^{\Ttk \ltimes \NtA, \ev}_\indep
:=
\indep \Bigt{ \big( \CRan \Otimes \bold\Dmod(\ch) \big)^{\Ttk \ltimes \NtA, \ev}  }.
$$
(We emphasize again that this notation is abusive, as the LHS depends also on $\CRan$, not just on $\Cind$.)
\end{defn}

As a particular case of the above construction, the following plays a central role in the geometric Langlands program. The \emph{extended Whittaker category of the affine Grassmannian} is the DG category
$$
\Wh(G,\ext) := \Wh_{G,\ext}[\Dmod(\Gr^{\indep}_{\Gt})],
$$
where $\Gr^{\indep}_{\Gt}$, the \emph{$\Omega$-twisted Grassmannian}, parametrizes $G$-bundles equipped with a rational isomorphism to $\Pom{G}$. 

\begin{rem}
By Lemma \ref{lem:indep-on.one.side}, we can also write
$$
\Wh_{G,\ext}[\Cind] 
:=
\bigt{ \CRan \Otimes \Dmod(\ch) }^{\Tt(\sK)}_{\indep} 
\ustimes{\CRan \Otimes \Dmod(\ch)}
\bigt{ \CRan \Otimes \Dmod(\ch) }^{\NtA, \ev}.
$$ 
This longer formula is particularly useful when treating the Whittaker category of the affine Grassmannian. The reason is that the independent category of $\Dmod(\Gr_{\Gt} \Times \ch) ^{\Tt(\sK)}$ is the category of $\fD$-modules on an explicit prestack.
\end{rem}

\ssec{The degenerate, or partial, Whittaker categories}

Important variants of Definitions \ref{defn:whit-ext-C-Ran} and \ref{defn:whit-ext-C-indep} arise by replacing $\ch_{G,\ext}$ with some particular $\Ttk$-invariant subspaces related to parabolic subgroups. 

\begin{example}

Let $P$ be a standard parabolic subgroup, with Levi $M$. The closed embedding $i: \ch_{M,\ext} \hto \ch_{G,\ext}$, determined in (\ref{eqn:def-ch}) by the vanishing of the $\eta_i$'s corresponding to $\I - \I_M$, yields the colocalization
$$
\bi_*:
\bigt{\CRan \Otimes \Dmod(\ch_{M,\ext})}^{\Ttk \ltimes \NtA, \ev}
\rightleftarrows
\Wh_{G,\ext}[\CRan]:
\bi^!.
$$
Passing to the independent subacategories, we derive the adjunction
$$
\bi_*:
\bigt{\CRan \Otimes \Dmod(\ch_{M,\ext})}^{\Ttk \ltimes \NtA, \ev}_\indep
\rightleftarrows
\Wh_{G,\ext}[\Cind]:
\bi^!,
$$
which we denote by the same symbols.
\end{example}

\sssec{}

Let $\chi_{G,P} \hto \ch_{G,\ext}$ be the image of the section of $\ch_{G,\ext} \to \Ran$ determined by the requirement that $\eta_i = 1$ for $i \in \I_M$ and $\eta_i = 0$ else.
Being able to canonically define these $\chi_{G,P}$'s is the \emph{only} reason for the usage of the $\Omega$-twisted versions of our groups (Section \ref{ssec:twists}).

\begin{defn}
For $\CRan \in \R\mmod$, define 
\begin{equation}
\Wh_{G,P}[\CRan] := (\CRan)^{\Zt_M(\sK) \ltimes \NtA,\chi_{G,P}}.
\end{equation}
If $\CRan$ is naive-unital, define also 
\begin{equation}
\Wh_{G,P}[\Cind] := 
\Bigt{\Wh_{G,P}[\CRan]}_\indep
\simeq
(\Cind)^{\Zt_M(\KK)} \ustimes{\CRan} {(\CRan)^{\NtA,\chi_{G,P}}}.
\end{equation}
\end{defn}

\sssec{}

The prestack $\chdom_{G,\ext} := \chind_{G,\ext}$ can be stratified in a $\TtKK$-invariant way according to the vanishing of the sections $\eta_i$ (equivalently, according to the poset of standard parabolics).\footnote{Recall we are assuming that $Z_G$ is connected.} In fact, define
\begin{equation}
\chdom_{G,P}:
S 
\mapsto
 \left\{ 
\big( U, \eta= \{\eta_i\}_{i \in \I_M} \big)
\left| 
\begin{array}{l}
 \mbox{$U \subseteq X_S$ a domain,} \\
\mbox{$\eta_i$ a nowhere zero section of $\restr{\O_S \boxtimes \O_X}{U} \to U$ }
\end{array} \right.
\right\} / \sim.
\end{equation}
It is easy to see that each $\chdom_{G,P} \hto \chdom_{G,\ext}$ is a locally closed embedding (cf. Corollary \ref{cor:QSect-closed-embedding}) and that the partition $\bigsqcup_{P \subseteq G} \chdom_{G,P}$ is indeed the $\TtKK$-orbit stratification of $\chdom_{G,\ext}$.

\begin{prop} \label{prop:Whit(G,P)-at-chi}
Let $\bar\chi_{G,P} \in \chdom_{G,P}$ be the image of $\chi_{G,P}$ along $\ch_{G,\ext} \to \chdom_{G,ext}$. Pullback along the inclusion $\{\bar\chi_{G,P}\} \hto \chdom_{G,P}$ yields an equivalence 
\begin{equation}
\displaystyle
\Wh_{G,P}[\Cind] 
\xto{\;\; \simeq \;\;}
\bigt{\Cind  \otimes \Dmod(\chind_{G,P})}^{\TtKK}
\ustimes{
\bigt{\Cind  \otimes \Dmod(\chind_{G,\ext})}^{\TtKK}
}
\Wh_{G,\ext}[\Cind].
\end{equation}
\end{prop}

\begin{proof}
This is an immediate application of Proposition \ref{prop:stabilizer}, after observing that the stabilizer of $\bar\chi_{G,P}$ within $\TtKK$ is exactly $Z_M(\KK)$.
\end{proof}

\begin{defn}
Pullback along the locally closed embedding $\chdom_{G,P} \hto \chdom_{G,\ext}$ yields the \emph{restriction functor} 
\begin{equation} \label{eqn:defn-rho}
\vrho_{G,\ext	\to P}:
\Wh_{G,\ext}[\Cind]
\longto
\Wh_{G,P}[\Cind].
\end{equation}
\end{defn}

The totality of these functors, as $P$ runs through the poset of standard parabolics, is \emph{jointly conservative} (this follows from the fact that $\bigsqcup_{P \subseteq G} \chdom_{G,P}$ is a stratification of $\chdom_{G,\ext}$). The reader may wonder whether (and if so, how) $\Wh_{G,\ext}[\Cind]$ can be reconstructed out of the $\Wh_{G,P}[\Cind]$. This is explained in \cite{Outline}.

\sec{The functor of Whittaker coefficient} \label{SEC:coeff}

The main goal of this paper is to show that a specific (continuous) functor \begin{equation} \label{eqn:coeff_bunG}
\coeff_{G,\ext}: \Dmod(\Bun_G) \to \Wh(G,\ext)
\end{equation}
is fully faithful for $G= GL_n$ and $G=PGL_n$.
As with the definition of the extended Whittaker category, we shall discuss the construction of this functor in the more general context of a naive-unital category over $\Ran$ acted on unitally by $M(\sK) \ltimes U(\sA)$. Here, $M \ltimes U$ is the Levi decomposition of an affine algebraic group.
In this section, the $\Omega$-twists are irrelevant (for instance, there is a canonical identification $\Bun_G \simeq \Bun_{\Gt}$) so we ignore them.


\medskip

The non-trivial point in the discussion of $\coeff_{G,\ext}$ is its continuity. This is ensured by the \emph{strong approximation} for meromorphic jets into a unipotent group $U$, Proposition \ref{prop:Av_^*-continuous-Wh-ext-noT}, parallel to the fact that in the function theoretic case $U(\AA)/U(K)$ is \emph{compact}.

\ssec{Definition of the coefficient functor and strong approximation}
\renc{\c}{\mathsf{c}}

\nc{\UA}{U(\sA)}
\nc{\rel}{\mathit{rel}}

\sssec{} \label{sssec:defn-coeff-U}

Let $M \ltimes U$ be the Levi decomposition of an algebraic group, and $\ch_U \to \Ran$ the relative indscheme of characters of $\UA$ trivial on $U(\sK)$.
For $\CRan \in \R\mmod$ acted on by $M(\sK) \ltimes \UA$, consider the inclusion 
$$
\oblv^{U(\sK)} : \big( (\CRan)^{U(\sK)} \Otimes \Dmod(\c) \big)^{M(\sK)} \to \big( \CRan \Otimes \Dmod(\c) \big)^{M(\sK)},
$$
where $\c \to \Ran$ is a unital relative indscheme equipped with a unital $M(\sK)$-equivariant map $\c \to \ch_U$. The induced evaluation $\UA \times_\Ran \c \to \GG_a$ will be also denoted by $\ev$.
The functor
$\oblv^{\UA, \ev}: \bigt{(\CRan) \Otimes \Dmod(\c)}^{M(\sK) \ltimes \UA, \ev}
\to
\big( \CRan \Otimes \Dmod(\c) \big)^{M(\sK)}$ factors through $\oblv^{U(\sK)}$, thus yielding the fully faithful functor
\begin{equation} \label{eqn:oblv-enhanced}
\oblv^{rel}:
\bigt{(\CRan) \Otimes \Dmod(\c)}^{M(\sK) \ltimes \UA, \ev}
\longto
\big( (\CRan)^{U(\sK)} \Otimes \Dmod(\c) \big)^{M(\sK)}.
\end{equation}
%
%
%
We may occasionally write $\bigt{\CRan \Otimes \Dmod(\c)}^{(M \ltimes U)(\sK)}$ for $\bigt{(\CRan)^{U(\sK)} \Otimes \Dmod(\c)}^{M(\sK)}$, being understood that the $U(\sK)$-action on $\Dmod(\c)$ is trivial.

\begin{rem}
Arguing as in Lemma \ref{lem:Whitt-unital}, we deduce that, if $\CRan$ is naive-unital acted on naive-unitally by $M(\sK) \ltimes U(\sA)$ , then $\oblv^\rel$ is tautologically a morphism of naive-unital categories.
\end{rem}

\begin{prop} \label{prop:Av_^*-continuous-Wh-ext}
The functor (\ref{eqn:oblv-enhanced}) admits a \emph{continuous} right adjoint, which we denote by $\Av_*^{\UA, \ev}$.
\end{prop}

The proof of Proposition \ref{prop:Av_^*-continuous-Wh-ext} is carried out just below. Assuming this result for the moment, we define $\coeffRan$ as the composition
$$
\coeffRan_{\c}:
(\CRan)^{(M \ltimes U)(\sK)} 
\xto{\mathsf{pullback}}
\big( (\CRan)^{U(\sK)} \Otimes \Dmod(\c) \big)^{M(\sK)}
\xto{ \,\, \Av_*^{\UA, \ev} \, \,}
\bigt{(\CRan) \Otimes \Dmod(\c)}^{M(\sK) \ltimes \UA, \ev}.
$$

\sssec{}

It is clear that $\Av_*^{U(\sA),\ev}$ is continuous if and only if so is $\Av_*^{U(\sA)}$. Thus, by replacing $\CRan \Otimes \Dmod(\ch)$ with $\CRan$, it suffices to prove the following:

\begin{prop} \label{prop:Av_^*-continuous-Wh-ext-noT}
Let $U$ be a unipotent group and $\CRan \in \R\mmod$ be acted on by $\UA$. The functor $\Av_*^{\UA}: \CRan \to (\CRan)^{\UA}$ is continuous on the subcategory $(\CRan)^{U(\sK)} \hto \CRan$.
\end{prop}

In turn, it suffices to prove the statement separately for each power of the curve $X^I$, the reason being that colimits in a limit category are computed naively.
The key is to estabilish the following general result.

\begin{lem}\label{lem:strong-approx}
Let $B$ be a base scheme of finite type, $\U$ a pro-unipotent group scheme over $B$ and $\U_1, \U_2 \subseteq \U$ two subgroups. Let $\Cmod{B} \in \DGCat_{/B} \simeq \Dmod(B) \mmod$ a category acted on by $\U$. If the multiplication map
$$
m: \U_1 \Times \U_2 \to \U
$$
is surjective, then there exists an equivalence $\Av_*^\U \simeq \Av_*^{\U_1} \circ \Av_*^{\U_2}$. (We are omitting to write the forgetful functors $\oblv^{\U}$ as they are fully faithful.)
\end{lem}

\begin{proof} 
Set $\U_1 \cap \U_2$ to be the subgroup $\U_1 \times_\U \U_2$ and denote by $\Times$ the fiber product $\times_B$. The assumption is equivalent to having an isomorphism
$$
\U_1 \overset{\U_1 \cap \U_2}{\Times} \U_2 \to \U.
$$
The functor $\Av_*^{\U_1} \circ \Av_*^{\U_2}$ is isomorphic to the composition
$$
\Cmod{B}
\xto{ - \Otimes k_{\U_1 \Times \U_2} }
\Cmod{B} \Otimes \Dmod^*(\U_1 \Times \U_2)
\xto{ m_*}
\Cmod{B} \Otimes \Dmod^*(\U)
\xto{\act_\U}
\Cmod{B}.
$$
Clearly, $m$ factors as $\hat m \circ q$, where $q: \U_1 \Times \U_2 \to \U_1 \overset{\U_1 \cap \U_2}{\Times} \U_2$ is a quotient map with contractible fibers. Hence, under the assumption, we have that $q_* (k_{\U_1 \Times \U_2}) \simeq k_\U$ and we are done.
\end{proof}

Let $m \geq n \geq 1$. In our situation, consider $\U:= \UA_{m,X^I}$ with subgroups $\U_1 := \UA_{n, X^I}$ and $\U_2 = \UA_{m, X^I} \cap U(\sK)_{X^I}$.

\begin{lem}
Fix $I \in \fset$. There exists $n \geq 1$ such that 
$$
\UA_{n, X^I} \ustimes{X^I}
\big( \UA_{m, X^I} \cap U(\sK)_{X^I} \big)
\to
 \UA_{m,X^I}
$$
is surjective on geometric points for any $m\geq n$.
\end{lem}

\begin{proof}
One easily reduces to $U = \GG_a$. Let $\x \in X^I$ be a geometric point and $D \subset X$ the divisor corresponding to it. Note that $\bigt{\UA_{m, X^I} \cap U(\sK)}(\bk) \simeq H^0(X_{\bk}, \O(m \cdot D))$ and that the statement is equivalent to the surjectivity of the map
$$
H^0(X_\bk, \O(m \cdot D)) \to H^0(X_\bk, \O(m \cdot D)/ \O(n \cdot D)).
$$
The obstruction lies in $H^1(X_\bk, \O(n \cdot D))$, which vanishes by Serre duality as soon as $n$ is large enough.
\end{proof}

\sssec{}

We can now complete the proofs of Proposition \ref{prop:Av_^*-continuous-Wh-ext-noT} and consequently Proposition \ref{prop:Av_^*-continuous-Wh-ext}.

\begin{proof}[Proof of Proposition \ref{prop:Av_^*-continuous-Wh-ext-noT}] 
We show that, for any $X^I$, the inverse system of functors $\Av_*^{\UA_{\ell, X^I}}$ eventually stabilizes on $\Cmod{X^I}^{U(\sK)}$. Indeed, combining the two lemmas above, we obtain that
$$
\Av_*^{\UA_{\ell, X^I}}
\simeq
\Av_*^{\UA_{k, X^I}}  \circ \Av_*^{\UA_{\ell, X^I} \cap\, U(\sK)_{X^I}} ,
$$
as soon as $\ell \geq k$. However, the latter is isomorphic to $\Av_*^{\UA_{k, X^I}}$ on $(\Cmod{X^I})^{U(\sK)}$.
\end{proof}

\ssec{The unital structures}

\begin{prop}
If $\CRan$ is naive-unital with a naive-unital action of $M(\sK) \ltimes U(\sA)$, then 
$$
\Av_*^{\UA, \ev}:
\big( (\CRan)^{U(\sK)} \Otimes \Dmod(\c) \big)^{M(\sK)}
\longto
\bigt{(\CRan) \Otimes \Dmod(\c)}^{M(\sK) \ltimes \UA, \ev}.
$$ is a morphism in $\R\mmod^\un$.
\end{prop}

\begin{proof}
First off, the twist by $\ev$ is irrelevant, whence it is enough to prove that $\Av_*^{U(\sA)}: (\CRan)^{U(\sK)} \to (\CRan)^{U(\sA)}$ is naive-unital whenever $\CRan \in \R\mmod^\un$ is acted on unitally by $U(\sA)$.
Writing 
$$
\CRan^{U(\sK)} 
\simeq 
\bigt{
\lim_{[n] \in \bold{\Delta}}
{}^{U(\sK)} \Dmod^!(U(\sA))
\Otimes
 \Dmod^!(U(\sA))^{\Otimes n}
 \Otimes
 \CRan
 },
$$
we see that it suffices to prove the assertion in the universal case, that is, for the regular representation $\CRan =  \Dmod^!(U(\sA))$. In other words, we need to prove that the right adjoint to 
$$
\Dmod(\Ran)
\xto{p^!}
\Dmod^!(U(\sA)/U(\sK)),
$$
which is continuous by strong approximation, is naive-unital. Denote by $p_?$ such functor.

By construction, $U(\sA)/U(\sK)$ belongs to $\RanCorr_{/\Ran}$, via the diagram 
\begin{gather}  
\xy
(50,0)*+{ \Ran \times \Ran}="10";
(90,0)*+{ \Ran;}="20";
(0,15)*+{ \Ran \times (U(\sA)/U(\sK)) }="01";
(50,15)*+{ U(\sO,\A)/(\Ran \times U(\sK))}="11";
(90,15)*+{ U(\sA)/U(\sK)}="21";
{\ar@{->}_{  \alpha   } "11";"01"};
{\ar@{->}^{ \add  } "10";"20"};
{\ar@{->}^{ \;\;\;\;\beta  } "11";"21"};
{\ar@{->}^{v} "11";"10"};
{\ar@{->}_{\id \times p} "01";"10"};
{\ar@{->}^{p} "21";"20"};
\endxy
\end{gather}	
this is just the quotient along the subgroup $U(\sK) \to U(\sA)$ in $\RanCorr_{/\Ran}$. We know that $p^!$ is a morphism in $\R\mmod^\un$: this means that there is a natural isomorphism $\beta_+ \circ \alpha^! \circ (id \times p)^! \simeq p^! \circ \add_*$, together with coherent compatibilities. We wish to prove that $p_?$ is also a morphism in $\R\mmod^\un$. To do this, we just need to show that the natural transformation 
$$
\add_* \circ  (\id \times p)_? 
\longto 
p_? \circ \beta_+ \circ \alpha^!
$$
is an isomorphism. This follows from a simple diagram chase together with the contractibility of the fibers of $v$ and $p$.
\end{proof}

Thus, if $\CRan$ is naive-unital, we define the functor
$$
\coeff_{\c}: 
(\Cind)^{(M \ltimes U)(\KK)} 
\longto
\big( (\CRan)^{U(\sK)} \Otimes \Dmod(\c) \big)^{M(\sK)}_\indep
\longto
\bigt{(\CRan) \Otimes \Dmod(\c)}^{M(\sK) \ltimes \UA, \ev}_\indep,
$$
to be the independent version of $\coeffRan_\c$.

\begin{defn}

Let $\CRan$ be naive-unital, acted on unitally by $T(\sK) \ltimes N(\sA)$. The above constructions specialize to the following: 
$$
\coeff_{G,\ext} := \coeff_{\ch_{G,\ext}}
:
(\Cind)^{B(\KK)} \longto \Wh_{G,\ext}[\Cind];
$$
$$
\coeff_{G,P} := \coeff_{\ch_{G,P}}
:
(\Cind)^{B(\KK)} \longto \Wh_{G,P}[\Cind].
$$ 
\end{defn}

\begin{rem}
The above coefficient functors have $(\Cind)^{B(\KK)}$ for source. The functor mentioned in (\ref{eqn:coeff_bunG}), which we denoted also $\coeff_{G,\ext}$ abusing notation, is the one obtained in the case $\CRan = \Dmod(\GrGRan)$ precomposing with the inclusion (as proven in \cite{Ba})
$$
\Dmod(\Bun_G) 
\hto
\Dmod(\Bun_G^{B-\gen})
\simeq 
\Dmod(\Bun_G^{1-\gen})^{B(\KK)}
\simeq
\Bigt{\Dmod(\GrGRan)^{B(\sK)}}_\indep.
$$
\end{rem}

The corollary below expresses the commutativity of diagram (\ref{triangle:geom}).

\begin{cor}
Let $\CRan$ be acted on by $T(\sK) \ltimes \NA$. For any $P \in \Par$, there is a natural isomorphism $\coeffRan_{G,P} \simeq \vrho_{G,\ext \to P} \circ \coeffRan_{G,\ext}$. A completely analogous statement holds in the presence of naive-unital structures.
\end{cor}

\begin{proof} 
Let $\ch_{G,P}$ be the $\Ran$ version of $\chdom_{G,P}$:
$$
\ch_{G,P}:
S 
\mapsto
 \left\{ 
\big( \x, \eta= \{\eta_i\}_{i \in \I_M} \big)
\left| 
\begin{array}{l}
 \mbox{$\x \in \Ran(S)$,} \\
\mbox{$\eta_i$ a fiberwise non-zero section of $\restr{\O_S \boxtimes \O_X}{U_\x}$ }
\end{array} \right.
\right\},
$$
and apply the commutativity of (\ref{diag_Av_*-f^!}) with $\ch_{G,P} \hto \ch_{G,\ext}$ playing the role of $\c \hto \c'$. 
\end{proof}

\begin{lem} \label{lem:Av_*-f^!}
In the situation of Section \ref{sssec:defn-coeff-U}, let $\c'$ be another relative indscheme mapping to $\ch_U$ and $f: \c \to \c'$ be a $M(\sK)$-equivariant morphism such that $\ev = \ev' \circ f$.
Then, lax-commutative diagram
\begin{gather}  \label{diag_Av_*-f^!}
\xy
(0,0)*+{ \bigt{ \CRan \Otimes \Dmod(\c) }^{M(\sK) \ltimes U(\sA), \ev}  }="00";
(55,0)*+{ \bigt{ \CRan \Otimes \Dmod(\c') }^{M(\sK) \ltimes U(\sA), \ev}, }="10";
(0,24)*+{ \bigt{ (\CRan)^{U(\sK)} \Otimes \Dmod(\c) }^{M(\sK)} }="01";
(55,24)*+{ \bigt{ (\CRan)^{U(\sK)} \Otimes \Dmod(\c') }^{M(\sK)}   }="11";
{\ar@{->}^{  \bf^! } "00";"10"};
{\ar@{->}^{   \bf^!  } "01";"11"};
{\ar@{->}^{ \Av_*^{U(\sA),\ev'}} "11";"10"};
{\ar@{->}^{  \Av_*^{U(\sA),\ev}   } "01";"00"};
\endxy
\end{gather}	
obtained from Lemma \ref{lem:f^!-oblv} by adjunction, is actually commutative. When everything in sight is naive-unital, we obtain a commutative diagram of independent subcategories.
\end{lem}

\begin{proof}
The twist by $\ev$ is irrelevant, so we assume that $\c$ is $\Ran$. It suffices to check the claim separately on each $X^I$.
Further, by the proof of Proposition \ref{prop:Av_^*-continuous-Wh-ext-noT}, we know that the inverse system of functors $\Av_*^{U(\sA)_{\ell, X^I}}$ stabilizes. Hence, it suffices to show that, for any $\ell \geq 1$, the functors $\bf^!$ intertwine the averaging functors $\Av_*^{U(\sA)_{\ell, X^I}}$. This is a straightforward computation using the fact that $\Av_*^{U(\sA)_{\ell, X^I}}$ is the functor of convolution with the constant sheaf on $U(\sA)_{\ell, X^I}$. 
\end{proof}

\sssec{} \label{sssec:coeff-GL2-conclusion}

We conclude this section by finishing the proof that $\coeffRan_{GL_2,\ext}$ is an equivalence. Since the coefficient functor is that case is essentially Fourier transform, the equivalence already holds at the level of categories over $\Ran$, with no \virg{independent} manipulations required.

\medskip

We need to verify that the equivalence (\ref{eqn:obvious2}) is given by the coefficient functor. For that, it suffices to observe that the equivalence of (\ref{eqn:coeff_V-abstract}) fits in the commutative triangle
\begin{gather}
\xy
(30,0)*+{ \bigt{\CC \otimes_{\Dmod(B)} \Dmod^*(\W^\perp)}^{\H \ltimes \W},}="basso";
(60,25)*+{ \bigt{\CC \otimes_{\Dmod(B)} \Dmod^*(\W^\perp)}^{\H \ltimes \V, \ev} }="dx";
(0,25)*+{ \CC^{\H \ltimes \W} }="sx";
{\ar@{->}^{ \simeq  } "sx";"dx"};
{\ar@{<-}^{ \mathsf{pullback}} "basso";"sx"};
{\ar@{<-}_{ \oblv } "basso";"dx"};
\endxy
\end{gather}	
where the functor $\oblv$ is fully faithful.

\sec{Fully faithfulness of the coefficient functor} \label{SEC:MAIN}

In this section, we prove the main result of this paper, Theorem \ref{thm:MAIN}.

\ssec{Statement and overview}

Let us fix some notation first. Unless stated otherwise, let $G := GL_n$ and $G' := GL_{n-1}$ for the entirety of this section. Standard parabolic subgroups of $GL_n$ correspond to partitions of $n$. We shall denote by $P_{m_1, \ldots, m_p} \subseteq GL_n$ the subgroup corresponding to the partition $n= m_1 + \cdots + m_p$. 

In particular, we set $Q:= P_{n-1, 1} \subset GL_n$ and $Q' := P_{n-2,1} \subseteq G'$. 
The Levi decomposition of $Q$ is $Q \simeq M \ltimes V$, where $M := G' \times \GG_m$ and  $V := \GG_a^{n-1}$.

\sssec{} \label{sssec:twists-xi}

Since it matters in the course of the proof, let us be precise about the $\Omega$-twist. Note that $\Vt(\sK)$ (resp., $\Vt(\sA)$) parametrizes rational sections (resp., meromorphic jets of sections) of the rank $(n-1)$ vector bundle 
$$
\Omega_X^{n-1} \oplus \Omega_X^{n-2} \cdots \oplus \Omega_X.  
$$
Hence, $\ch_V \simeq \GSectRan{\BV\E}$, the annihilator of $\Vt(\sK)$, is the relative indscheme of rational sections of 
$$
\E :=\Omega_X^{2-n} \oplus \Omega_X^{1-n} \cdots \oplus \O_X.
$$
The obvious inclusion $\O_X \hto \E$ yields a section $\xi$ of $\ch_V \to \Ran$.

\begin{thm} \label{thm:MAIN}
Let $G=GL_n$, or $G=PGL_n$, and $Q$ be the standard parabolic subgroup corresponding to the partition $n= (n-1)+1$. For any $\CRan \in \R\mmod^\un$ acted on unitally by $\Qt(\sA)$, the extended coefficient functor
$$
\coeff_{G,\ext}: (\Cind)^{\Qt(\KK)} 
\longto 
\Wh_{G,\ext}[\Cind]
$$
is \emph{fully faithful}.
\end{thm}

We will only run the proof in the $GL_n$ case, the case of $PGL_n$ is almost identical and left to the reader.

\sssec{}
\nc{\VtA}{\Vt(\sA)}

The statement of Theorem \ref{thm:MAIN} is trivial for $n=1$: in that case, the coefficient functor is the identity of $(\Cind)^{\GG_m(\sK)}$. The proof ultimately goes by induction on $n$. More precisely, we shall express $\coeff_{G,\ext}$ as a composition of four fully faithful functors: the first functor is an equivalence by Fourier transform, the second is fully faithful thanks to a geometric argument explained in Section \ref{ssec:blowup}, the third functor is an equivalence thanks to a standard result on transitive actions (Proposition \ref{prop:stabilizer}), the fourth functor is fully faithful by induction.

\medskip

We start from $\CRan$, viewed as a category acted on by $\Mt(\K) \ltimes \VtA$, and use Proposition \ref{prop:first-step} to obtain the equivalence
\begin{equation} \nonumber
\coeffRan_{\ch_V} =\Av_*^{\VtA, \ev} \circ (\omega_{\ch_V} \Otimes -): 
(\CRan)^{\Qt(\sK)} \xto{ \; \; \simeq \;\;}  \bigt{ \CRan \Otimes \Dmod (\ch_V) }^{\Mt(\sK) \ltimes \Vt(\sA), \ev}.
\end{equation}
Our first functor is the induced equivalence 
\begin{equation} \label{eqn:PROOF1}
\coeff_{\ch_V}: 
(\Cind)^{\Qt(\sK)} \xto{ \; \; \simeq \;\;}  \bigt{ \CRan \Otimes \Dmod (\ch_V) }^{\Mt(\sK) \ltimes \Vt(\sA), \ev}_\indep.
\end{equation}
on the independent subcategories.

\ssec{Blowing up} \label{ssec:blowup}

We now wish to compare $\ch_V$ with $\BB\ch_V$, the space of \emph{generic sections into the blow-up of $\BV\E$ along the zero section}. 

\sssec{} \label{sssec:blow-up-indschemes}

Denoting by $\BB\E$ such blow-up and by $\BP\E$ the projective bundle associated to $\E$, 
define
$$
\BP\ch_V :=
\GSectRan{\BP\E}
\hspace{.6cm}
\BB\ch_V :=
\GSectRan{\BB\E}.
$$
By construction, $\BP\ch_V$ and $\BB\ch_V$ are the functors
\begin{equation} \label{eqn:Pch_V}
\BP\ch_V:
S \mapsto 
\left\{ 
(\x, \L, \iota) \left| 
\begin{array}{l}
 \x \in \Ran(S), \\
\iota: \L \hto \restr{ (\O_S \boxtimes \E)}{U_\x} \mbox{ a line sub-bundle}
\end{array}
 \right.
\right\} \big/ \approx
\end{equation}
and
\begin{equation} \label{eqn:blow-up}
\BB\ch_V:
S \mapsto 
\left\{ 
(\x, \L, \iota, \sigma) \left| 
\begin{array}{l}
 \x \in \Ran(S), \\
\iota: \L \hto \restr{ (\O_S \boxtimes \E)}{U_\x} \mbox{ a line sub-bundle,} \\
\sigma: \O_{U_\x} \to \L \mbox{ a section}
\end{array}
 \right.
\right\}\big/ \approx,
\end{equation}
where in both cases $\approx$ identifies isomorphic pairs of line bundles compatible with the other pieces of data.
The tautological incidence correspondence
\begin{gather}  \label{eqn:blow-up-incidence}
\xy
(0,0)*+{ (\x, \L, \iota) }="00";
(30,0)*+{ (\x, \L, \iota, \sigma )}="10";
(60,0)*+{ (\x, \iota \circ \sigma) }="20";
(0,7)*+{ \BP \ch_V }="01";
(30,7)*+{ \BB\ch_V}="11";
(60,7)*+{ \ch_V}="21";
{\ar@{|->}^{ } "10";"00"};
{\ar@{->}_{ \lambda} "11";"01"};
{\ar@{|->}^{ } "10";"20"};
{\ar@{->}^{ \pi} "11";"21"};
\endxy
\end{gather}	
is a diagram in the category of spaces over $\Ran$ with $M(\sK)$-action. Hence, setting $\wt\ev = \ev \circ \pi$, we conclude (Lemma \ref{lem:f^!-oblv}) that $\pi^!$ induces a functor
$$
\bold\pi^!: \bigt{ \CRan \Otimes \Dmod (\ch_V) }^{\Mt(\sK) \ltimes \Vt(\sA), \ev}
\longto
\bigt{ \CRan \Otimes \Dmod (\BB\ch_V) }^{\Mt(\sK) \ltimes \Vt(\sA), \wt\ev}.
$$

\sssec{}

As usual, denote by $\chind$, $\BP\chind$, $\BB\chind$ the independent versions of the above prestacks.
The following proposition, whose proof is postponed to Section \ref{ssec:crucial}, is crucial.

\begin{prop} \label{prop:crucial}
The pullback $\pi^!: \Dmod(\chind_V) \to \Dmod(\BB\chind_V)$ is \emph{fully faithful} and admits a left adjoint $\pi_!$.
\end{prop} 

\begin{cor}
The functor
\begin{equation} \label{eqn:PROOF2}
\bigt{ \CRan \Otimes \Dmod (\ch_V) }_\indep^{\Mt(\sK) \ltimes \Vt(\sA), \ev}
\longto
\bigt{ \CRan \Otimes \Dmod (\BB\ch_V) }_\indep^{\Mt(\sK) \ltimes \Vt(\sA), \wt\ev}
\end{equation}
induced by $\bold\pi^!$ is fully faithful.
\end{cor}

This is our second functor.

\begin{proof}[Proof of the corollary]
As an immediate consequence of Proposition \ref{prop:crucial}, the functor
\begin{equation} \label{eqn:pullback-tensored}
\id \otimes \pi^!:  \Cind \otimes \Dmod (\chind_V) 
\longto
 \Cind \otimes \Dmod (\BB\chind_V),
\end{equation}
is fully faithful.
Then, Lemma \ref{lem:useful-base-change-limits} guarantees that the functor in question is a base-change of 
$$
\bold\pi^!: \bigt{ \Cind \otimes \Dmod (\chind_V) }^{\Mt(\KK)}
\longto
\bigt{ \Cind \otimes \Dmod (\BB\chind_V) }^{\Mt(\KK)}.
$$
In turn, thanks to the contractibility of $\Mt(\KK)$, the latter is a base-change of (\ref{eqn:pullback-tensored}), and we are done.
\end{proof}

\ssec{Transitivity}

The reason for bringing $\BB\chind_V$ and $\BP\chind_V$ into play is that the their quotients by $\Mt(\KK)$ are really simple.

\begin{lem} \label{lem:Pch_V = quotient}
Let $[\xi]$ be the image of $\xi$ along the natural morphism $(\chind_V - O) \to \BP\chind_V$. The inclusion $i_{[\xi]} : \pt \hto \BP\chind_V$ yields an equivalence
$$
i_{[\xi]}: \pt/(\Qt' \times \GG_m)(\KK) \xto{\; \; \simeq \;\;}  \BP\chind_V / \Mt(\KK)
$$
(the quotients are taken in the {\'e}tale topology). In other words, the $\Mt(\KK)$-action on $\BP\chind_V$ is transitive and the stabilizer of $[\xi]$ is $(\Qt' \times \GG_m)(\KK)$.
\end{lem}

\begin{proof}
We show that the \virg{action on $[\xi]$} map $\Mt(\KK) \to \BP \chind_V$ is surjective on geometric points. Let $(U, \L_U \hto \restr \E U)$ be a presentation of some $f \in \BP\chind_V(\Spec(k))$. Shrinking $U$, we may assume that both $\L_U$ and $\restr{\E}U$ are trivial, whence $f$ consists of $n-1$ functions on $U$. Shrinking $U$ further, we may suppose that at least one of those functions is \emph{never} vanishing. This guarantees the existence of an element of $GL_{n-1}(\O_U)$ whose last column in $f$.
\end{proof}

The fiber of $\lambda: \BB\ch_V \to \BP\ch_V$ over $[\xi]$ identifies with $\ch_{\GG_a}$, thought of as the space of characters of the last entry of 
$$
\Vt(\sA) \simeq \prod_{i=n-2}^{0}(\GG_a)_{\Omega^{i}}(\sA),
$$
trivial on $\GG_a(K)$. For this reason, we denote this fiber by $\ch_{n-1}$. 
We are now ready to present our third functor.

\begin{cor}
Pullback along $i_{[\xi]}$ induces an equivalence
\begin{equation} \label{eqn:PROOF3}
(\bi_{[\xi]})^!:
\bigt{ \CRan \Otimes \Dmod (\BB\ch_V)}_\indep^{\Mt(\sK) \ltimes \Vt(\sA), \wt\ev}
\xto{ \; \; \simeq \;\;}
\bigt{ \CRan \Otimes \Dmod (\ch_{n-1}) }_\indep^{(\Qt' \times \GG_m)(\sK) \ltimes \Vt(\sA), \ev}.
\end{equation}
\end{cor}

\begin{proof}
Lemma \ref{lem:Pch_V = quotient} implies that the inclusion $\ch_{n-1} \hto \BB\ch_V$ yields an isomoprhism $\chind_{n-1} / (\Qt' \times \GG_m)(\KK) \simeq \BB\chind_V/\Mt(\KK)$. By Proposition \ref{prop:stabilizer}, we deduce that
$$
\bi_{[\xi]}^!: 
\bigt{ \Cind \otimes \Dmod (\BB\chind_V) }^{\Mt(\KK)}
\longto
\bigt{ \Cind \otimes \Dmod (\chind_{n-1}) }^{(\Qt \times \GG_m)(\KK)}
$$
is an equivalence. Next, base-change with respect to 
$$
- \ustimes
{ \CRan \Otimes \Dmod (\ch_{n-1}) }
\bigt{ \CRan \Otimes \Dmod (\ch_{n-1}) }^{\VtA,\ev}
$$
and use Lemma \ref{lem:useful-base-change-limits}.
\end{proof}

\ssec{Induction}

After rearranging using Lemma \ref{lem:Av_*-f^!}, we see that the composition of the arrows (\ref{eqn:PROOF1}), (\ref{eqn:PROOF2}), (\ref{eqn:PROOF3}) is the fully faithful functor
$$
\coeff_{\ch_{n-1}} 
\simeq
\Av_*^{\VtA, \ev} \circ (\omega_{\ch_{n-1}} \Otimes -):
(\Cind)^{\Qt(\KK)} 
\longto
\left(
\bigt{ 
\CRan \Otimes \Dmod (\ch_{n-1}) 
}
^{\GG_m(\sK) \ltimes \Vt(\sA), \ev}_\indep
\right)
^{\Qt'(\KK)}.
$$
Next, we apply the induction hypothesis to the naive-unital category $\bigt{ 
\CRan \Otimes \Dmod (\ch_{n-1}) 
}
^{\GG_m(\sK) \ltimes \Vt(\sA), \ev}$,
deducing that 
$$
\coeff_{G', \ext}:  
\left(
\bigt{ 
\CRan \Otimes \Dmod (\ch_{n-1}) 
}
^{\GG_m(\sK) \ltimes \Vt(\sA), \ev}_\indep
\right)
^{\Qt'(\KK)}
\longto
\Wh_{GL_{n-1},\ext} \Big[ \bigt{ 
\CRan \Otimes \Dmod (\ch_{n-1}) 
}_\indep
^{\GG_m(\sK) \ltimes \Vt(\sA), \ev} \Big]
$$
is fully faithful. This is our fourth functor.
Thanks to the obvious isomorphism
$$
\Tt(\K) \ltimes \NtA 
\simeq
\big( (\Tt'(\K) \ltimes \Nt'(\sA) \big) \ltimes \big( \GG_m(\K)  \ltimes \Vt(\sA) \big),
$$
one easily checks that the category on the RHS is equivalent to $\Wh_{GL_n, \ext}[\Cind]$. Moreover, using Lemma \ref{lem:Av_*-f^!} once more, the resulting composition $(\Cind)^{\Qt(\KK)} \hto \Wh_{GL_n,\ext}[\Cind]$ is the extended coefficient functor.

\medskip

Modulo Proposition \ref{prop:crucial}, this concludes the proof of Theorem \ref{thm:MAIN}.

\ssec{The proof of Proposition \ref{prop:crucial}} \label{ssec:crucial}

The argument rests on two ingredients: first off, $\pi$ enjoys a kind of properness property thanks to which $\pi_!$ is defined; second off, the exceptional fiber of $\pi: \BB\ch_V \to \ch_V$ is contractible.

\medskip

The categorical input needed, whose proof is purely formal, is the lemma below. Before stating it, recall the notion of \virg{short exact sequence} of categories: a diagram
\begin{gather}
\xy
(0,0)*+{ \Ccat_1 }="01";
(20,0)*+{ \Ccat}="11";
(40,0)*+{ \Ccat_2}="21";
{\ar@<-.5ex>@{->}_{ j_* } "01";"11"};
{\ar@<.4ex>@{<-}^{ j^! } "01";"11"};
{\ar@<.4ex>@{<-}^{ i_* } "11";"21"};
{\ar@<-.5ex>@{->}_{ i^! } "11";"21"};
\endxy
\end{gather}	
where the two pairs of arrows are adjoint, the morphisms $i_*: \Ccat_1 \to \Ccat$ and $j_*: \Ccat_2 \to \Ccat$ are fully faithful and the kernel of $i^!$ is equivalent to the essential image of $j_*$.\footnote{By adjunction, the kernel of $j^!$ is the essential image of $i_*$.}

\begin{lem} \label{lem:exact-seq-cat}
Assume given a diagram
\begin{gather}  \label{diag:exact-seq-cats}
\xy
(0,0)*+{ \E_1 }="00";
(30,0)*+{ \E}="10";
(60,0)*+{ \E_2}="20";
(0,15)*+{ \Ccat_1 }="01";
(30,15)*+{ \Ccat}="11";
(60,15)*+{ \Ccat_2}="21";
{\ar@<-.5ex>@{->}_{ j_* } "00";"10"};
{\ar@<.4ex>@{<-}^{ j^! } "00";"10"};
{\ar@<-.5ex>@{->}_{ j_* } "01";"11"};
{\ar@<.4ex>@{<-}^{ j^! } "01";"11"};
{\ar@<.4ex>@{<-}^{ i_* } "10";"20"};
{\ar@<.4ex>@{<-}^{ i_* } "11";"21"};
{\ar@<-.5ex>@{->}_{ i^! } "10";"20"};
{\ar@<-.5ex>@{->}_{ i^! } "11";"21"};
{\ar@{->}^{ \psi} "11";"10"};
{\ar@{->}^{ \psi_1 } "01";"00"};
{\ar@{->}^{ \psi_2 } "21";"20"};
\endxy
\end{gather}	
where the two rows are \emph{short exact sequences of categories} and the four squares are commutative.
Suppose that 
\begin{itemize}
\item[(a)]
$\psi_1$ and $\psi_2$ are fully faithful;
\item[(b)]
$\psi_1$ admits a left adjoint, denoted $(\psi_1)^L$;
\item[(c)]
$\psi$ admits a left adjoint, denoted $(\psi)^L$;
\item[(d)]
the natural transformation $\psi^L \circ j_* \to j_* \circ (\psi_1)^L$ is an equivalence.
\end{itemize}
Then $\psi$ is fully faithful.
\end{lem}

\begin{proof}
It suffices to show that the counit natural transformation $(\psi)^L \circ \psi \to \id$ is an equivalence when evaluated separately on the essential images of $j_*$ and $i_*$. Thanks to (d), this is obvious for the image of $j_*$. As for $i_*$, we need to show that the natural map 
$$
 \Hom_\Ccat(i_*(c_2), d) 
 \to 
\Hom_{\E} (\psi(i_* (c_1)), \psi(d))
$$
is an isomorphism for any $d \in \Ccat$. Again, it is enough to check this claim for $d = j_*(c')$ and $d= i_*(c'')$ separately. This is obvious for the latter case (fully faithfulness of $\psi_2$). As for the former, one easily sees by adjunction that both terms are zero.
\end{proof}

\sssec{}

We wish to use the above paradigm to prove that $\pi^!: \Dmod(\chind_V) \to \Dmod(\BB\chind_V)$ is fully faithful. More generally, we shall prove an analogous statement for any vector bundle $E$ over $X$.

\medskip

Let $E \to X$ be a vector bundle. Denote by $\zeta: X \hto E$ the zero section and by  $E^\circ$ its open complement.
Applying the functor $\GSectdom{-}$ to $\zeta$, we obtain the closed embedding $i: O \simeq \pt \hto \GSectdom{E}$, see Corollary \ref{cor:QSect-closed-embedding}.
Its open complement is $\GSectdom{E}^\circ \simeq \GSectdom{E^\circ}$.

Similarly, we have the closed embedding $\wt i:\GSectdom{\BPE} \hto \GSectdom{\BBE}$, with open complement $\GSectdom{\BBE^\circ}$, where $\BBE^\circ := \BBE \times_{E} E^\circ$.
Hence, with self-explanatory notation, the two rows in the following diagram are exact sequences of categories:
\begin{gather}  \label{diag:blowup-stratified}
\xy
(0,0)*+{ \Dmod (\GSectdom{\BBE^\circ})  }="00";
(45,0)*+{ \Dmod (\GSectdom{\BBE}) }="10";
(90,0)*+{ \Dmod (\GSectdom{\BPE}).}="20";
(0,20)*+{ \Dmod (\GSectdom{E^\circ})  }="01";
(45,20)*+{ \Dmod(\GSectdom{E})}="11";
(90,20)*+{ \Dmod(O)}="21";
{\ar@<-.5ex>@{->}_{ \wt j_* } "00";"10"};
{\ar@<.4ex>@{<-}^{ \wt j^! } "00";"10"};
{\ar@<-.5ex>@{->}_{  j_* } "01";"11"};
{\ar@<.4ex>@{<-}^{  j^! } "01";"11"};
{\ar@<.4ex>@{<-}^{ \wt i_* } "10";"20"};
{\ar@<.4ex>@{<-}^{ i_* } "11";"21"};
{\ar@<-.5ex>@{->}_{ \wt i^! } "10";"20"};
{\ar@<-.5ex>@{->}_{ i^! } "11";"21"};
{\ar@{->}^{ \pi^!} "11";"10"};
{\ar@{->}^{ (\pi^\circ)^! } "01";"00"};
{\ar@{->}^{ p^! } "21";"20"};
\endxy
\end{gather}	
\sssec{}

We will verify that the hypotheses of Lemma \ref{lem:exact-seq-cat} are met in this case.
Obviously, the four squares of (\ref{diag:blowup-stratified}) commute, either by the functoriality of the $!$-pullback or by base-change. Also, $(\pi^\circ)^!$ is an equivalence (just because $\pi^\circ$ is an isomorphism), whence $(\pi^\circ)_!$ is well-defined.

\begin{lem}
The map $p^!: \Vect = \Dmod(O) \to \Dmod(\GSectdom{\BPE})$ is fully faithful.
\end{lem}

\begin{proof}
It is enough to prove the assertion locally on $X$, hence we may assume $E$ is trivial with fiber $W$.
As $\GMapsdom{\BP{W}}$ is the quotient of two contractible prestacks (Lemma \ref{lem:Pch_V = quotient} and Theorem \ref{thm:contr:Barlev}), the claim amounts to the compatibility of homology with colimits.
\end{proof}

So far, we have assured that hypotheses (a) and (b) of Lemma \ref{lem:exact-seq-cat} hold. To address (c) and (d), we must make a digression (Section \ref{sssec:digressione-inizio}) about maps of prestacks having well-defined $!$-pushforward on $\fD$-modules.

\sssec{} \label{sssec:digressione-inizio}

Let us say that a map $F: \X \to \Y$ in $\PreStk^{\fty}$ is \emph{weakly proper} if $F^!: \Dmod(\Y) \to \Dmod(\X)$ admits a left adjoint.
Obviously, a proper map (resp., a composition of weakly proper maps) is weakly proper. Weak properness is preserved by products: if $F: \X \to \Y$ is weak proper, then so is $F \times \id_\Z : \X \times \Z \to \Y \times \Z$. Indeed, under the equivalence $\Dmod(- \times \Z) \simeq \Dmod(-) \otimes \Dmod(\Z)$, the left adjoint of $(F \times \id_\Z)^!$ is just $F_! \otimes \id_{\Dmod(\Z)}$.
However, weak properness is \emph{not} preserved by fiber products in general.

\begin{lem} \label{lem:weak-proper-colim}
Let $\I$ be an arbitrary index category and 
$$
\Phi: \I \times \Delta^1 \to \PreStk^{\fty}
$$
be a diagram of schemes with all arrows weakly proper. Denote by $X_i$ (resp. $Y_i$) the image of $(i,0)$ (resp. $(i,1)$) under $\Phi$. 
The resulting map $F: \colim_i X_i \to \colim_i Y_i$ is weakly proper.
\end{lem}

\begin{proof}
For any $i,j \in \I$, denote by $F_i: X_i \to Y_i$, $p_{i \to j}: X_i \to X_j$ and $q_{i \to j}: Y_i \to Y_j$ the structure arrows. Then, $F^!$ arises as
$$
F^! \simeq \lim_{i \in \I} (F_i)^!: 
\lim_{i, q^!} \Dmod(Y_i) \longto \lim_{i, p^!} \Dmod(X_i).
$$
The weak properness of each $p_{i\to j}$ and $q_{i \to j}$ allows to express the two limit categories as colimit categories along the $!$-pushforward functors. Under these equivalences, it is immediate to verify that 
$$
\uscolim{i \in \I} (F_i)_!:
\uscolim{i, p_!}\, \Dmod(X_i) \longto \uscolim{i, p_!}\, \Dmod(Y_i),
$$ 
well-defined by the assumption, is the sought-after left adjoint.
\end{proof}

\begin{rem}
In the situation of the above lemma, assume that each arrow is proper (not just weakly proper). Then $F_!$ is computed as $\colim_i (F_i)_*$. 
\end{rem}

\begin{cor} \label{cor:weakly-proper-q}
The maps $\q: \QSect{\BPE} \to \GSectdom{\BPE}$ and $\GSectdom{\BPE} \to \pt$ are both weakly proper. 
\end{cor}

\begin{proof}
Let $Q_\bullet: \bold\Delta^\op \to \Sch$ be the Cech simplicial scheme generated by $\q$, whose colimit (in the Zariski topology) is $\GSectdom{\BPE}$, see Section \ref{ssec:QSECT}. Apply the above lemma to the canonical diagram $Q_{\bullet + 1 } \to Q_\bullet$, where $Q_{\bullet +1}$ is the split simplicial scheme that resolves $Q_0 = \QSect{\BPE}$. This yields the first assertion.

To prove the second one, use the resolution $Q_\bullet$ again to argue that $\GSectdom{\BPE}$ can be expressed as a colimit of proper (in fact, projective) schemes.
\end{proof}

\sssec{}

Resuming the proof of Proposition \ref{prop:crucial}, we now show that the map $F: \GSectdom{\BBE} \to \GSectdom{\BVE}$, induced by $\BBE \to \BVE$, is weakly proper. This would verify condition (c) of Lemma \ref{lem:exact-seq-cat}. 

More generally, we prove:

\begin{prop} \label{prop:general-weakly-proper}
Let $Y \to X$ be a quasi-projective morphism and $f: W \to Y$ a projective morphism. The map $F: \GSectdom{W}\to \GSectdom{Y}$ induced by $f$ is weakly proper.
\end{prop}

\begin{proof}
Write $W \to Y$ as a composition $W \hto \BP^N_Y \to Y$, with $W \hto \BP^N_Y$ closed.
Accordingly, $F$ factors as
$$
\GSectdom{W} \to \GSectdom{\BP^N_Y} \to \GSectdom{Y}.
$$
The first arrow is a closed embedding: in fact, its base change along the Zariski-surjection $\q:\QSect{\BP^N_Y} \twoheadrightarrow \GSectdom{\BP^N_Y}$ is the closed embedding $\QSect{W} \hto \QSect{ \BP^N_Y}$. 
Next, we show that the second arrow is weakly proper. Since $\GSectdom{-}: \Sch^{\fty}_{/X} \to \PreStk$ preserves finite limits, the arrow in question is the projection
$$
\GSectdom{\BP^N_Y} \simeq \GSectdom{\BP^N_X} \times \GSectdom{Y} \to \GSectdom{Y},
$$
whence it suffices to invoke the weak properness of $\GSectdom{\BP^N_X} \to \pt$ (Corollary \ref{cor:weakly-proper-q}).
\end{proof}

\sssec{}

It remains to check condition (d) of Lemma \ref{lem:exact-seq-cat}. In the present case, such condition is an instance of the following more general situation.
Let
\begin{gather} \nonumber
\xy
(00,0)*+{ U }="00";
(20,0)*+{ Y }="10";
(0,12)*+{ V }="01";
(20,12)*+{ Z }="11";
{\ar@{->}^{ g } "00";"10"};
{\ar@{->}^{ g' } "01";"11"};
{\ar@{->}^{ f  } "11";"10"};
{\ar@{->}^{ f'   } "01";"00"};
\endxy
\end{gather}	
be a Cartesian diagram of schemes over $X$, with $Y \to X$ quasi-projective, $f$ and $f'$ projective, $g$ and $g'$ open embeddings. Applying $\GSectdom{-}$, we obtain a Cartesian diagram 
\begin{gather} \label{diag:last}
\xy
(00,0)*+{ \GSectdom{U} }="00";
(35,0)*+{ \GSectdom{Y}, }="10";
(0,12)*+{ \GSectdom{V} }="01";
(35,12)*+{ \GSectdom{Z} }="11";
{\ar@{->}^{ j } "00";"10"};
{\ar@{->}^{ j' } "01";"11"};
{\ar@{->}^{ F  } "11";"10"};
{\ar@{->}^{ F'   } "01";"00"};
\endxy
\end{gather}	
where the horizontal maps are open embeddings and the vertical ones are weakly proper, by the proposition above.

\begin{prop}
The natural transformation $F'_! \circ j_* \to j_* \circ F_!$ is an equivalence.
\end{prop}

\begin{proof}
As in the proof of Proposition \ref{prop:general-weakly-proper}, it suffices to verify the assertion separately for $f: \BP^n_Y \to Y$ being the structure map and for $f: Z \hto Y$ being a closed embedding. 
In the latter case, we have already observed that $F$ and $F'$ are closed embeddings, whence the claim reduces to the functoriality of $*$-pushforwards.
In the former situation, the diagram in question splits into a product: it is equivalent to one of the form
\begin{gather} \nonumber
\xy
(00,0)*+{ \CA \otimes \D }="00";
(30,0)*+{ \B \otimes \D, }="10";
(0,17)*+{ \CA \otimes \Ccat }="01";
(30,17)*+{ \B \otimes \Ccat }="11";
{\ar@{->}^{  \gamma \otimes \id_{\D} } "00";"10"};
{\ar@{->}^{ \gamma \otimes \id_{\Ccat} } "01";"11"};
{\ar@{->}^{ \id_\B \otimes \delta   } "11";"10"};
{\ar@{->}^{ \id_{\CA} \otimes \delta   } "01";"00"};
\endxy
\end{gather}	
which is evidently commutative.
\end{proof}


\end{document}